\documentclass[12pt, reqno]{amsart}
\usepackage{pdfsync}
\usepackage{amssymb, amsthm, amsmath, amsfonts}
\usepackage{array, epsfig}
\usepackage{bbm}
\usepackage{hyperref}
\usepackage[numbers,square]{natbib}

  \evensidemargin
\oddsidemargin
\newcommand{\bD}{\mathbb{D}}

\DeclareMathOperator*{\wlim}{w\textsuperscript{*}-}
\newcommand{\Res}{\mathop{\mathrm{Res}}\nolimits}

\newcommand{\E}{\mathbb E}

\newcommand{\R}{\mathbb{R}}
\newcommand{\N}{\mathbb{N}}

\newcommand{\C}{\mathbb{C}}

\renewcommand{\Im}{\operatorname{Im}}

\newcommand{\supp}{\mathop{\mathrm{supp}}\nolimits}

\newcommand{\eps}{\varepsilon}

\newcommand{\toweak}{\overset{w}{\underset{n\to\infty}\longrightarrow}}
\newcommand{\toweaknon}{\overset{w}{\longrightarrow}}
\newcommand{\toprobab}{\overset{P}{\underset{n\to\infty}\longrightarrow}}

\newcommand{\ton}{\overset{}{\underset{n\to\infty}\longrightarrow}}

\newcommand{\ind}{\mathbbm{1}}

\newcommand{\dd}{{\rm d}}
\newcommand{\eee}{{\rm e}}
\newcommand{\ii}{{\rm i}}

\theoremstyle{plain}
\newtheorem{theorem}{Theorem}[section]

\theoremstyle{definition}

\newtheorem{example}[theorem]{Example}

\newtheorem{conjecture}[theorem]{Conjecture}

\theoremstyle{remark}
\newtheorem{remark}[theorem]{Remark}

\begin{document}

\author{Jeremy Hoskins and Zakhar Kabluchko}

\address{Jeremy Hoskins:
University of Chicago,
Department of Statistics.}
\email{jeremyhoskins@uchicago.edu}

\address{Zakhar Kabluchko: Institut f\"ur Mathematische Stochastik,
Westf\"alische Wilhelms-Universit\"at M\"unster,
Orl\'eans-Ring 10,
48149 M\"unster, Germany.}
\email{zakhar.kabluchko@uni-muenster.de}

\title[Zeroes under repeated differentiation]{Dynamics of zeroes under repeated differentiation}



\begin{abstract}
Consider a random polynomial $P_n$ of degree $n$ whose roots are independent random variables sampled according to some probability distribution $\mu_0$ on the complex plane $\mathbb C$. It is natural to conjecture that, for a fixed $t\in [0,1)$ and as $n\to\infty$, the zeroes of the $[tn]$-th derivative of $P_n$ are distributed according to some measure $\mu_t$ on $\mathbb C$. Assuming either that $\mu_0$ is concentrated on the real line or that it is rotationally invariant, Steinerberger~[Proc.\ AMS, 2019] and O'Rourke and Steinerberger~[arXiv:1910.12161] derived nonlocal transport equations for the density of roots. We introduce a different method to treat such problems.  In the rotationally invariant case, we obtain a closed formula for $\psi(x,t)$, the asymptotic density of the radial parts of the roots of the $[tn]$-th derivative of $P_n$.  Although its derivation is non-rigorous, we provide numerical evidence for its correctness and prove that it solves the PDE of O'Rourke and Steinerberger. Moreover, we present several examples in which the solution is fully explicit (including the special case in which the initial condition $\psi(x,0)$ is an arbitrary convex combination of delta functions) and analyze some properties of the solutions such as the behavior of void annuli and circles of zeroes.
As an additional support for the correctness of the method, we show that a similar method, applied to the case when $\mu_0$ is concentrated on the real line, gives a correct result which is known to have an interpretation in terms of free probability.
\end{abstract}

\maketitle

\section{Introduction}

\subsection{Statement of the problem}
Take some probability distribution $\mu_0$ on the complex plane $\C$ and consider independent complex-valued random variables $Z_1,Z_2,\ldots$ distributed according to $\mu_0$.  Let $P_n$ be a monic random polynomial of degree $n$ whose zeroes are $Z_1,\ldots,Z_n$, that is
$$
P_n(z) := \prod_{k=1}^n (z-Z_k), \qquad z\in \C.
$$
The critical points of the  polynomial $P_n$ are defined as the zeroes of its derivative $P_n'$. It has been conjectured by Pemantle and Rivin~\cite{pemantle_rivin} and proved by one of the authors in~\cite{kabluchko15} that the critical points of $P_n$ have the same asymptotic distribution as the roots. More precisely, we have
$$
\frac{1}{n-1} \sum_{z\in \C: P_n'(z) =0} \delta_z \toprobab \mu_0
$$
in probability, where $\delta_z$ denotes the unit mass at $z\in \C$, and both sides are viewed as random elements with values in the space $\mathcal M(\C)$ of finite measures on $\C$ endowed with the topology of weak convergence. Although this is  not essential in most cases, let us agree that the zeroes are always counted with multiplicities.  For further results on critical points of random polynomials  we refer to~\cite{subramanian,subramanian_phd,orourke_matrices,hu_chang,byun,Tulasi16,tulasi_phd}. One particularly interesting phenomenon is the existence of  ``pairing'' between the roots and the critical points established in various forms in~\cite{hanin_riemann,hanin_gauss,hanin_poly,orourke_williams_pairing,orourke_williams_local_pairing,kabluchko_seidel,steinerberger}. For example, if $\mu_0$ has a continuous Lebesgue density, then it is known that for each zero $Z_k$ of $P_n$ with high probability there exists a zero of $P_n'$ having a distance of order $\text{1}/n$ to $Z_k$, whereas the distance between neighboring zeroes of $P_n$ is of much larger order $\text{1}/\sqrt n$; see, e.g.,~\cite[Theorem~2.1]{kabluchko_seidel}.  Thus, under a single differentiation, the zeroes of $P_n$ move  by distances of order $\text{1}/n$. For this reason, it is natural to conjecture that if the differentiation operation is repeated $[tn]$ times, where $t$ is viewed as the ``time'' parameter ranging between  $0$ and $1$, some  natural and non-trivial macroscopic dynamics of roots should emerge.

We are therefore interested in the asymptotic distribution, as $n\to\infty$, of the zeroes of the  $[tn]$-th derivative of $P_n$, denoted by $P_n^{([tn])}$. Here, $t\in [0,1)$ stays fixed, and $[x]$ denotes the integer part of $x$. This question has been raised and studied in the papers of Steinerberger~\cite{steinerberger_real}, O'Rourke and Steinerberger~\cite{orourke_steinerberger_nonlocal} and Feng and Yao~\cite{feng_yao}; see also~\cite{hoskins_steinerberger,steinerberger_conservation,steinerberger_free}. Assigning a weight $1/n$ to each zero of the $[tn]$-th derivative, one can construct a sub-probability measure on $\C$ denoted by
$$
\mu_t^{(n)} := \frac 1n \sum_{z\in \C: P_n^{([tn])}(z) = 0} \delta_z.
$$
Following~\cite{steinerberger_real} and~\cite{orourke_steinerberger_nonlocal}, it is then natural to conjecture that for every fixed $t\in [0,1)$, this measure converges to certain deterministic limit measure denoted by $\mu_t$, as $n\to\infty$. More precisely, we should have
\begin{equation}\label{eq:mu_t_n_conv_mu_t_conj}
\mu_t^{(n)} \toprobab \mu_t
\end{equation}
in probability, where both sides are viewed as random elements with values in $\mathcal M(\C)$.
Since the total number of roots of the $[tn]$-th derivative is $n- [tn]$, the total mass of $\mu_t$ should be $\mu_t(\C) = 1-t$. For instance, Figure~\ref{pic:three_circles_zeroes_plane} shows the roots of the repeated derivatives of a polynomial of degree $n=30000$ in the special case when $\mu_0$ is a convex combination of $3$ uniform distributions on concentric circles.

The main problem studied in the present paper is how to determine $\mu_t$ given the initial distribution $\mu_0$.
We shall be  interested in the following two special cases:
\begin{itemize}
\item (Rotationally invariant) \emph{complex zeroes:} The initial distribution $\mu_0$ is invariant under rotations of the complex plane around the origin.
\item \emph{Real zeroes:} $\mu_0$ is concentrated on the real line.
\end{itemize}
Let us discuss these cases in more detail.

\begin{figure}[!tbp]
\includegraphics[width=0.32\textwidth]{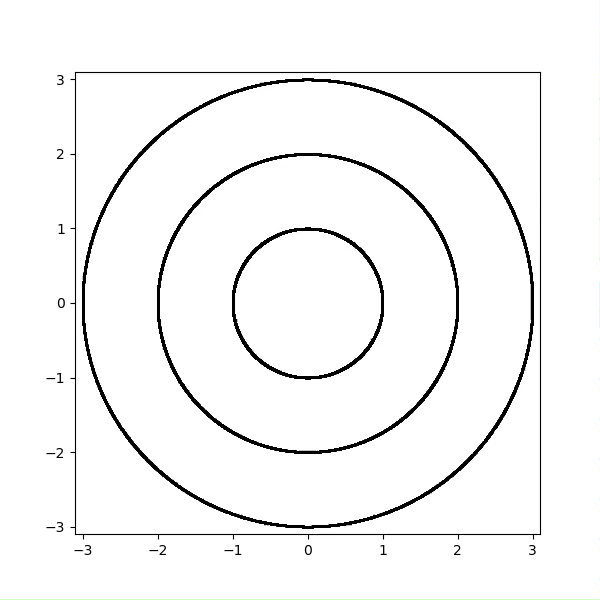}
\includegraphics[width=0.32\textwidth]{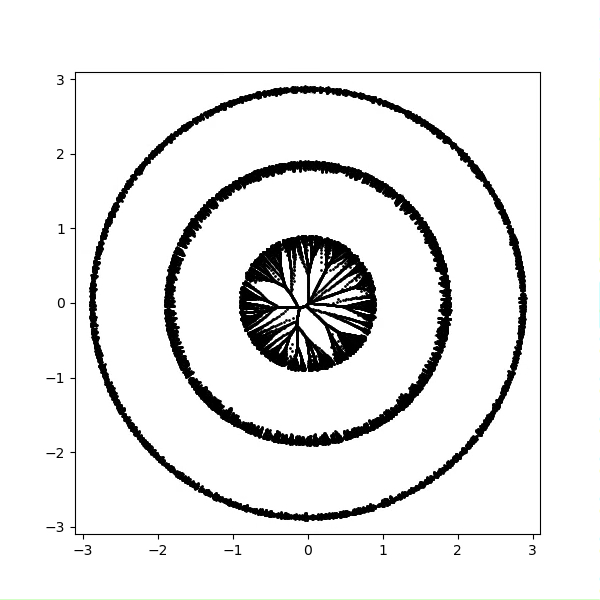}
\includegraphics[width=0.32\textwidth]{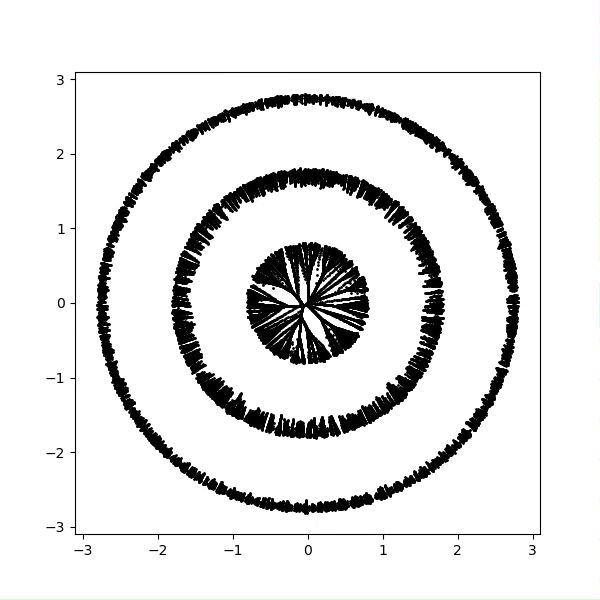}
\includegraphics[width=0.32\textwidth]{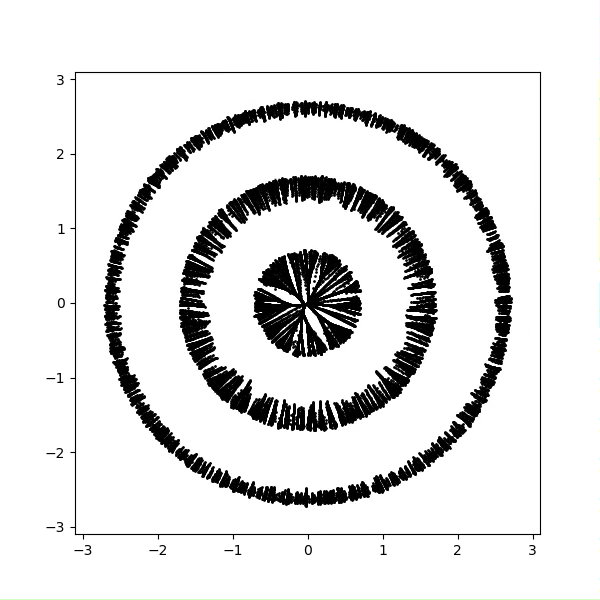}
\includegraphics[width=0.32\textwidth]{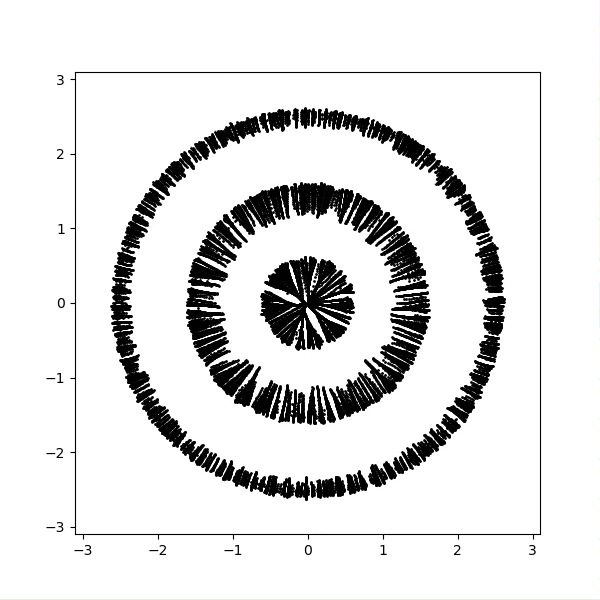}
\includegraphics[width=0.32\textwidth]{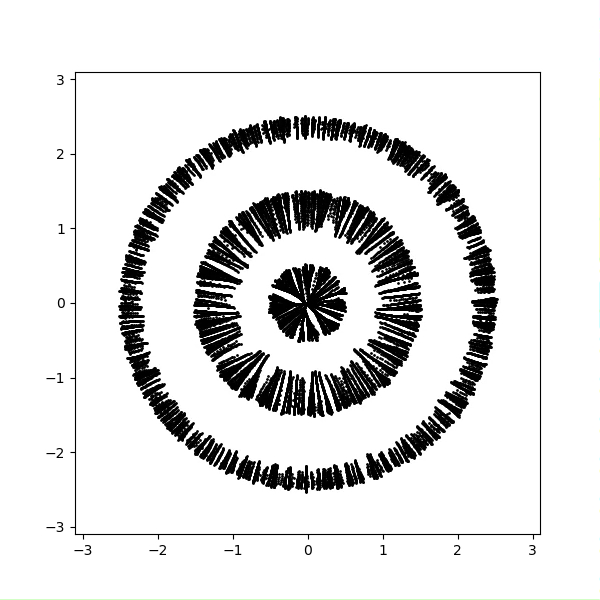}
\includegraphics[width=0.32\textwidth]{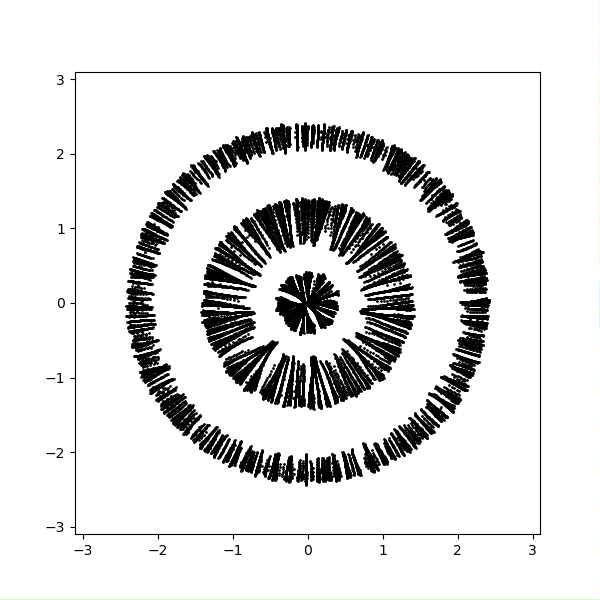}
\includegraphics[width=0.32\textwidth]{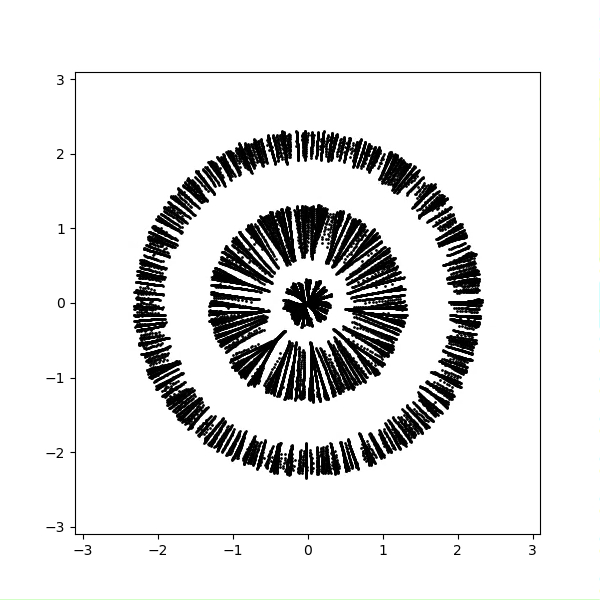}
\includegraphics[width=0.32\textwidth]{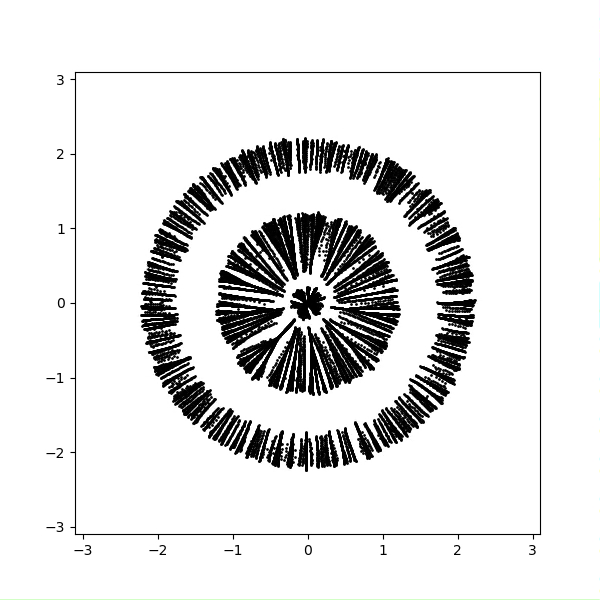}
\includegraphics[width=0.32\textwidth]{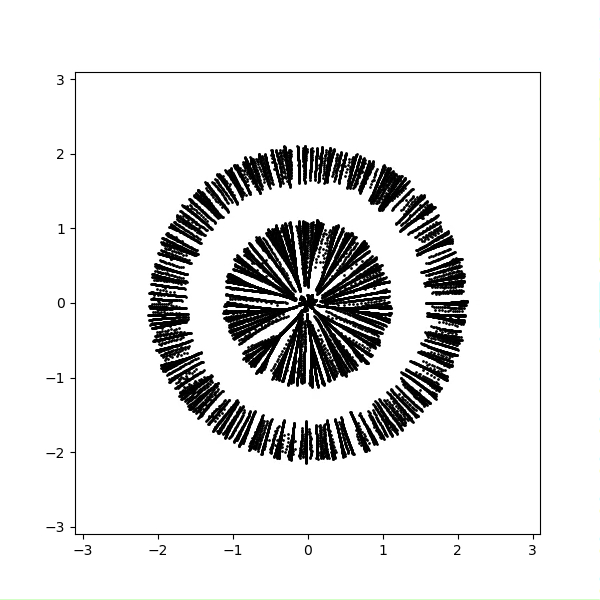}
\includegraphics[width=0.32\textwidth]{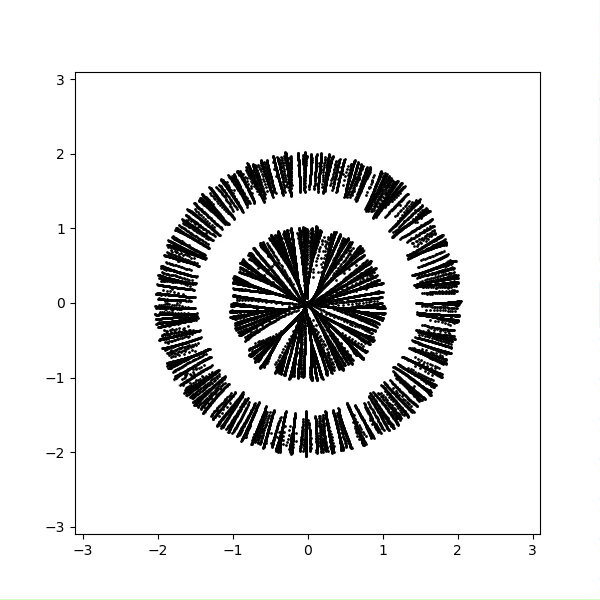}
\caption{
Roots of repeated derivatives for a random polynomial whose roots are independent and uniformly distributed on $3$ circles.  The number of roots at the beginning is $n=30000$.
The last snapshot shows the roots of the $10000$-th derivative, which corresponds to $t=1/3$.
}
\label{pic:three_circles_zeroes_plane}
\end{figure}

\subsection{Rotationally invariant complex zeroes}
It is natural to expect that the rotational invariance of $\mu_0$ is inherited by each $\mu_t$.
For this reason, it suffices to study the distribution of the absolute values (also called radial parts) of the zeroes.  We let $\psi(x,t)$ denote the limit density (if it exists) of the radial parts of the zeroes of the $[tn]$-th derivative at a point $x\geq 0$, that is
$$
\mu_t(\bD_r)  =  \int_0^r \psi(x,t) \dd x =: \Psi(r,t), \qquad t\in [0,1),
$$
where $\bD_r=\{z\in \C: |z|<r\}$ is an open disk of radius $r > 0$ centered at the origin.
The density of the measure $\mu_t$ with respect to the standard Lebesgue measure $\lambda$ on $\C$ is then given by
$$
u(z,t) := \frac{\dd \mu_t}{\dd \lambda} (z) =   \frac{\psi(|z|,t)}{2\pi |z|},
\qquad z\in\C.
$$
Given the initial density $\psi(x,0)$ of the radial parts of the zeroes at time $t=0$, we are interested in determining $\psi(x,t)$ for all $0\leq t<1$. A non-rigorous solution to this problem has been obtained by O'Rourke and Steinerberger~\cite{orourke_steinerberger_nonlocal} who derived the following PDE for the function $\psi(x,t)$:
\begin{equation}\label{eq:PDE}
\frac{\partial \psi (x,t)}{\partial t} = \frac{\partial}{\partial x}  \left( \frac{\psi(x,t)}{\frac 1x \int_0^x \psi(y,t) \dd y}\right),
\qquad
x\geq 0, \;\; t\in [0,1).
\end{equation}

In Section~\ref{sec:two_approaches} of the present paper, we shall use a completely different approach 
to derive an expression for $\psi(x,t)$ which is explicit up to evaluating certain inverse functions at two places. At a first glance, both approaches lead to very different results, but we shall show their equivalence by verifying that our explicit solution satisfies the PDE~\eqref{eq:PDE} of O'Rourke and Steinerberger. This will be done in the same Section~\ref{sec:two_approaches}.   In Section~\ref{sec:solvable_cases} we shall present a number of special cases in which the solution can be written down in a fully explicit form. For example, we shall analyze the case in which $\psi(x, 0)$ is  an arbitrary convex combination of delta functions.  Although we do not have a rigorous proof that our $\psi(x,t)$ indeed describes the asymptotic distribution of zeroes (in the sense that~\eqref{eq:mu_t_n_conv_mu_t_conj} holds), we shall present  strong numerical evidence for the validity of~\eqref{eq:mu_t_n_conv_mu_t_conj}. This will be done in the same Section~\ref{sec:solvable_cases}.

\subsection{Real zeroes}\label{subsec:real_zeroes}
Let now $\mu_0$ be concentrated on the real axis. This property is shared by all $\mu_t$'s since by Rolle's theorem all zeroes of any derivative of $P_n$ are real. Let $u(x,t)$ be the Lebesgue density of the measure $\mu_t$. Steinerberger~\cite{steinerberger_real} argued that $u(x,t)$  should satisfy the following PDE:
\begin{equation}\label{eq:PDE_real}
\frac{\partial u}{\partial t}  +  \frac 1 \pi \frac{\partial}{\partial x} \left(\arctan \left(\frac {Hu}{u}\right)\right) = 0,
\end{equation}
where
$$
H u(x,t) = \frac 1\pi \, \text{p.v.} \int_\R \frac{u(y,t)}{x-y} \dd y
$$
is the Hilbert transform of $y\mapsto u(y,t)$ and  the integral is taken in the sense of principal value.  A periodic version of this PDE describing the roots of trigonometric polynomials has been studied by Kiselev and Tan~\cite{kiselev_tan}. In~\cite{steinerberger_conservation}, Steinerberger has shown that $u(x,t)$ satisfies an infinite number of conservation laws.  Finally, in a recent paper~\cite{steinerberger_free}  he derived an interpretation of $u(x,t)$ as a free convolution power of the initial condition $u(x,0)$.
It has been pointed out by an anonymous referee that the most natural way to treat the real  case is to interpret it in terms of finite free probability, a subject developed in~\cite{marcus}, \cite{marcus_spielman_srivastava}, \cite{gorin_marcus}. We quote the following three observations from the referee's report: ``1) The finite free probability converges to the usual free probability in its limit; 2) Taking the derivative corresponds to a finite free projection of an $n\times n$-matrix into $n-1$ dimensions; 3) The operation of taking $tn$ derivatives of an $n$-dimensional polynomial and then taking the limit as $n$ goes to infinity should then converge to the the free multiplicative convolution of (a) an operator with spectral distribution matching the original root distribution, and (b) a projection having normalized trace $t$ (modulo an extra weight of $1-t$ placed at $0$)'' - end of quote.

In Section~\ref{sec:real_zeroes}, we shall present a different  approach to the real case which is similar in spirit to the one used in the case of complex zeroes. The idea is first to relate the zeroes of $P_n$ to the exponential rate of growth of the coefficients of $P_n$, which is a result of Van Assche, Fano and Ortolani~\cite{van_assche_fano_ortolani},  then to compute the impact of repeated differentiation on the coefficients, and finally to go back to zeroes.
In contrast to the case of the complex zeroes, this approach is rigorous. While the case of the real zeroes seems to be known in the free probability community (although we were not able to find an explicit statement in the literature), the approach we present in Section~\ref{sec:real_zeroes} may be of some interest because it provides an additional support for the conjectures in the complex case which do not admit a known interpretation in terms of free probability.


\subsection{Remark on continuous-time dynamics}
For simplicity of notation, we usually consider the $[tn]$-th derivatives of polynomials, although the results apply without changes to derivatives of any order $tn + o(n)$.  Let us note in passing that although the $tn$-th derivative is well-defined for $t\in \{0, \frac 1n, \frac 2n, \ldots, 1\}$ only, it is possible to embed this discrete-time dynamics into a continuous-time one as follows. Define the fractional derivatives of $P_n$  of any order $\alpha = tn  \in [0,n]$ by
\begin{equation}\label{eq:fract_der}
z^\alpha \cdot  P_n^{(\alpha)} (z) := \sum_{k=\lfloor \alpha\rfloor}^n  \frac{P^{(k)}(0)}{\Gamma(k-\alpha+1)}  z^{k}.
\end{equation}
Note that the right-hand side of~\eqref{eq:fract_der} is a polynomial of degree $n$ (which is the reason why we prefer to consider $z^\alpha \cdot  P_n^{(\alpha)} (z)$ rather than $P_n^{(\alpha)} (z)$). The right-hand side has $n$ zeroes for every, not necessarily integer, $\alpha \in [0,n]$. If $\alpha = m\in \{0,1,\ldots,n\}$ happens to be integer, these zeroes coincide with the $n-m$ zeroes of $P_n^{(m)} (z)$ together with a zero at the origin having multiplicity $m$. Since the coefficients of the right-hand side of~\eqref{eq:fract_der}  depend on $\alpha$ continuously, the same conclusion holds for its zeroes. As $\alpha$ approaches an integer value $m\in \{1,2,\ldots,n\}$ from the left, one of the zeroes converges to the origin and stays there for $\alpha>m$.

\section{Complex zeroes of repeated derivatives}\label{sec:two_approaches}
\subsection{The PDE approach}
The argument of O'Rourke and Steinerberger~\cite{orourke_steinerberger_nonlocal} used to derive~\eqref{eq:PDE} goes essentially as follows.
It is known from several papers, see, e.g.,~\cite[Eqn.~(2.8)]{kabluchko_seidel}, that near each zero $z=Z_k$ of the polynomial $P_n$ there is (with probability converging to $1$) a zero $\zeta$ of $P_n'$ located at $z  - \frac {1}{n G_0(z)} + o(\frac 1n)$, where
\begin{equation}\label{eq:cauchy_tranf_def}
G_0(z) = \int_\C  \frac {\mu_0 (\dd u)} {z-u}, \qquad z\in \C,
\end{equation}
is the Cauchy-Stieltjes transform of $\mu_0$. On a non-rigorous level, this formula can be easily guessed by noting that any critical point $\zeta$ is a  zero of the logarithmic derivative
$$
\frac{P_n'(z)}{P_n(z)} = \sum_{j=1}^n \frac {1}{z-Z_j}
$$
 and therefore the critical point $\zeta$ close to $z=Z_k$ satisfies
$$
\frac{1}{\zeta - z} = - \sum_{\substack{j\in \{1,\ldots,n\}\\j\neq k}} \frac 1 {\zeta - Z_j}
\sim - n \int_\C  \frac {\mu_0 (\dd u)} {\zeta-u}
\sim - n G_0(z)
$$
by the law of large numbers. This yields the claimed formula for $\zeta$.

Since passing from the original polynomial $P_n$ to its first derivative $P_n'$ corresponds to passing from $t=0$ to $t=1/n$,
this means that at time $t=0$ the complex root at $z\in \C$ moves at a speed given by $-1/G_0(z)$.  Assuming that this can be extended to the dynamics of the roots at any time $t$ (even though the roots of $P_n^{(tn)}$  are not stochastically independent anymore), one arrives at the conjecture that the speed of the root at position $z$ and time $t\in [0,1)$ is given by $-1/G_t(z)$, where $G_t$ is the Cauchy-Stieltjes transform  of $\mu_t$. Let us note in passing that, conjecturally, this conclusion applies to distributions of roots that are not  necessarily rotationally invariant. In the special case of rotationally invariant distributions, the Cauchy-Stieltjes transform can be computed explicitly, namely
$$
G_t (z)
:=
\int_\C  \frac {\mu_t (\dd u)} {z-u}
=
\frac {2\pi}{z} \int_0^{|z|} y u(y,t) \dd y
=
\frac {1}{z} \int_0^{|z|} \psi(y,t) \dd y;
$$
see, e.g., \cite[Proposition~3.1]{kabluchko_seidel}.
It follows that  under repeated differentiation the roots move in the radial direction towards the origin and the speed of the radial parts at $x>0$ is given by
$$
v(x,t) = -\left(\frac 1x \int_0^x \psi(y,t) \dd y\right)^{-1}.
$$
Thus, the density of the radial parts of the roots evolves according to the convection equation
$$
\frac{\partial \psi (x,t)}{\partial t} = - \frac{\partial}{\partial x}  \left( v(x,t) \psi(x,t) \right),
\qquad
x\geq 0, \;\; t\in [0,1),
$$
which is the PDE~\eqref{eq:PDE} derived by O'Rourke and Steinerberger~\cite{orourke_steinerberger_nonlocal}.
For polynomials whose roots are real, similar arguments~\cite{steinerberger_real} yield the non-local transport equation~\eqref{eq:PDE_real}; see also~\cite{steinerberger_conservation,steinerberger_free}.

\subsection{Approach based on polynomials with independent coefficients}\label{subsec:log_potential}
The basic idea of our approach is to pass from polynomials whose \textit{roots} are stochastically independent to polynomials whose \textit{coefficients} are stochastically independent. Although we are not able to justify this in a rigorous manner, it is natural to assume that both types of random polynomials behave in a similar way  and lead to the same dynamics of the roots under repeated differentiation in the large degree limit. We start by recalling the results on polynomials with independent coefficients from~\cite{kabluchko_zaporozhets12a}.

\subsubsection{Polynomials with independent coefficients}
Let $\xi_0,\xi_1,\ldots$ be independent and identically distributed (i.i.d.)\ non-degenerate random variables with values in $\C$ such that $\E \log (1+|\xi_0|) <\infty$. The exact form of the distribution is irrelevant for what follows. The reader may think of real or complex Gaussian variables, for example.  We are interested in the following random polynomials of the complex variable $z$:
\begin{equation}\label{eq:G_n_def}
G_n(z) = \sum_{k=0}^n \xi_k f_{k,n} z^k.
\end{equation}
The coefficients $f_{k,n}$, where $n\in\N$ and $k\in \{0,\ldots,n\}$, are required to be deterministic complex numbers satisfying the condition
\begin{equation}\label{eq:f_k_n_v}
\lim_{n\to\infty} \max_{k\in \{0,\ldots,n\}} \left|  \frac 1n \log |f_{k,n}| + v\left(\frac kn\right) \right|=0
\end{equation}
for some continuous function $v:[0,1]\to \R$. Equivalently, we have $f_{k,n} = \eee^{- n v(k/n) + o(n)}$ with an $o$-term that is uniform in $k$, which is why the function $-v$ may be called the \textit{exponential profile} of the coefficients. As we shall see in a moment, the exponential profile determines the distribution of the roots.

Under the above (or even weaker) assumptions, it is known~\cite[Theorem~2.8]{kabluchko_zaporozhets12a} that the empirical measure of zeroes of $G_n$ converges to a well-defined limit as $n\to\infty$, namely
\begin{equation}
\frac 1n \sum_{z\in \C: G_n(z) = 0} \delta_z \toprobab \mu_0,
\end{equation}
in probability on the space $\mathcal M(\C)$ of finite measures on $\C$ endowed with the topology of weak convergence. Here, $\mu_0$ is a rotationally invariant deterministic probability measure on $\C$ characterized by the formula
\begin{equation}\label{eq:mu_0_I'}
\mu_0(\bD_r) = I'_-(\log r), \qquad r>0,
\end{equation}
where $\bD_r=\{z\in \C: |z|<r\}$ is an open disc of radius $r>0$ centered at the origin, while $I'_-(s)$ denotes the left derivative of the convex function $I(s)$ defined as  the Legendre transform of $v(x)$, that is
\begin{equation}\label{eq:I_def}
I(s) = \sup_{x \in [0,1]} (sx - v(x)), \qquad s\in \R.
\end{equation}

Let the  function $v:[0,1]\to \R$ be convex (as is the case in all examples studied below). Then, $v$ has well-defined left and right derivatives $v_-'$ and $v_+'$. The support of $\mu_0$ is contained in the annulus $\{z\in \C: r_-\leq |z|\leq r_+\}$ whose inner and outer radii are $r_-=\eee^{v'_+(0)}$ and $r_+=\eee^{v_-'(1)}$, respectively (the former number may be $0$, while the latter one may be $+\infty$). If the function $v$ is differentiable on $(0,1)$ and  $v'$ is \emph{strictly} increasing (i.e.\ $v$ has no linearity intervals), then $I':(v_+'(0), v_-'(1))\to (0,1)$ is just the inverse function of $v':(0,1)\to (v_+'(0), v_-'(1))$ and vice versa, which is a well-known property of the Legendre transform. This remark will be frequently used to compute $v'$ or $I'$ in the explicit examples given below.
In the general case, two sorts of complications are possible.

\vspace*{2mm}
\noindent
\textit{Jumps of $v'$ correspond to void annuli}. It may happen that at some point $x\in (0,1)$ we have $v_-'(x) < v_+'(x)$. Such a jump of the derivative corresponds to an interval $(v_-'(x), v_+'(x))$ on which  $I'=x$ stays constant, which means that there is a void annulus in the support of $\mu_0$. The inner and outer radii of this void annulus are $\eee^{v_-'(x)}$ and $\eee^{v_+'(x)}$, respectively.

\vspace*{2mm}
\noindent
\textit{Constancy intervals of $v'$ correspond to circles of zeroes}.
If the function $v'$ takes a constant value $c$ on some interval $(x_0,x_1)$, then the function $I'$ has a jump at $c$, and the size of the jump is $x_1-x_0$. This means that the radial component of the measure $\mu_0$ has an atom of mass $x_1-x_0$ at the point $\eee^c$. That is,  there is a circle of zeroes of radius $\eee^c$ and total mass $x_1-x_0$.

\vspace*{2mm}
Generalizing the above considerations, we can characterize the left derivative $I_-':(v_+'(0), v_-'(1))\to (0,1)$ as the generalized left-continuous inverse of the function $v_-':(0,1)\to (v_+'(0), v_-'(1))$, namely
$$
I_-'(s) = \inf\{x \in  \R : v'_-(x) \geq  s\}.
$$

\subsubsection{Repeated derivatives}
We shall now describe how the exponential profile of a random polynomial changes under taking repeated derivatives. This has been done by Feng and Yao~\cite[Theorem~5]{feng_yao}. We provide the details of the argument since it will be needed in the following.
We take some $t\in [0,1)$ and look at the $[tn]$-th derivative of $G_n$ as defined in~\eqref{eq:G_n_def}:
\begin{align*}
G_n^{([tn])}(z)
&=
 \sum_{k=[nt]}^{n} \xi_k f_{k,n} k(k-1)\ldots (k-[nt]+1) z^{k-[nt]}\\
&=
\sum_{\ell=0}^{n-[nt]} \xi_{\ell + [nt]} f_{\ell + [nt],n} (\ell+[nt])(\ell + [nt] -1) \ldots (\ell + 1) z^{\ell}.
\end{align*}
This function has almost the same form as the original function $G_n$, but the coefficients $f_{k,n}$ should be replaced by the following new ones:
\begin{align*}
\tilde f_{\ell, n}
&:=
f_{\ell + [nt],n} (\ell+[nt])(\ell + [nt] -1) \ldots (\ell + 1)\\
&=
f_{\ell + [nt],n} \frac{\Gamma(\ell + [nt] + 1)}{\Gamma(\ell+1)},
\qquad
\ell\in \{0,\ldots, n-[nt]\}.
\end{align*}
Let us now put $\alpha:= \ell/n\in [0,1-t+o(1)]$ and compute the $v$-function of the new coefficients:
\begin{align*}
-\frac 1n \log \tilde f_{\ell, n}
&=
- \frac 1n \log  f_{\alpha n+ [t n], n} - \frac 1n \log \frac{\Gamma(\alpha n + [t n] + 1)}{\Gamma(\alpha n+1)}\\
&=
v(\alpha+t) - \left((\alpha+t) \log (\alpha+t)  - \alpha \log \alpha \right) + t - t\log n + o(1),
\end{align*}
as $n\to\infty$, where we used~\eqref{eq:f_k_n_v}, the continuity of $v$ and the asymptotics $\Gamma(x+1) = x\log x - x + o(x)$, as $x\to\infty$. The $o$-term is uniform in $\ell\in \{0,\ldots, [nt]\}$. The terms $t$ and $t\log n$ do not depend on $\alpha$ and can be eliminated by multiplying all coefficients $\tilde f_{\ell,n}$ by $\eee^{tn-tn\log n}$, which does not influence the distribution of  zeroes. Thus, we can drop these terms and arrive at the conclusion that the distribution of zeroes of the $[tn]$-th derivative of $G_n$ can be computed using the same recipe as for $G_n,$ the only difference being that the function $v(x)$ should be replaced  by the following one:
\begin{equation}\label{eq:v_x_t_def}
v(x,t) = v(x+t) -  (x+t) \log (x+t)  + x \log x, \qquad 0\leq x \leq 1-t, \;\; 0\leq t <1.
\end{equation}
Note that the function $v(\cdot, t)$ is defined on the interval $[0,1-t]$, which is slightly different from the setting introduced above but is covered by the more general assumptions of~\cite{kabluchko_zaporozhets12a} (one may naturally extend $v(\cdot, t)$ to the interval $[0,1]$ by putting $v(x,t) = +\infty$ for $x>t$). We define the Legendre transform of $x\mapsto v(x,t)$ as follows:
\begin{equation}\label{eq:I_s_t_def}
I(s,t) := \sup_{x\in [0,1-t]} (sx - v(x,t)), \qquad s\in \R.
\end{equation}
Then, applying~\cite[Theorem~2.8]{kabluchko_zaporozhets12a}  yields that the zeroes of the $[tn]$-th derivative of $G_n$ are distributed according to a certain rotationally invariant measure $\mu_t$ on $\C$, namely
\begin{equation}
\frac 1n \sum_{z\in \C: G_n^{([tn])}(z) = 0} \delta_z \toprobab \mu_t
\end{equation}
in probability on $\mathcal M(\C)$. Feng and Yao~\cite[Theorem~5]{feng_yao} proved this claim rigorously.
Moreover, $\mu_t$ has  total mass $1-t$ and is characterized by
\begin{equation}\label{eq:mu_t_formula}
\mu_t(\bD_r) = (\partial_1 I)(\log r, t), \qquad r>0,
\end{equation}
where $\partial_1 I(s,t) =: J(s,t)$ denotes the (left) partial derivative of $I(s,t)$ with respect to its first argument. Again, the function $\partial_1 I(\cdot, t)$ is the generalized left-continuous inverse of the function $\partial_1 v(\cdot, t)$ and vice versa.

\subsubsection{Polynomials with i.i.d.\ roots}
Consider now a polynomial $P_n(z) := \prod_{k=1}^n (z-Z_k)$ whose zeroes $Z_1,\ldots,Z_n$ are i.i.d.\ random variables distributed according to certain rotationally invariant probability measure $\mu_0$ on $\C$ having no atom at $0$. The above considerations suggest the following recipe for computing the limit density of zeroes of the $[tn]$-th derivative.   First, find a convex function $v:[0,1]\to\R$ such that~\eqref{eq:mu_0_I'} and~\eqref{eq:I_def} hold. To this end, put $I'_-(s) := \mu_0(\bD_{e^s})$, $s\in \R$, and define $v'_-$ by inverting this function. Consider random polynomials $G_n$ defined by~\eqref{eq:G_n_def} with $f_{k,n}:= \eee^{-n v(k/n)}$. Then, the roots of $P_n$ and $G_n$ have the same asymptotic distribution $\mu_0$ in the large degree limit. It is natural to conjecture that in the large degree limit, the the roots of both polynomials behave in the same way under repeated differentiation.
More precisely, we have the following
\begin{conjecture}
For every $t\in [0,1)$ we have
$$
\frac 1n \sum_{z\in \C: P_n^{([tn])}(z)=0} \delta_z \toprobab \mu_t
$$
in probability on the space $\mathcal M(\C)$, where $\mu_t$ is a deterministic rotationally invariant measure on $\C$ given by~\eqref{eq:v_x_t_def},  \eqref{eq:I_s_t_def}, \eqref{eq:mu_t_formula}.
\end{conjecture}

Numerical evidence for the validity of the conjecture will be provided below.  A simple way to express $\mu_t$ through $\mu_0$ will be stated in Section~\ref{subsec:Psi_t_Psi_0}.

\subsection{Comparison to the PDE approach}
Let $\psi(x,t)$ be the limit density of the absolute values of the zeroes of the $[tn]$-th derivative, for $t\in [0,1)$. Recall that the PDE derived by O'Rourke and Steinerberger~\cite{orourke_steinerberger_nonlocal} reads as follows:
\begin{equation}\label{eq:PDE_psi}
\frac{\partial \psi (x,t)}{\partial t} = \frac{\partial}{\partial x}  \left( \frac{\psi(x,t)}{\frac 1x \int_0^x \psi(y,t) \dd y}\right).
\end{equation}
Let us restate this PDE in terms of the corresponding distribution function
$$
\Psi(x,t) = \int_0^x \psi(y,t)\dd y.
$$
Then, we claim that $\Psi(x,t)$ satisfies the following PDE:
\begin{equation}\label{eq:PDE_Psi}
\frac{\partial \Psi (x,t)}{\partial t} = x  \, \frac{\frac{\partial}{\partial x} \Psi(x,t)}{\Psi(x,t)} - 1.
\end{equation}
Indeed,
\begin{align*}
\frac{\partial \Psi (x,t)}{\partial t}
&=
\frac{\partial}{\partial t}\int_0^x \psi(y,t)\dd y\\
&=
\int_0^x \frac{\partial}{\partial t} \psi(y,t)\dd y\\
&=
\int_0^x \frac{\partial}{\partial y}  \left( \frac{\psi(y,t)}{\frac 1y \int_0^y \psi(z,t) \dd z}\right) \dd y\\
&=
\int_0^x \frac{\partial}{\partial y}   \left( y  \, \frac{\frac{\partial}{\partial y} \Psi(y,t)}{\Psi(y,t)}\right)  \dd y\\
&=
x  \, \frac{\frac{\partial}{\partial x} \Psi(x,t)}{\Psi(x,t)}-1.
\end{align*}
Conversely, if $\Psi(x,t)$ solves~\eqref{eq:PDE_Psi}, then $\psi(x,t):= \frac{\partial}{\partial x} \Psi(x,t)$ solves~\eqref{eq:PDE_psi}, as one can easily check by taking the derivative in $x$ on both sides of~\eqref{eq:PDE_Psi}.

We are now ready to describe a method  assigning to each convex function $v(x)$ a certain solution of~\eqref{eq:PDE_Psi}. Given a function $f(x_1,x_2,\ldots)$, we denote by $\partial_i f$ its partial derivative with respect to the $i$-th argument $x_i$.

\begin{theorem}\label{theo:PDE_solution}
Let $v:[0,1] \to \R$ be a convex, two times differentiable function 
and define
\begin{equation}\label{eq:v_two_var}
v(x,t) := v(x+t) + x\log x - (x+t) \log (x+t),
\qquad
0\leq x \leq 1-t, \;\; 0 \leq t < 1.
\end{equation}
For every fixed $t\in (0,1)$ let $y\mapsto J(y,t)$ be the inverse function of the strictly monotone function $x\mapsto \partial_1 v(x,t)$, that is
\begin{equation}\label{eq:inverse_v_prime}
J(\partial_1 v(x,t),t) = x,
\end{equation}
where $y := \partial_1 v(x,t)$ takes values in the range $-\infty < y < v'_-(1)+\log (1-t)$.
Then, the following function solves~\eqref{eq:PDE_Psi}:
\begin{equation}\label{eq:Psi_J}
\Psi(r,t) := J(\log r, t), \qquad 0 < r < (1-t) \eee^{v'_-(1)}.
\end{equation}
\end{theorem}
\begin{remark}
The requirement of differentiability of $v$ can be relaxed to piecewise differentiability provided one carefully treats the points $x$ where $v_-'(x) < v_+'(x)$; see Section~\ref{subsec:log_potential} for details. 
\end{remark}
\begin{proof}[Proof of Theorem~\ref{theo:PDE_solution}]
First of all, note that for every fixed $t\in (0,1)$, the function $v(x,t)$ is  \textit{strictly} convex in $x\in [0,1-t]$ because
$$
\frac{\partial^2}{\partial x^2} \left(x\log x - (x+t) \log (x+t)\right) =  \frac{1}{x} - \frac 1{x+t} > 0.
$$
Therefore, the function $x\mapsto \partial_1 v(x,t)$ is strictly monotone and the inverse function $J$ exists. To prove~\eqref{eq:PDE_Psi}, it suffices to check that
$$
(\partial_2 J) (\log x, t) = \frac{(\partial_1 J)(\log x,t)}{J(\log x,t)}-1.
$$
Writing  $y:=\log x$, it suffices to show that
$$
(\partial_2 J) (y, t)  - \frac{(\partial_1 J)(y,t)}{J(y,t)} = -1.
$$
Writing $y := \partial_1 v(z,t)$, our task reduces to showing that
\begin{equation}\label{eq:PDE_techn1}
(\partial_2 J) (\partial_1 v(z,t), t)  - \frac{(\partial_1 J)(\partial_1 v(z,t),t)}{J(\partial_1 v(z,t),t)} = -1.
\end{equation}
To verify this relation, we proceed as follows.
Applying $\frac{\partial}{\partial t}$ to both sides of~\eqref{eq:inverse_v_prime}, we obtain
\begin{equation}\label{eq:J_partial_1}
(\partial_1 J)(\partial_1 v(x,t),t)\cdot (\partial_1\partial_2 v)(x,t) + (\partial_2 J)(\partial_1 v(x,t),t) = 0.
\end{equation}
Similarly, applying $\frac{\partial}{\partial x}$ to~\eqref{eq:inverse_v_prime}, yields
\begin{equation}\label{eq:J_partial_2}
(\partial_1 J)(\partial_1 v(x,t),t) \cdot (\partial_1 \partial_1 v)(x,t) =  1.
\end{equation}
Using~\eqref{eq:J_partial_1}, \eqref{eq:inverse_v_prime} and finally~\eqref{eq:J_partial_2}, we can rewrite the left-hand side of~\eqref{eq:PDE_techn1} as follows:
\begin{align*}
(\partial_2 J) (\partial_1 v(z,t), t)  &- \frac{(\partial_1 J)(\partial_1 v(z,t),t)}{J(\partial_1 v(z,t),t)}\\
&=
-(\partial_1 J)(\partial_1 v(z,t),t)\cdot (\partial_1\partial_2 v)(z,t) -\frac{(\partial_1 J)(\partial_1 v(z,t),t)}{z}
\\
&=
- \frac{(\partial_1\partial_2 v)(z,t)}{(\partial_1 \partial_1 v)(z,t)} -\frac{1}{z (\partial_1 \partial_1 v)(z,t)}.
\end{align*}
Thus, to prove~\eqref{eq:PDE_techn1} it suffices to verify that
\begin{equation}\label{eq:suffices_to_verify}
(\partial_1\partial_2 v)(z,t) + \frac{1}{z} = (\partial_1 \partial_1 v)(z,t).
\end{equation}
Differentiating~\eqref{eq:v_two_var}, we have
\begin{align*}
(\partial_1\partial_2 v)(z,t)  &= v''(z+t) - \frac 1 {z+t}, \\
(\partial_1 \partial_1 v)(z,t) &= v''(z+t) + \frac 1 z - \frac 1 {z+t},
\end{align*}
and the claim~\eqref{eq:suffices_to_verify} follows.
\end{proof}

\subsection{Solution to the PDE}\label{subsec:Psi_t_Psi_0}
The simplest way in which $\Psi_t(r) := \Psi(r,t) = \mu_t(\bD_r)$ can be related to $\Psi_0(r):= \Psi(r,0)=\mu_0(\bD_r)$ is the following one:
\begin{equation}\label{eq:Psi_t_Psi_0}
\frac{\Psi_t^{-1}(x)}{x} = \frac{\Psi_0^{-1}(x+t)}{x+t}, \qquad 0< x < 1-t, \;\; 0 \leq t < 1,
\end{equation}
where $\Psi_t^{-1}(x)$ denotes the  inverse function of $\Psi_t$. Let us derive this relation assuming for simplicity that the initial density $\psi(x,0)$ is a measurable function strictly positive on the interval $[0,R]$ and vanishing outside it. This guarantees that the inverse functions are well defined.  Define a function $w_0(x) = \eee^{v'(x)}: [0,1]\to [0,R]$ by the relation $\Psi_0(w_0(x)) = x$.  By differentiating~\eqref{eq:v_x_t_def}, the function $w_t(x) := \eee^{\partial_1 v(x,t)}:[0,1-t]\to [0, (1-t) R]$ satisfies
$$
w_t(x) = w_0(x+t) \frac{x}{x+t}, \qquad 0 \leq t < 1,\;\;  0\leq x \leq 1-t.
$$
Moreover, \eqref{eq:mu_t_formula} yields
$$
\Psi_t (w_t(x)) = \Psi_t (\eee^{\partial_1 v(x,t)}) = \mu_t (\bD_{\eee^{\partial_1 v(x,t)}}) = \partial_1 I( \partial_1 v(x,t),t ) = x.
$$
Specializing this to $t=0$, we have $w_0(x) = \Psi_0^{-1}(x)$. For arbitrary $t\in [0,1)$ and $x\in [0,1-t]$, we obtain
$$
\Psi_t^{-1}(x) = w_t(x) = w_0(x+t) \frac{x}{x+t} = \Psi_0^{-1}(x+t)\frac{x}{x+t},
$$
which proves~\eqref{eq:Psi_t_Psi_0}.

\section{Explicitly solvable special cases for complex zeroes}\label{sec:solvable_cases}
\subsection{Uniform distribution on the circle: Kac polynomials}
Let the roots of the polynomial $P_n$ be i.i.d.\ random variables  with the uniform distribution on the unit circle $\{|z| = 1\}$. The radial parts of all zeroes are equal to $1$, hence
$$
\psi(x,0) = \delta(x-1), \qquad x\geq 0,
$$
is the Dirac delta-function with peak at $1$.  The corresponding distribution function of the radial parts is
$$
\Psi(x,0)
=
\begin{cases}
0, &\text{if } x\in [0,1],\\
1,  &\text{if } x>1.
\end{cases}
$$
An example of random polynomials whose zeroes are asymptotically uniformly distributed on the unit circle is given by the
Kac polynomials
\begin{equation}\label{eq:kac_poly_def}
K_n(z) = \sum_{k=0}^n \xi_k z^k,
\end{equation}
where $\xi_0,\xi_1,\ldots$ are i.i.d.\ random variables with $\E \log (1+|\xi_0|)<\infty$;  see, e.g., \cite{iz_log}.
The corresponding coefficients $f_{k,n}=1$, $k\in \{0,\ldots,n\}$,  satisfy~\eqref{eq:f_k_n_v} with
$$
v(x)
=
0, \qquad x\in [0,1].
$$
It would be possible to compute $v(x)$ directly using~\eqref{eq:mu_0_I'} and~\eqref{eq:I_def}, see Section~\ref{subsec:log_potential}, but we omit this since a more general computation will be done in Section~\ref{subsec:void_annuli}.
We are now going to use Theorem~\ref{theo:PDE_solution} to compute the corresponding solution of PDE~\eqref{eq:PDE}. For arbitrary $t\in [0,1)$ we have
$$
v(x,t)
=
-(x+t)\log (x+t) + x\log x, \qquad 0\leq x\leq 1-t.
$$
The derivative in $x$ is given by
$$
\partial_1 v(x,t) = \log \frac{x}{x+t}, \qquad 0 < x < 1-t.
$$
For $t\neq 0$ this function is monotone increasing and its range is the interval $(-\infty, \log (1-t))$. The inverse function $y\mapsto J(y,t)$ is therefore characterized  by
$$
\log \left(\frac{J(y,t)}{J(y,t) + t}\right) = y,\qquad y\in (-\infty, \log (1-t)).
$$
Solving this equation, we arrive at
$$
J(y,t) = \frac{\eee^y t}{1-\eee^y}, \qquad y\in (-\infty, \log (1-t)).
$$
It follows that
$$
\Psi(x,t) = J(\log x, t) = \frac{xt}{1-x}, \qquad 0 < x <  1-t, \;\; 0<t<1.
$$
Differentiating in $x$, we arrive at the following expression for the density of the radial parts at time $t$:
\begin{equation}\label{eq:kac_evolution}
\psi(x,t) = \frac{t}{(1-x)^2} \ind_{\{0 <  x < 1-t\}}, \qquad  x\geq 0, \;\; 0 < t < 1.
\end{equation}
We recall that our normalization is such that $\int_{0}^\infty \psi(x,t) \dd x = 1-t$. For the repeated derivatives of the Kac polynomials $K_n(z)$, Feng and Yao~\cite[Theorem~3, part~(2)]{feng_yao} proved rigorously that at time $t\in (0,1)$  the radial parts of the roots are distributed according to~\eqref{eq:kac_evolution}. To this end, they showed that the repeated derivatives of $K_n(z)$ satisfy the general conditions of Theorem~2.5 in~\cite{kabluchko_zaporozhets12a}.  A visualization of this setting is shown in the first row of Figure~\ref{pic:zeroes_kac_and_three_circles}. For polynomials with independent roots, the same claim remains a conjecture.
Dropping the i.i.d.\ assumption, one may conjecture that~\eqref{eq:kac_evolution} continues to hold if the empirical measure of roots of the initial polynomial converges weakly to the uniform distribution on the unit circle and some additional condition excluding the trivial counterexample $P_n(z) = z^n-1$ is satisfied.

\begin{figure}[!tbp]
\includegraphics[width=0.32\textwidth]{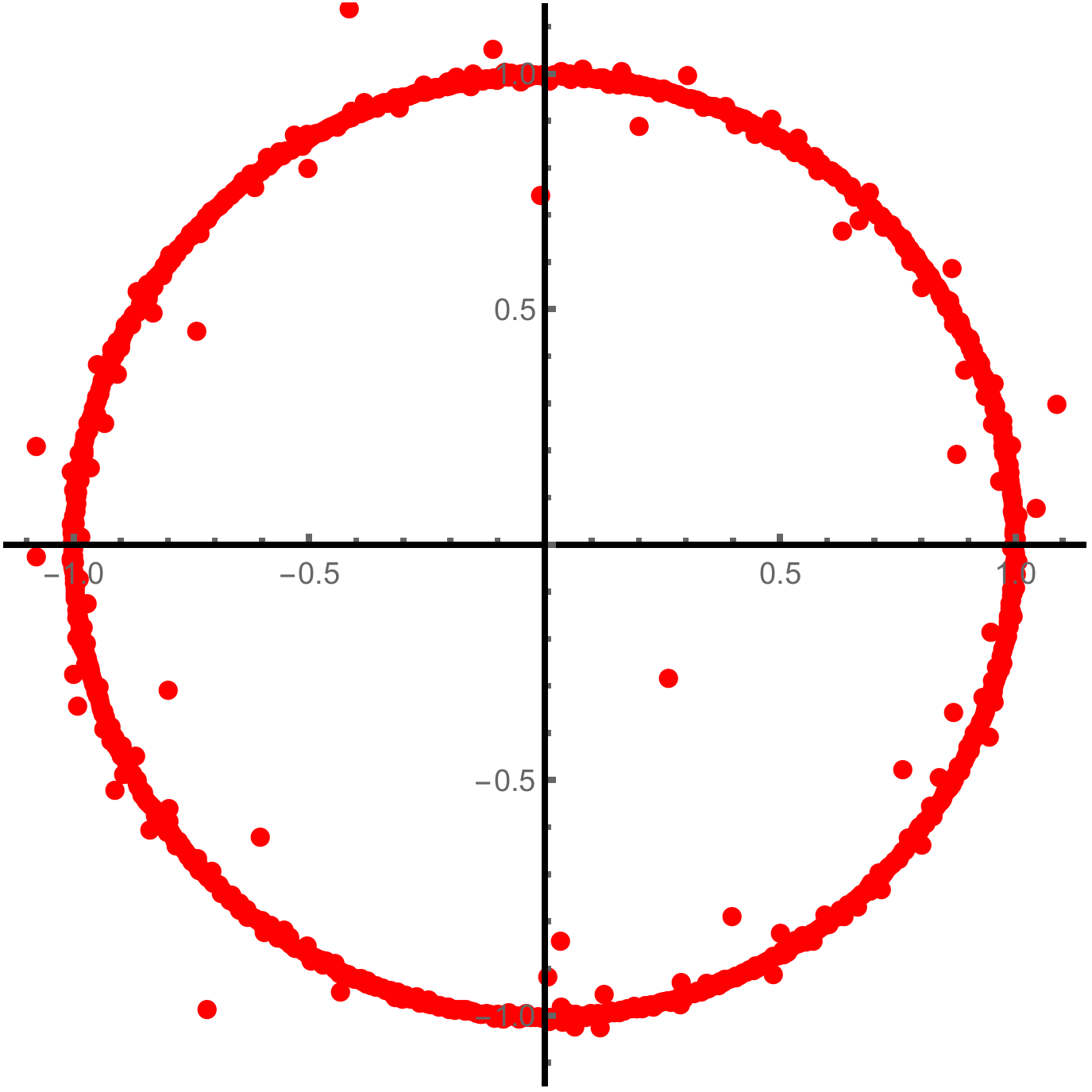}
\includegraphics[width=0.32\textwidth]{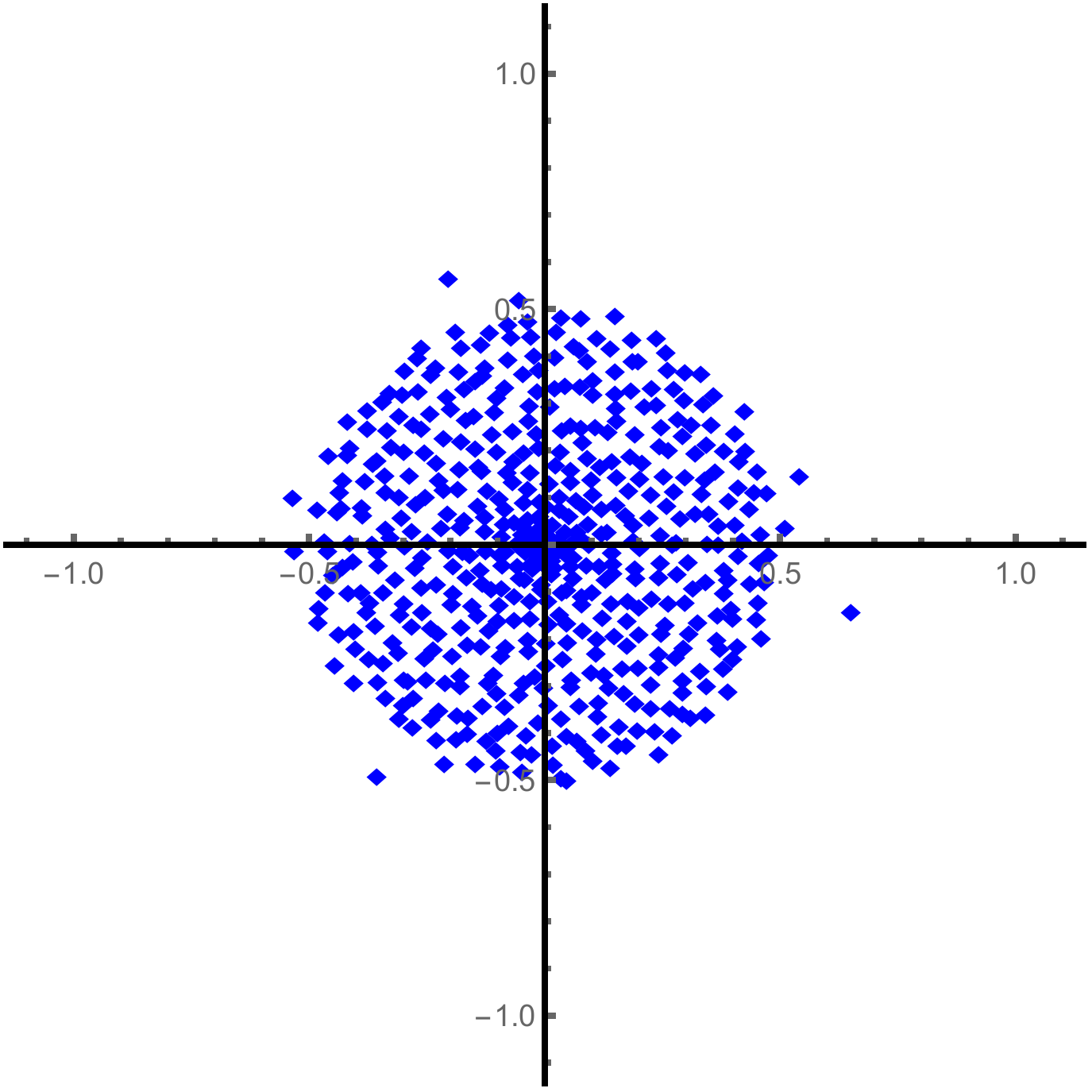}
\includegraphics[width=0.32\textwidth]{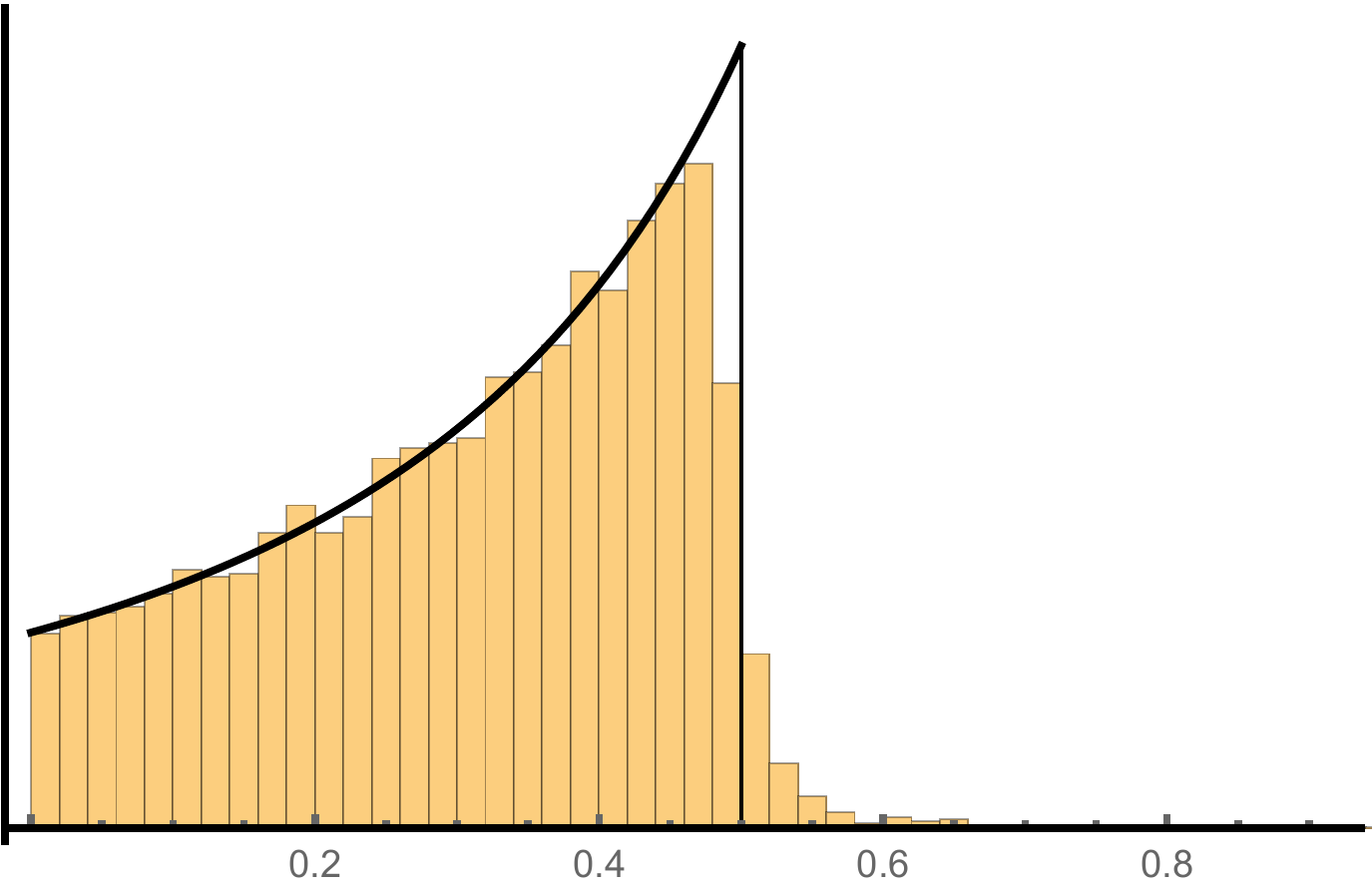}
\includegraphics[width=0.32\textwidth]{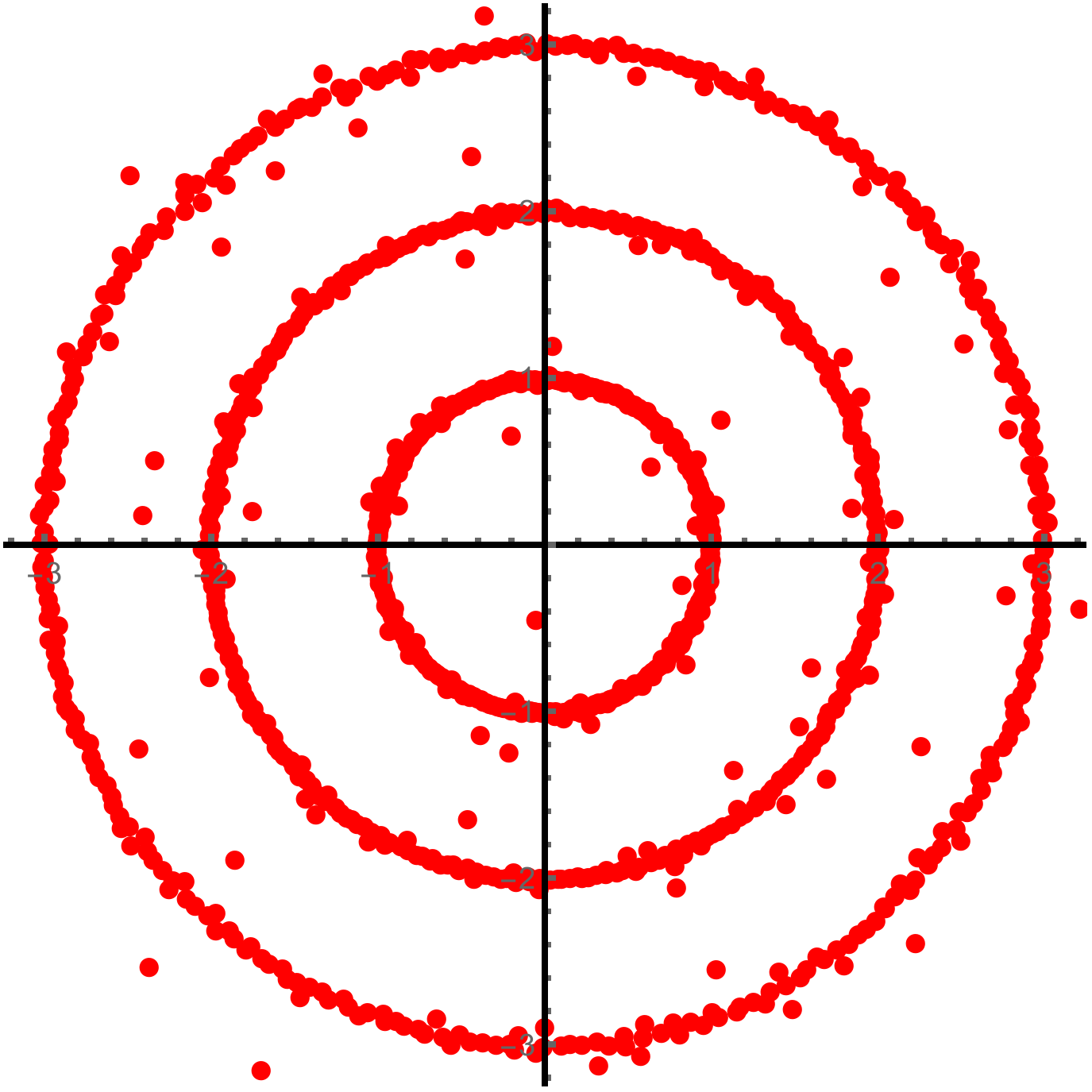}
\includegraphics[width=0.32\textwidth]{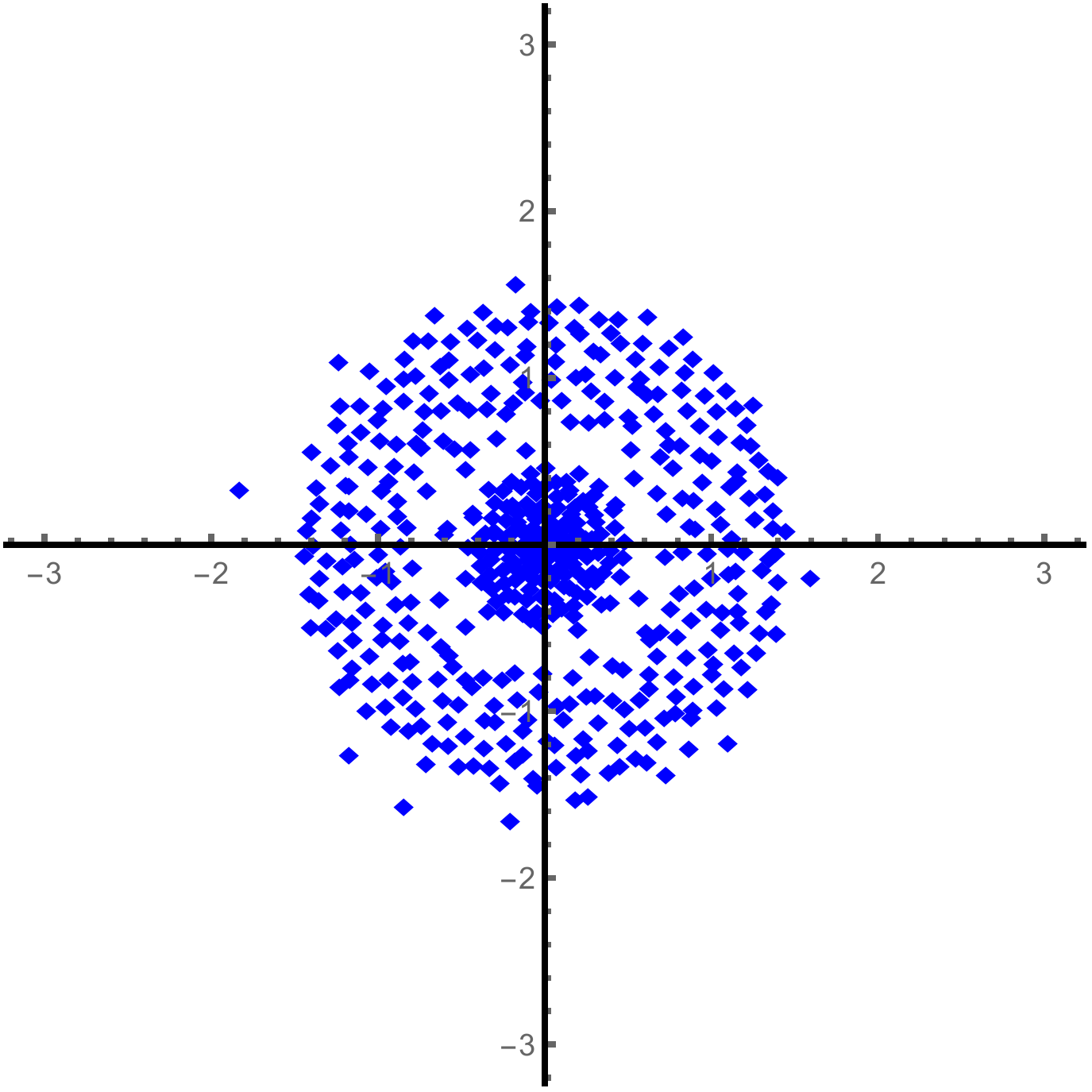}
\includegraphics[width=0.32\textwidth]{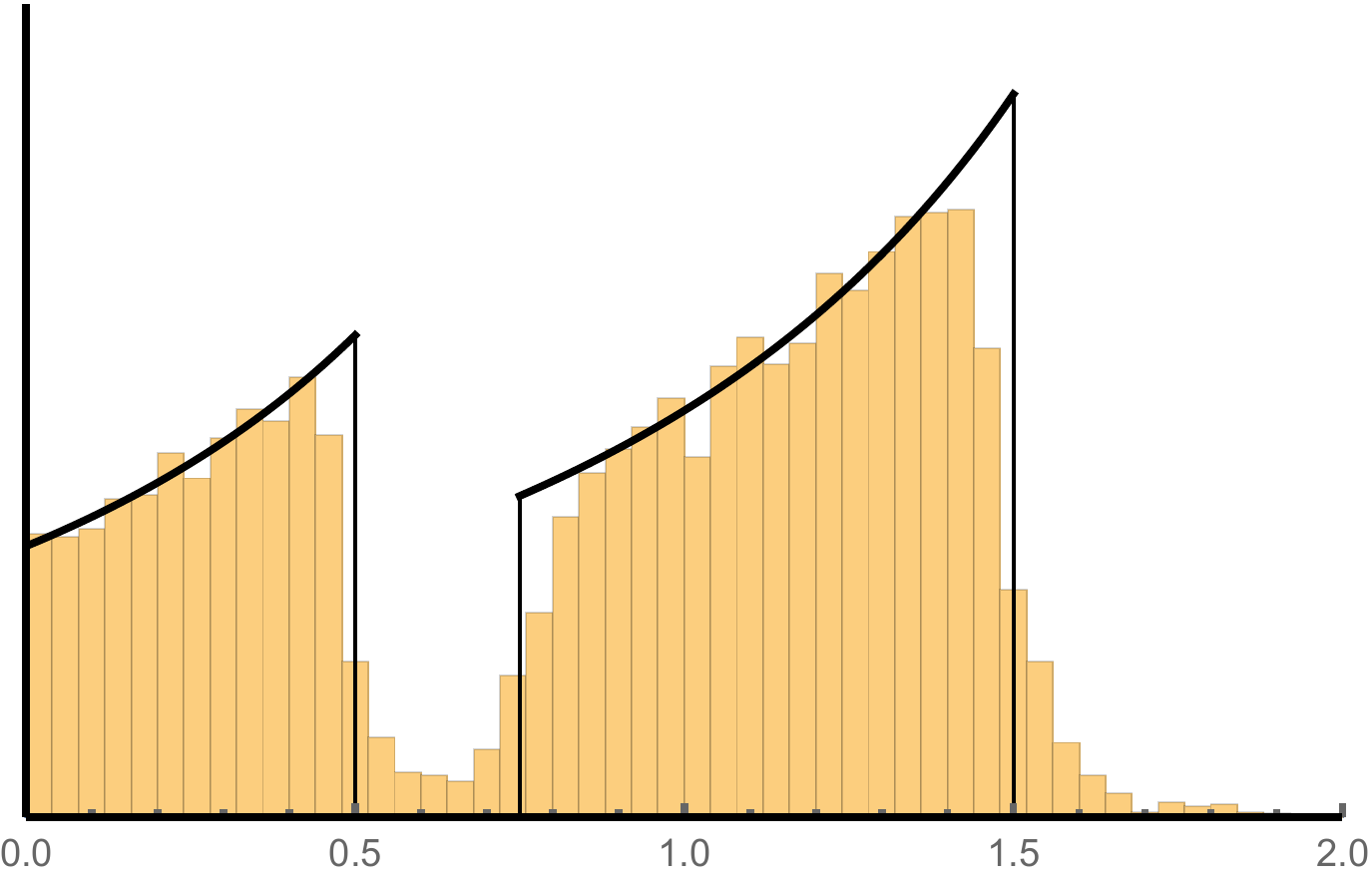}
\caption{
First row. Left: Zeroes of a Kac polynomial $K_n(z)$ of degree $n=1000$ as defined in~\eqref{eq:kac_poly_def}.  Middle: Zeroes of its $500$-th derivative. Right: Histogram together with the theoretical density~\eqref{eq:kac_evolution} derived in~\cite{feng_yao}. Second row: Same for polynomials with independent coefficients whose exponential profile is such that the zeroes are uniformly distributed on a union of three circles with radii $r_1=1,r_2=2, r_3=3$; see Section~\ref{subsec:void_annuli}.
Both histograms show roots of $20$ independent realizations.
}
\label{pic:zeroes_kac_and_three_circles}
\end{figure}

\subsection{Several circles of  zeroes}\label{subsec:void_annuli}
Let us now consider an example in which the initial condition consists of several circles of zeroes.  We shall describe the solution of PDE~\eqref{eq:PDE} with the initial condition of the form
\begin{equation}\label{eq:initial_dirac_deltas}
\psi(x,0) = \sum_{i=1}^k p_i \delta(x-r_i),
\end{equation}
where $\delta(\cdot)$ is the Dirac delta-function,  $k\in\N$, $0 < r_1 < r_2 < \ldots < r_k<\infty$ and  $p_1,\ldots,p_k\geq 0$ satisfy $p_1+\ldots+p_k = 1$.  The distribution of the zeroes of the corresponding polynomial $P_n$ is thus a mixture of $k$ uniform distributions on  circles with radii $r_1,\ldots,r_k$, with $p_1,\ldots, p_k$ being the weights of the circles. The corresponding distribution function is piecewise constant and given by
$$
\Psi(x,0) = p_1+\ldots + p_\ell  = P_\ell \qquad \text{ if } r_\ell <  x  \leq  r_{\ell+1}, \;\; \ell\in \{0,\ldots,k\},
$$
where we defined $r_0:= 0$, $r_{k+1} := +\infty$ and
$$
P_\ell := p_1+\ldots + p_\ell, \qquad \ell\in \{1,\ldots, k\}, \qquad P_0:=0, \;\;\; P_k := 1.
$$
Let us determine the corresponding function $v(x)$. Using the identity $I'_-(\log r) = \Psi(r,0)$ we can first determine the function $I'_-(s)$ as follows:
$$
I'_-(s) = P_\ell \qquad \text{ if } \log r_\ell <  s  \leq  \log r_{\ell+1}, \;\; \ell\in \{0,\ldots,k\}.
$$
The generalized inverse function of $I'_-(s)$ is the piecewise constant function $v'_-(x)$, defined for $x\in (0,1]$ and given by
$$
v'_-(x) = \log r_\ell \qquad \text{ if }   P_{\ell-1} <  x  \leq  P_\ell, \;\; \ell\in \{1,\ldots,k\}.
$$
Now, let us take some $0<t<1$. Then, by~\eqref{eq:v_x_t_def}, the function $\partial_1 v(x,t)$ is defined for $x\in (0,1-t]$ and is explicitly given by
$$
\partial_1 v(x,t) = \log \frac{r_\ell x}{x+t},  \quad \text{ if }  P_{\ell-1}-t <  x  \leq  P_\ell-t, \;\; \ell\in \{1,\ldots,k\}.
$$

Let now $t\in (P_{m-1}, P_m)$ for some $m\in \{1,\ldots,k\}$. Then,  $x\mapsto \partial_1 v(x,t)$ is a piecewise continuous function on the interval $(0,1-t]$ with jumps at the points $P_m - t, P_{m+1}-t,\ldots, P_{k-1}-t$. Let us consider the intervals on which this function is continuous separately.

Let first $x\in (0,P_m-t)$. This is the first continuity interval of $\partial_1 v(x,t)$, and the range of this function on this interval is $(-\infty, \log (r_m(P_m-t)/P_m))$. The inverse function $J(y,t)$ is given by
$$
J(y,t) =  \frac{t \eee^y}{r_m-\eee^y},
\qquad
\text{ if }
-\infty < y < \log \frac{r_m(P_m-t)}{P_m}.
$$
Let now $P_{\ell-1}-t <  x  <  P_\ell-t$ for some $\ell\in \{m+1,\ldots,k\}$. Then, the inverse function is given by
$$
J(y,t) =  \frac{t \eee^y}{r_\ell-\eee^y},
\qquad
\text{ if }
\log \frac{r_\ell(P_{\ell-1}-t)}{P_{\ell-1}} < y < \log \frac{r_\ell(P_\ell-t)}{P_\ell}.
$$
Summarizing, it follows that
$$
\Psi(x,t)
=
J(\log x,t)
=
\begin{cases}
\frac{t x}{r_m-x}, &\text{ if } 0 < x < \frac{r_m(P_m-t)}{P_m},\\
\frac{t x}{r_\ell-x}, &\text{ if } \frac{r_\ell(P_{\ell-1}-t)}{P_{\ell-1}} < x < \frac{r_\ell(P_\ell-t)}{P_\ell} \text{ for some } \ell\in \{m+1,\ldots,k\}.
\end{cases}
$$
Differentiating with respect to $x$, we obtain the following solution to the PDE~\eqref{eq:PDE} with the initial condition~\eqref{eq:initial_dirac_deltas}: If  $t\in (P_{m-1}, P_m)$ for some $m\in \{1,\ldots,k\}$, then
\begin{equation}\label{eq:solution_dirac_deltas}
\psi(x,t)
=
\begin{cases}
\frac{t r_m}{(r_m-x)^2}, &\text{ if } 0 < x < \frac{r_m(P_m-t)}{P_m},\\
\frac{t r_\ell}{(r_\ell-x)^2}, &\text{ if } \frac{r_\ell(P_{\ell-1}-t)}{P_{\ell-1}} < x < \frac{r_\ell(P_\ell-t)}{P_\ell} \text{ for some } \ell\in \{m+1,\ldots,k\}.
\end{cases}
\end{equation}

\begin{figure}[!tbp]
\includegraphics[width=0.32\textwidth]{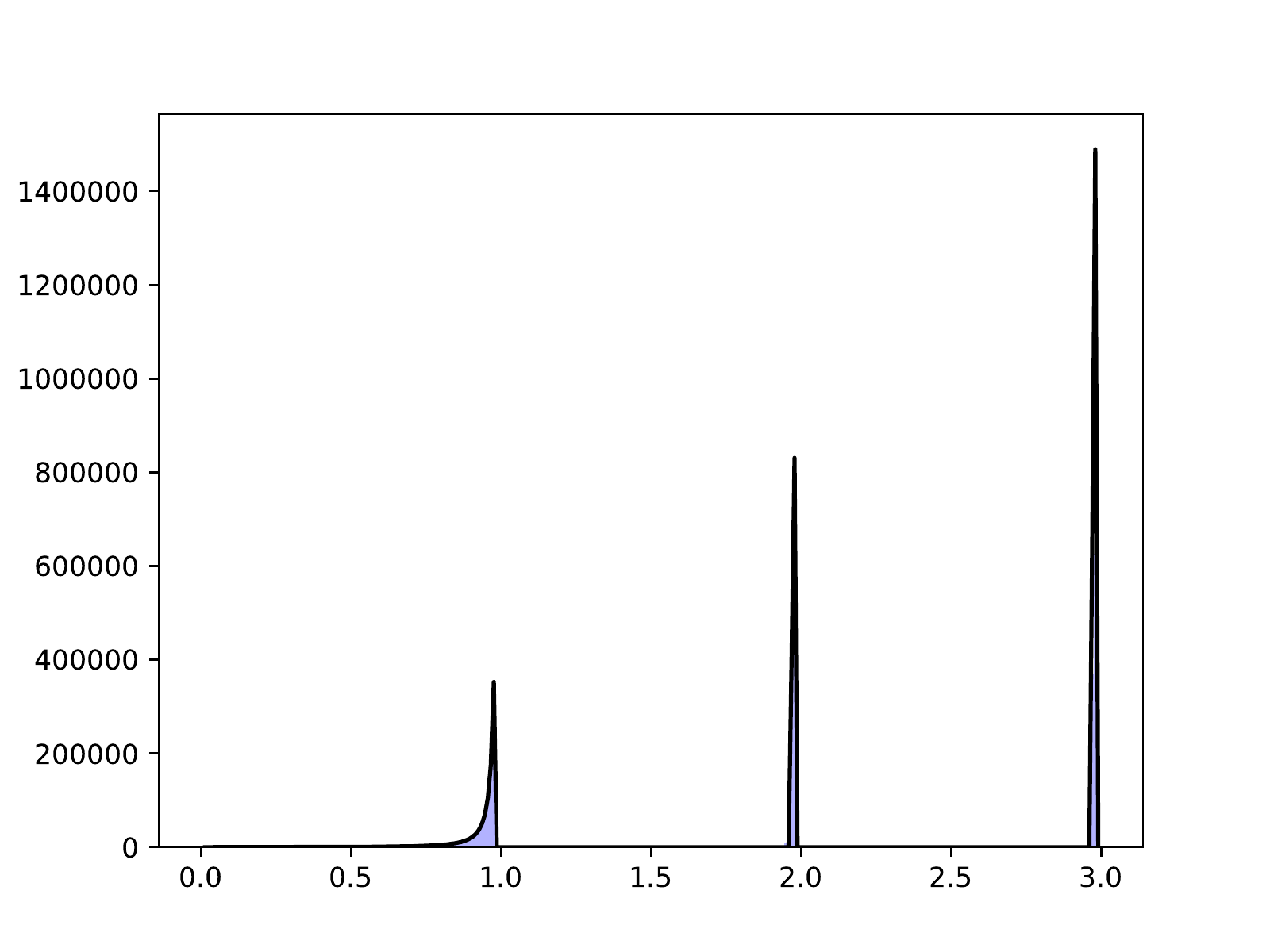}
\includegraphics[width=0.32\textwidth]{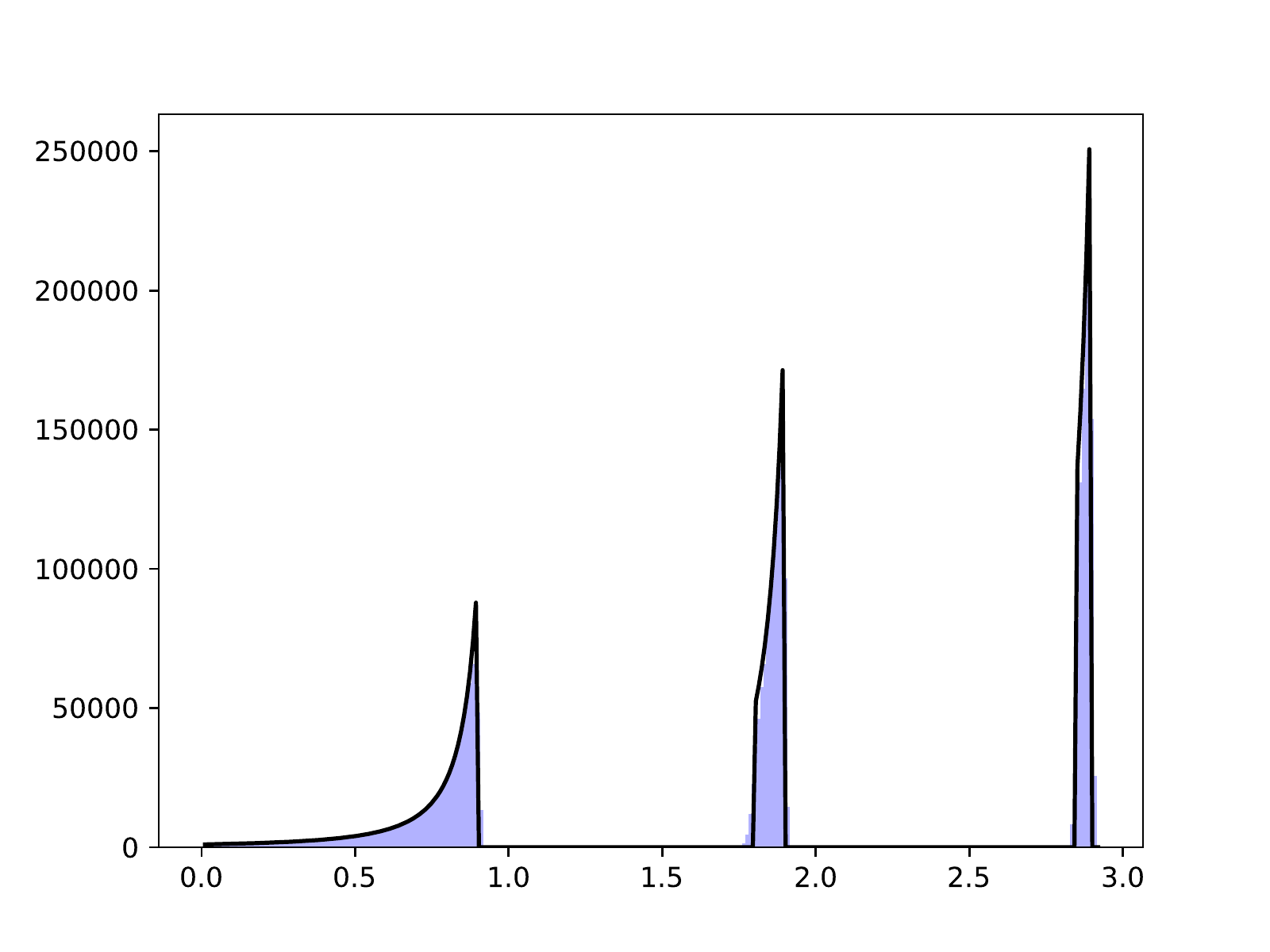}
\includegraphics[width=0.32\textwidth]{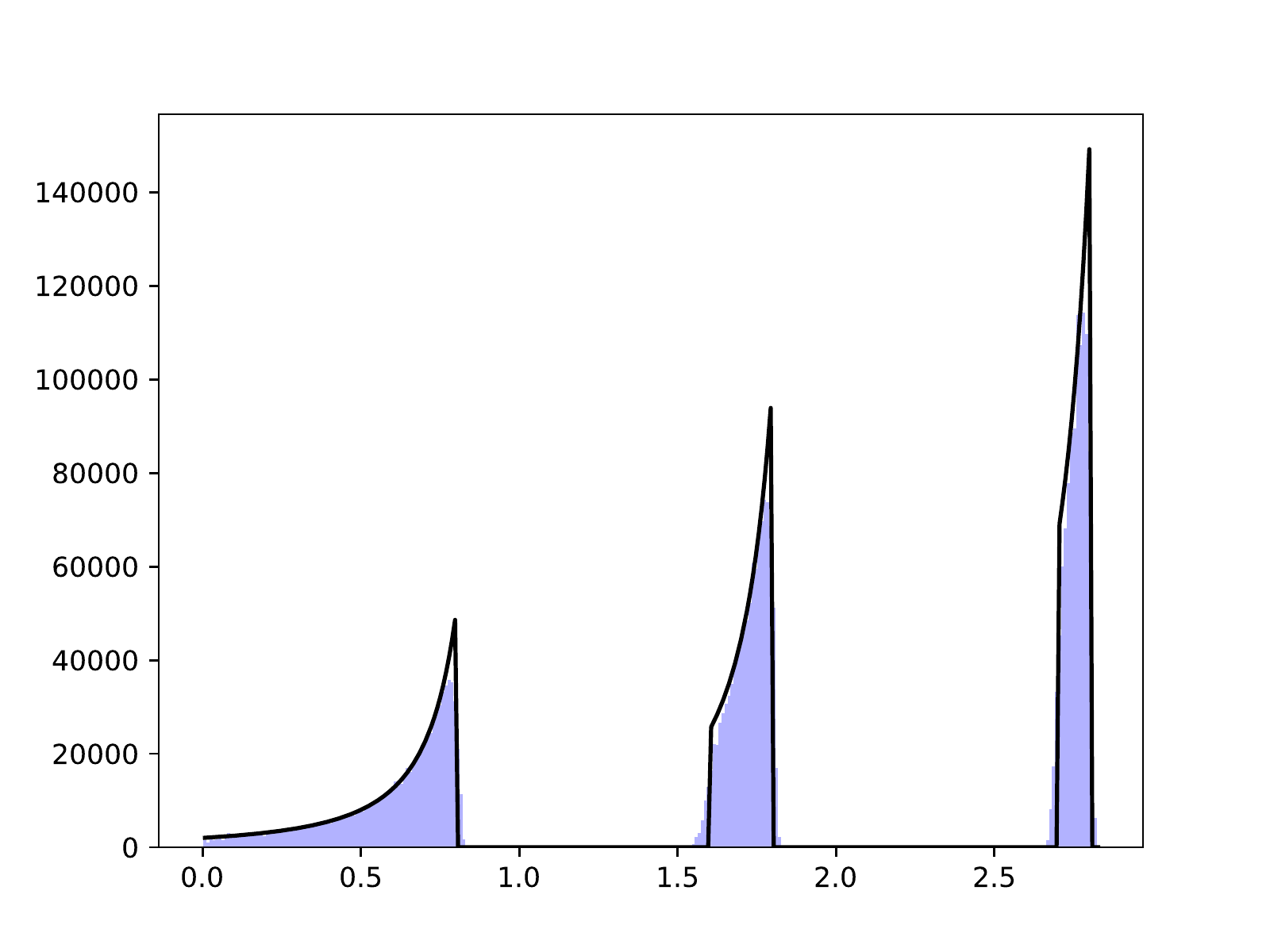}
\includegraphics[width=0.32\textwidth]{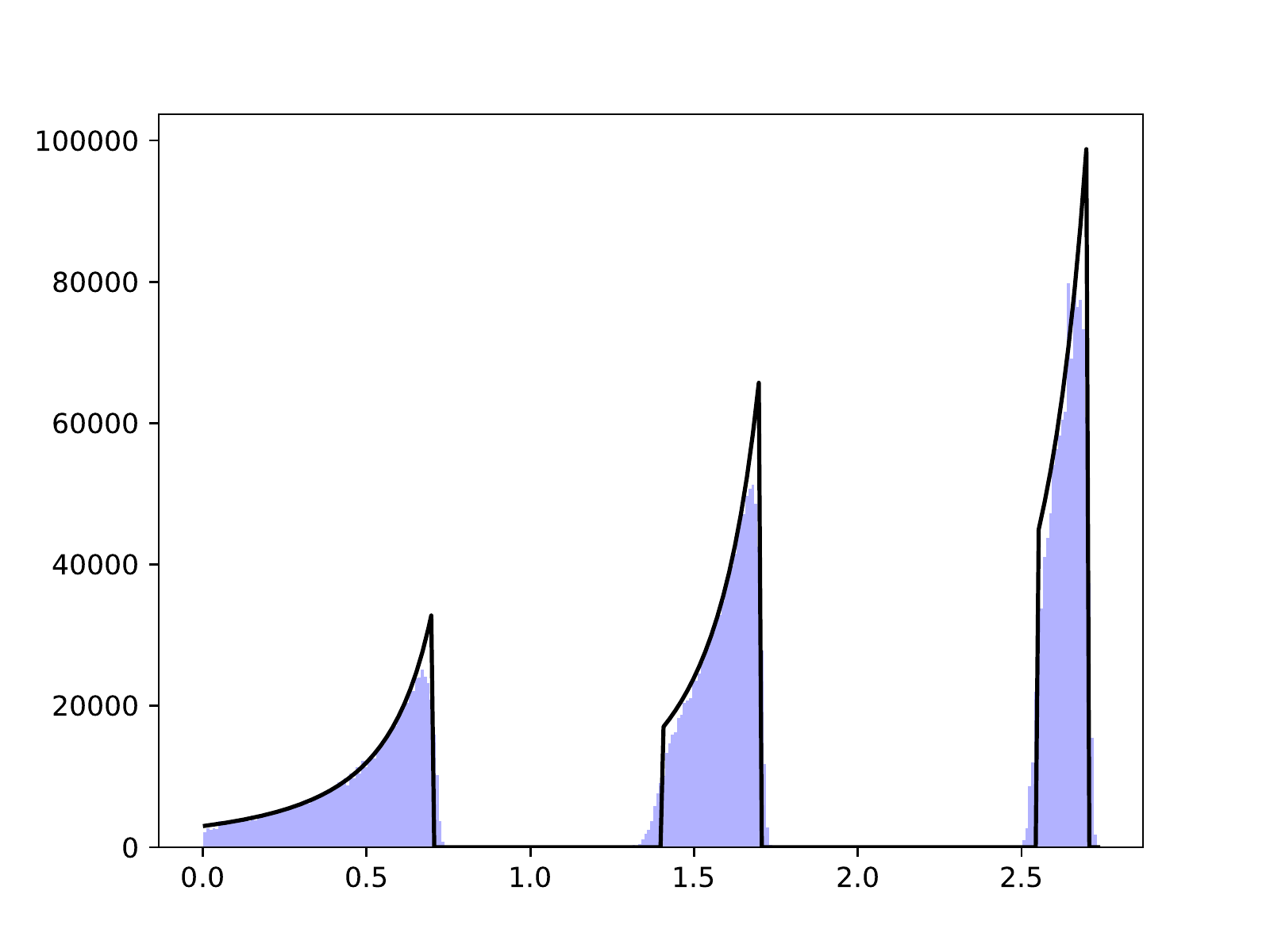}
\includegraphics[width=0.32\textwidth]{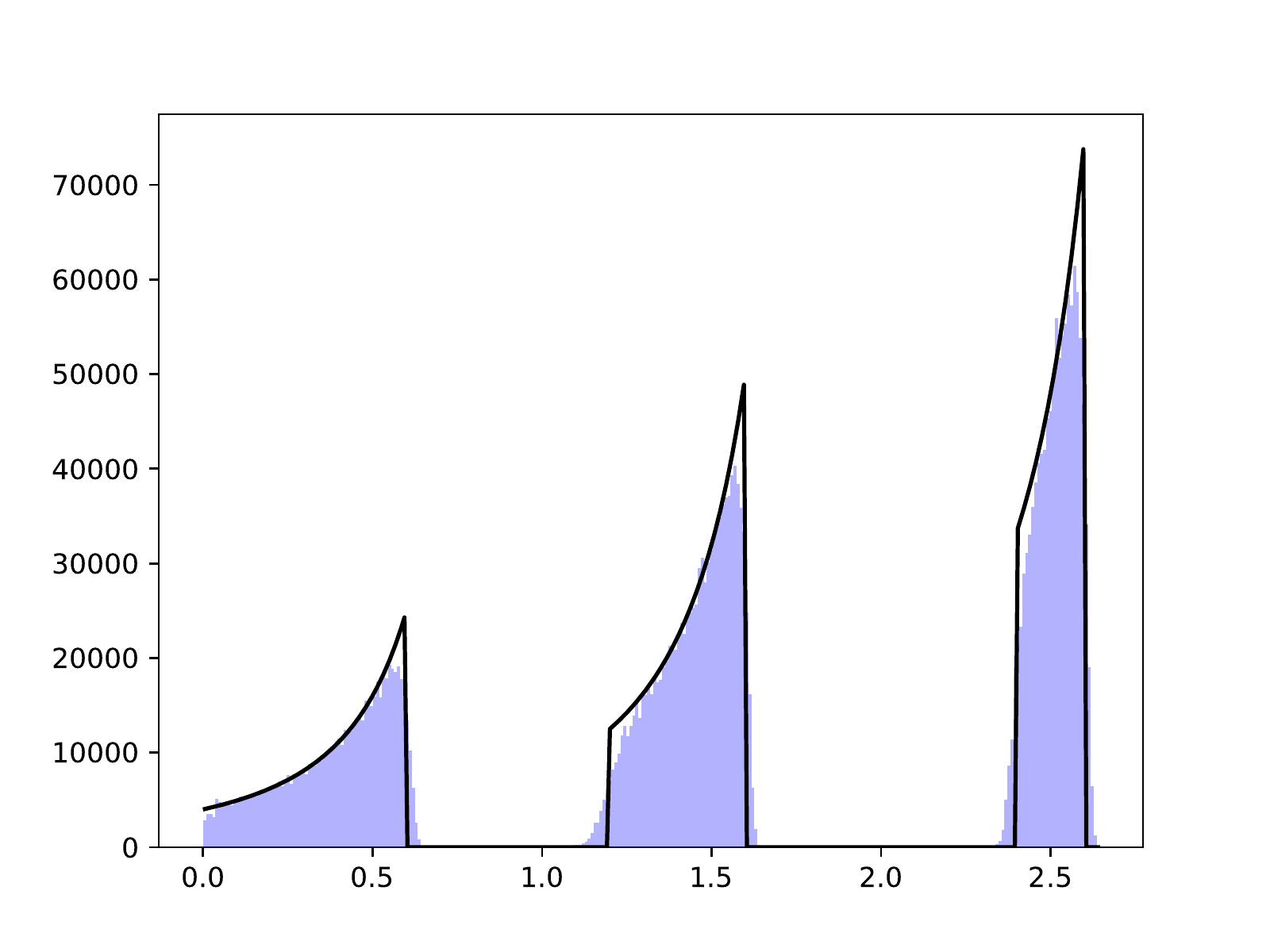}
\includegraphics[width=0.32\textwidth]{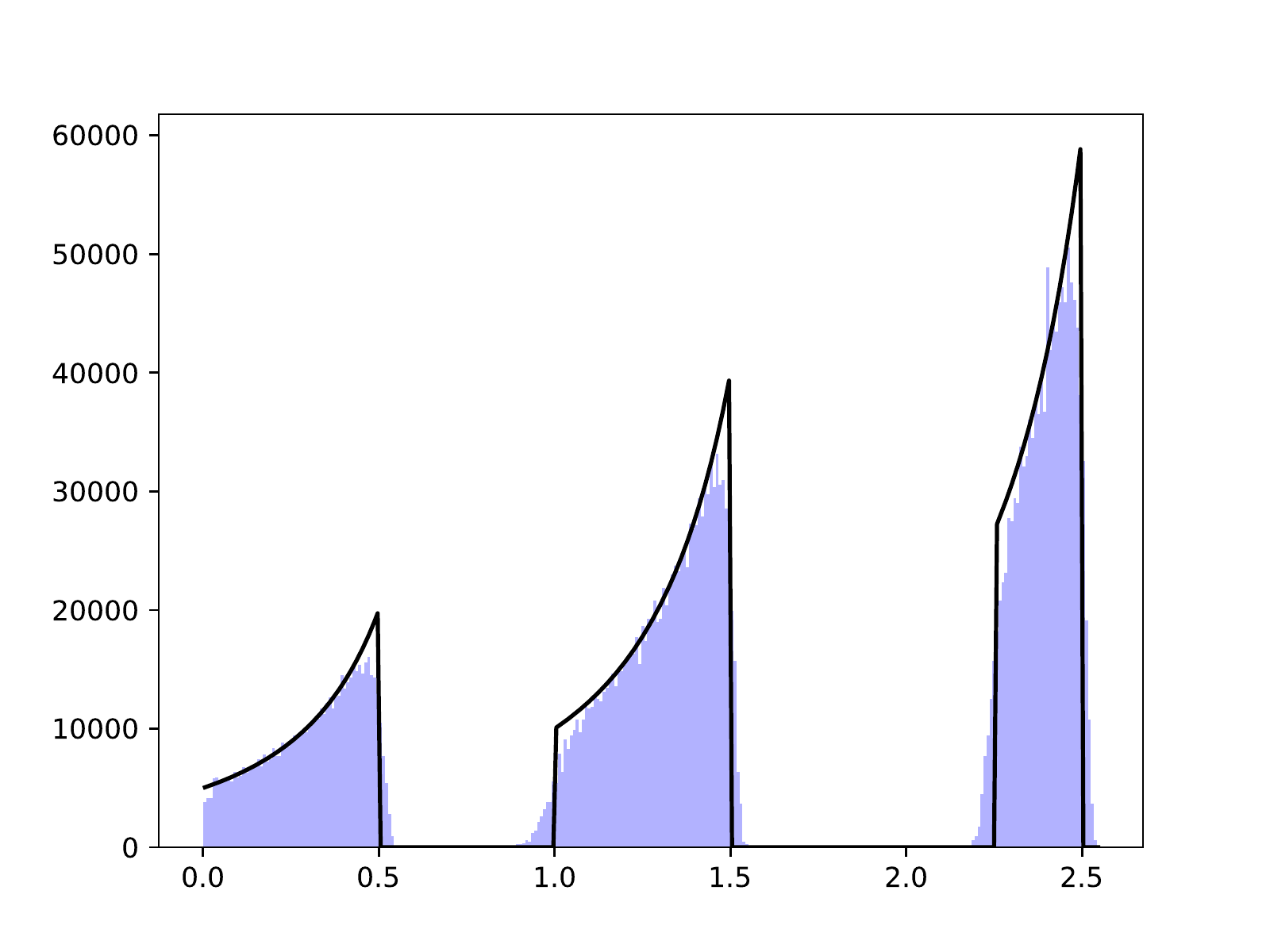}
\includegraphics[width=0.32\textwidth]{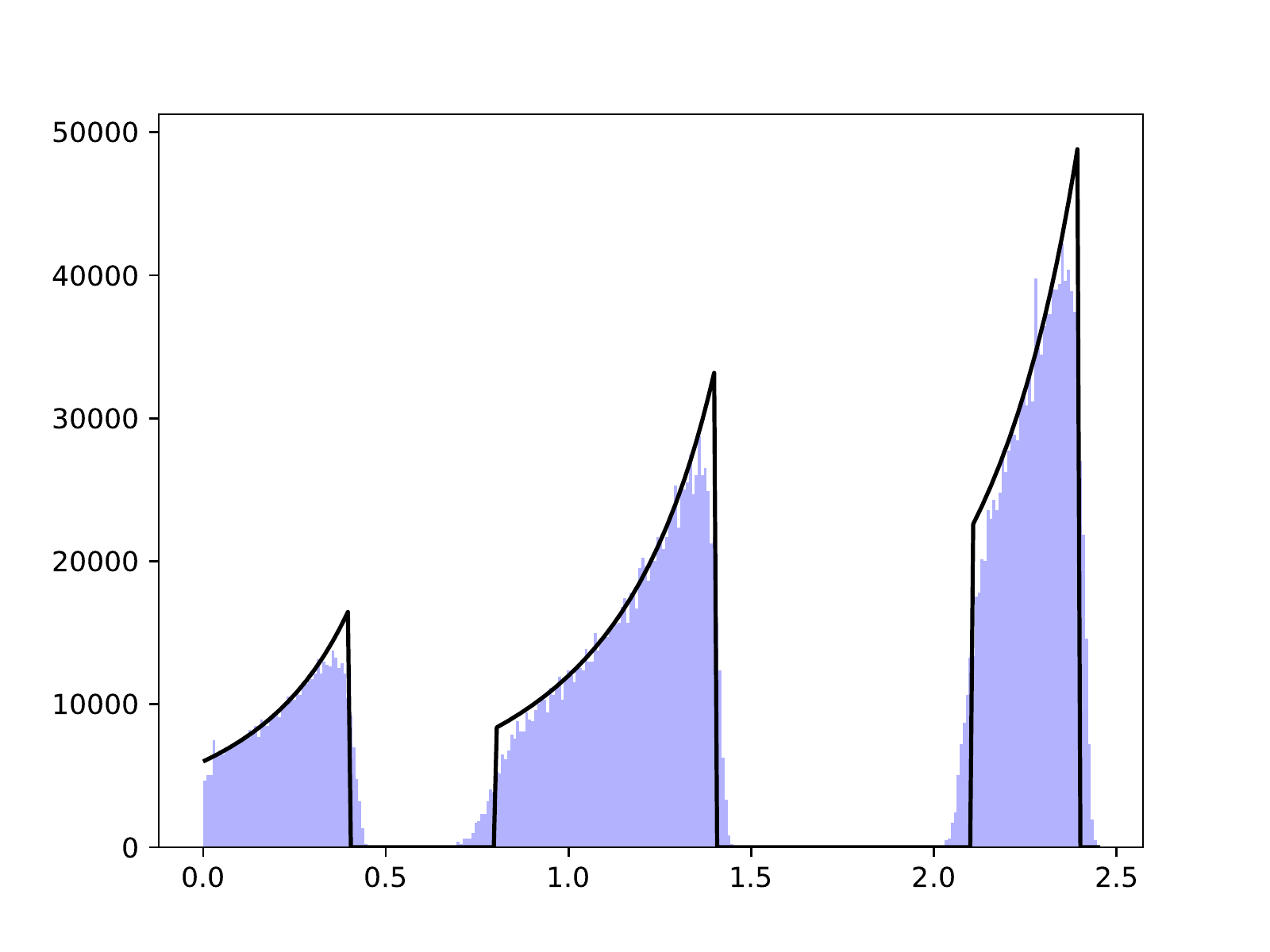}
\includegraphics[width=0.32\textwidth]{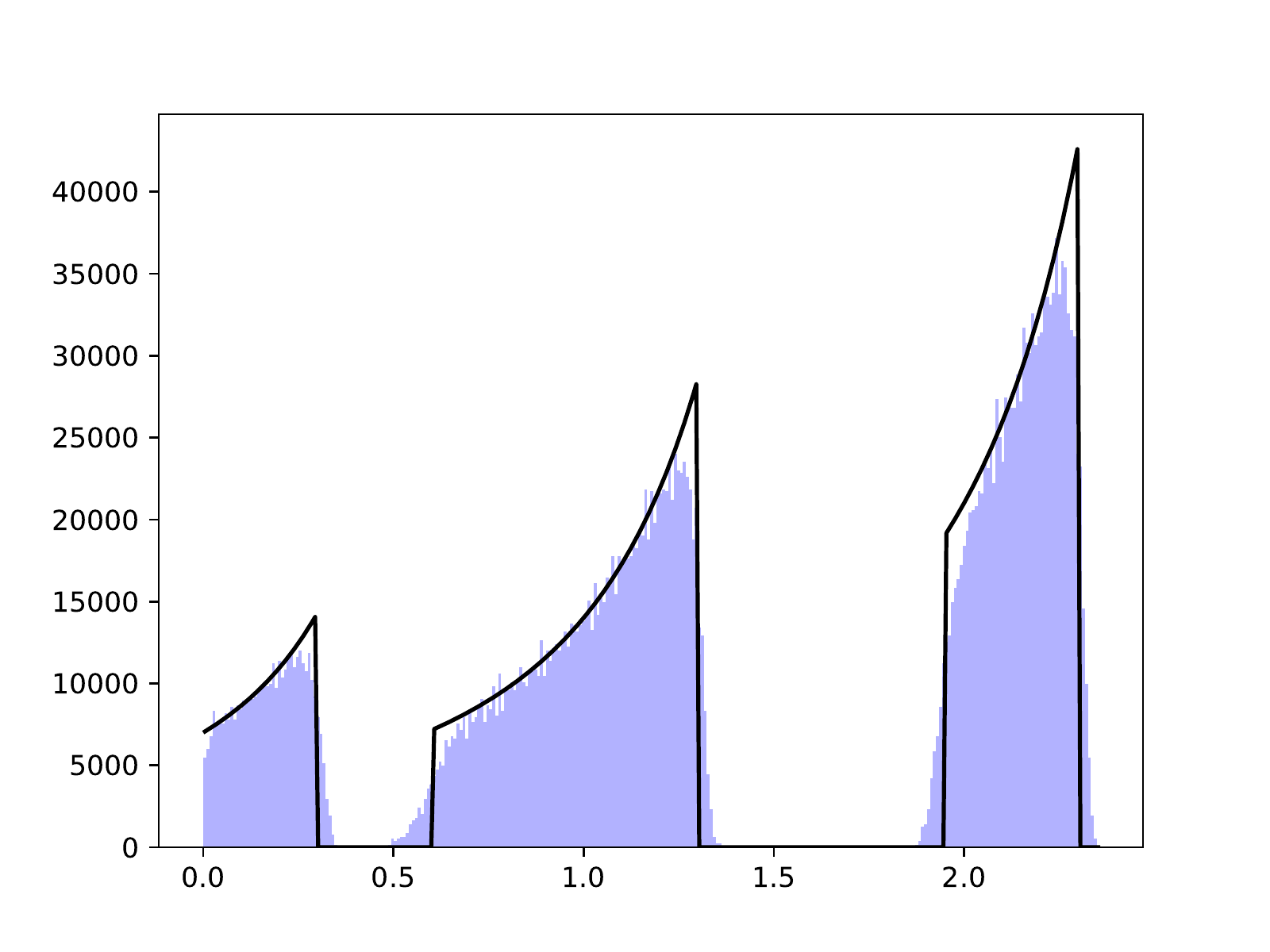}
\includegraphics[width=0.32\textwidth]{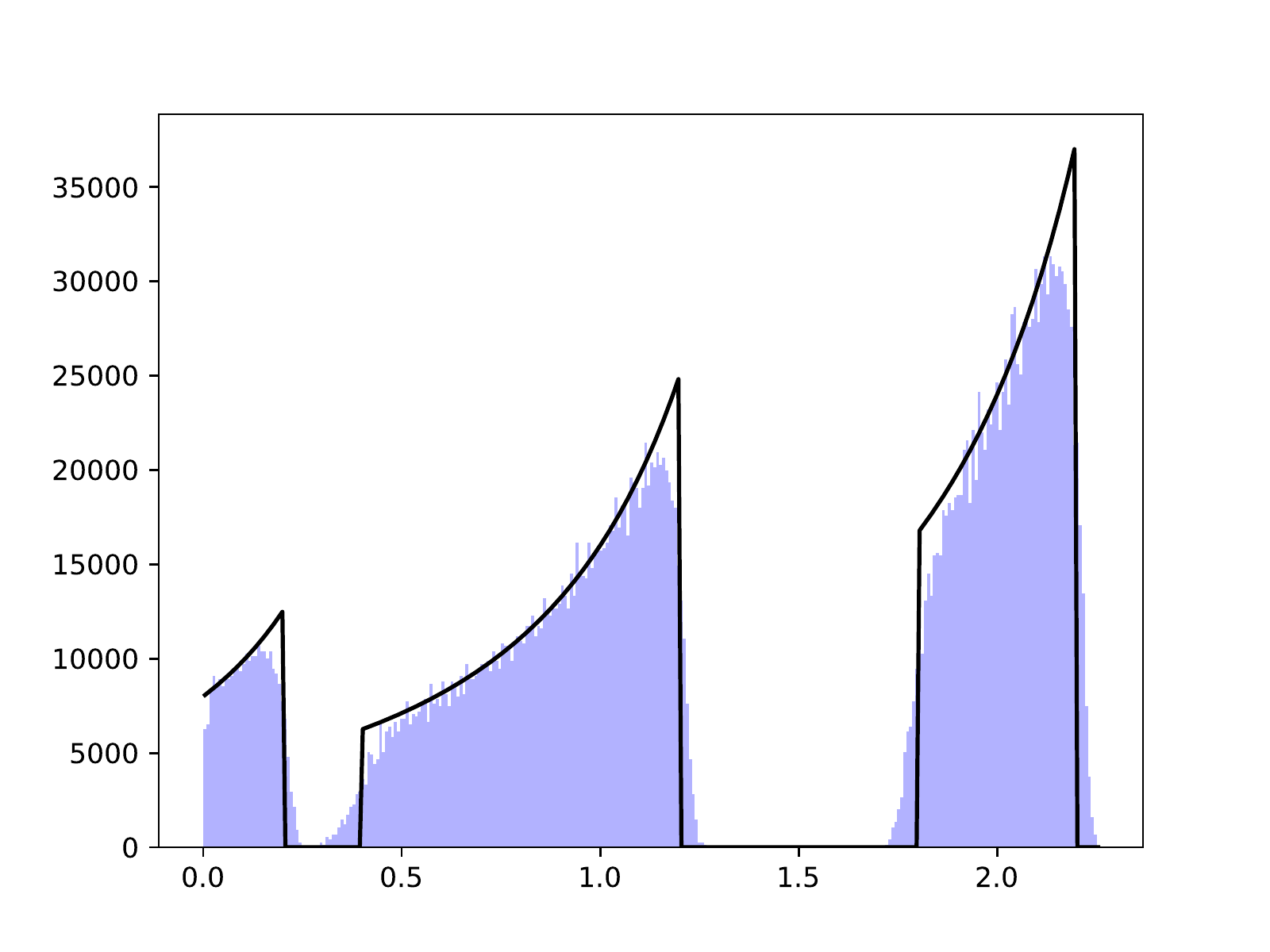}
\includegraphics[width=0.32\textwidth]{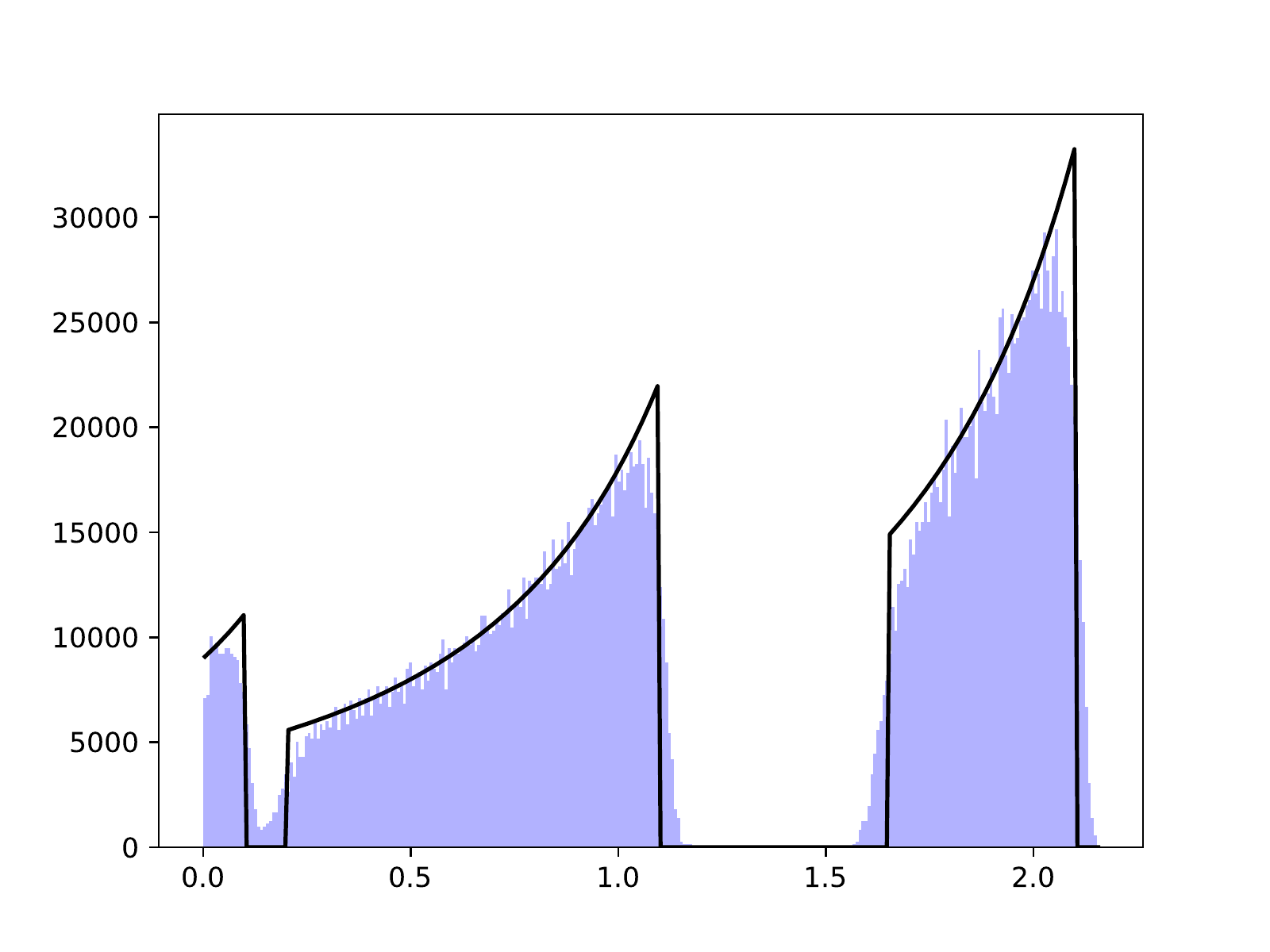}
\includegraphics[width=0.32\textwidth]{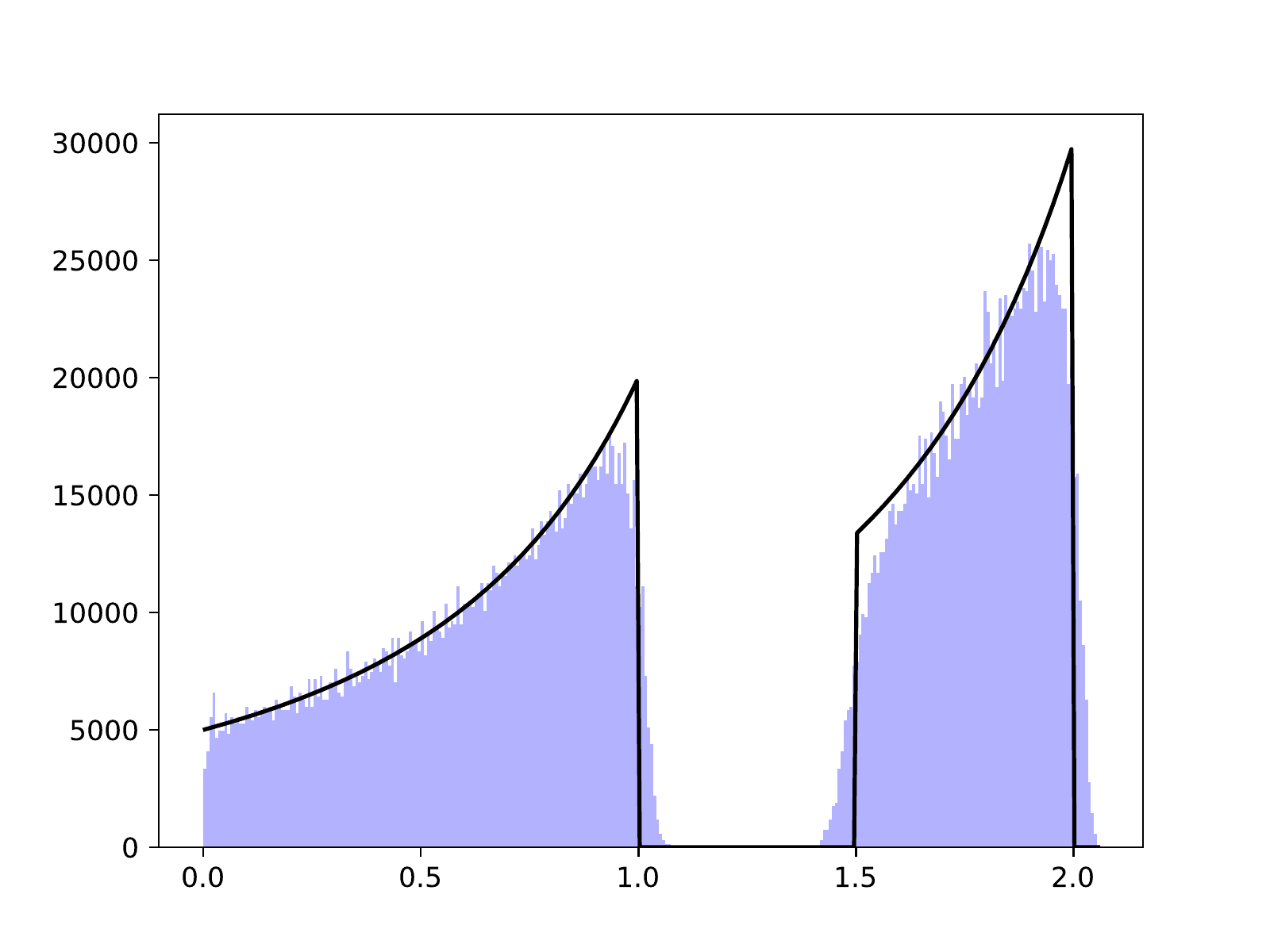}
\caption{
Histograms for an initial polynomial of degree $n=30000$ whose zeroes are i.i.d.\ on $k=3$ circles. The initial distribution of radial parts is given by~\eqref{eq:initial_dirac_deltas} with radii $r_1=1$, $r_2=2$, $r_3=3$ and weights $p_1=p_2=p_3 =1/3$.
The orders of the derivatives are $200, 1000, 2000, \ldots, 10000$.  The black curve shows the theoretical density given in~\eqref{eq:solution_dirac_deltas}.}
\label{pic:three_circles_zeroes_histogram}
\end{figure}
If the degree of the initial polynomial is large enough, this formula is in excellent agreement with the results of numerical simulation; see Figure~\ref{pic:three_circles_zeroes_histogram} showing histograms with the initial number of zeroes being $n=30000$. Figure~\ref{pic:three_circles_zeroes_plane} visualizes the same zeroes in the complex plane. It is interesting that the local structure of zeroes seems to be different for the model with independent roots and the model with independent coefficients, compare Figure~\ref{pic:three_circles_zeroes_plane} to the second row of Figure~\ref{pic:zeroes_kac_and_three_circles}, but the global behavior is the same.

The above results can be interpreted as follows. Take some $m\in \{1,\ldots,k\}$. The circles of zeroes with radii $r_1,\ldots, r_{m-1}$ present in the initial condition at time $0$ are killed by the repeated differentiation, if $t\in (P_{m-1}, P_m)$.  The circle of zeroes of radius $r_m$ turns into a two-dimensional distribution of zeroes on the disk with radius $\frac{r_m(P_m-t)}{P_m}$ with the explicit density of radial parts given by the first case of~\eqref{eq:solution_dirac_deltas}. For every $\ell\in \{m+1,\ldots,k\}$, the $\ell$-th circle of zeroes with radius $r_\ell$ turns into a two-dimensional distribution of zeroes on an annulus whose inner and outer radii are given by $\frac{r_\ell(P_{\ell-1}-t)}{P_{\ell-1}}$ and $\frac{r_\ell(P_\ell-t)}{P_\ell}$, respectively, with the density of radial parts given by the second case in~\eqref{eq:solution_dirac_deltas}. Finally, if $t$ approaches the value $P_m$ from the left, the disk of zeroes with radius $\frac{r_m(P_m-t)}{P_m}$ disappears and at the same time the annulus of zeroes with radii $\frac{r_{m+1}(P_{m}-t)}{P_{m}}$ and $\frac{r_{m+1}(P_{m+1}-t)}{P_{m+1}}$ turns into a disk of zeroes with radius $\frac{r_{m+1}(P_{m+1}-P_m)}{P_{m+1}}$.  This transition is visualized in the last two snapshots/histograms of Figures~\ref{pic:three_circles_zeroes_plane} and~\ref{pic:three_circles_zeroes_histogram}.

\subsection{Remarks on circles of zeroes and void annuli}
Let us make several general observations on the behavior of circles of zeroes and void annuli under repeated differentiation; see  Figures~\ref{pic:three_circles_zeroes_plane}, \ref{pic:zeroes_kac_and_three_circles}, \ref{pic:three_circles_zeroes_histogram} for simulations.

\vspace*{2mm}
\noindent
\textit{Circles of zeroes}.
As we already explained in Section~\ref{subsec:log_potential}, any constancy interval of the function $v'(x)$ corresponds to a circle of zeroes in the initial distribution $\mu_0$. On the other hand, we argued in the proof of Theorem~\ref{theo:PDE_solution} that for every $t\neq 0$  the function $x\mapsto v(x,t)$ is strictly convex. Thus, even though circles of zeroes may be present in the initial condition at time $t=0$, they are instantaneously  destroyed by repeated differentiation at any time $t\neq 0$.

\vspace*{2mm}
\noindent
\textit{Void annuli}.
On the other hand, any point $x_0\in (0,1)$ at which $v_-'(x_0) < v_+'(x_0)$ corresponds to a void annulus in $\mu_0$ with inner/outer radius equal to $\eee^{v_-'(x_0)}$, respectively $\eee^{v_+'(x_0)}$. The definition of $v(x,t)$ given in~\eqref{eq:v_x_t_def} implies that the derivative of $x\mapsto v(x,t)$ has a jump at $x=x_0-t$ provided that $t<x_0$. This means that  the void annulus persists as long as $t < x_0$ and its inner/outer radii are given by
$$
r_-(t) = \eee^{v_-'(x_0)} \cdot \frac{x_0-t}{x_0},
\qquad
r_+(t) = \eee^{v_+'(x_0)} \cdot \frac{x_0-t}{x_0},
\qquad
0 < t < x_0.
$$
Note that the quotient of the radii $r_+(t)/r_-(t)$ stays constant. At time $t=x_0$, the void annulus disappears and does not exist for $t>x_0$.

\vspace*{2mm}
\noindent
\textit{Void disk}.
Let us finally consider the situation with $v_+'(0)\neq -\infty$ which means that in $\mu_0$ there is a void disk of radius $\eee^{v_+'(0)}$. From the definition of $v(x,t)$ given in~\eqref{eq:v_x_t_def} it follows that for every $t\in (0,1)$, the right derivative of $x\mapsto v(x,t)$ at $x=0$ is $-\infty$, which means that the void disk is instantaneously destroyed at any time $t\neq 0$;  see the first two snapshots of Figure~\ref{pic:three_circles_zeroes_plane} for a visualization.

\subsection{Uniform distribution on the disk and its generalizations: Littlewood-Offord and Weyl polynomials}
Our next example is the following initial density of the complex roots:
\begin{equation}\label{eq:u_initial_weyl}
u(z,0)=\frac 1 {2\pi \alpha} |z|^{(1/\alpha)-2}\ind_{\{|z| \leq 1\}},
\qquad z\in \C.
\end{equation}
Here, $\alpha>0$ is a parameter. Note that the case $\alpha=1/2$ corresponds to the uniform distribution of the complex roots on the unit disk.
The density of the radial parts is given by
\begin{equation}\label{eq:psi_initial_weyl}
\psi(x,0)
=
2\pi x u(x,0)
=
\frac 1 {\alpha} x^{(1/\alpha)-1}\ind_{\{0<x<1\}},
\qquad x\geq 0.
\end{equation}
This asymptotic distribution of roots is realized by the following sequence of Littlewood-Offord random polynomials with independent coefficients:
\begin{equation}\label{eq:weyl_poly_def}
W_n(z) :=
\sum_{k=0}^n \frac{\xi_k}{(k!)^\alpha} (n^{\alpha} z)^k,
\qquad z\in \C;
\end{equation}
see Theorem~2.3 in~\cite{kabluchko_zaporozhets12a}.
The special case $\alpha=1/2$ is known under the name Weyl polynomials.

Let us now compute the evolution of the density of roots under repeated differentiation. To compute  $v(x)$ by means of~\eqref{eq:mu_0_I'} and~\eqref{eq:I_def}, we first  observe that
$$
\Psi(x,0) = \mu_0(\bD_x)
=
\int_0^x \psi(y,0)\dd y
=
\begin{cases}
 x^{1/\alpha}, &\text{if } x\in [0,1],\\
1,  &\text{if } x>1.
\end{cases}
$$
The function $I'(y) = J(y,0)$ is given by
$$
I'(y)
=
\Psi(\eee^y, 0)
=
\mu_0(\bD_{\eee^y})
=
\begin{cases}
\eee^{y/\alpha}, &\text{if } y\leq 0,\\
1,  &\text{if } y\geq0.
\end{cases}
$$
The inverse function is
$$
\partial_1 v(x,0)
=
\alpha \log x, \qquad x \in [0,1].
$$
Integrating, we obtain the following function $v(x) = v(x,0)$ corresponding to the initial conditions~\eqref{eq:u_initial_weyl} and~\eqref{eq:psi_initial_weyl}:
$$
v(x)
=
\alpha (x\log x - x), \qquad  x\in [0,1].
$$
Alternatively, one could compute $v(x)$ from the asymptotics of the coefficients of the polynomials~\eqref{eq:weyl_poly_def} using the Stirling formula;  see~\cite[p.~1385]{kabluchko_zaporozhets12a}.

Now, we can compute $\psi(x,t)$ using the recipe described in Section~\ref{subsec:log_potential}. By~\eqref{eq:v_x_t_def}, for arbitrary $t\in [0,1]$ we have
$$
v(x,t) = \alpha ((x+t)\log (x+t) - (x+t)) - (x+t) \log (x+t) + x\log x, \qquad 0 \leq x \leq 1-t.
$$
The derivative in $x$ is given by
$$
\partial_1 v(x,t) = (\alpha-1) \log (x+t) + \log x,  \qquad 0 \leq x \leq 1-t.
$$
Note that on the interval $x\in (0, 1-t]$, the function $\partial_1 v(x,t)$ is monotone increasing and its image is the interval $(-\infty, \log (1-t)]$. For every fixed $t\in [0,1)$,  the inverse function denoted by $y\mapsto J(y,t)$ satisfies
$$
(\alpha-1) \log (J(y,t) + t) + \log J(y,t) = y, \qquad  -\infty < y  < \log (1-t).
$$
Taking $y= \log x$, we obtain the following implicit equation for $\Psi(x,t) = J(\log x, t)$:
\begin{equation}\label{eq:Psi_Weyl_implicit}
(\alpha-1) \log (\Psi(x,t) + t) + \log \Psi(x,t) = \log x, \qquad  0 < x  < 1-t.
\end{equation}

Now we consider several special cases in which the inverse function can be expressed in a closed form.

\begin{figure}[!tbp]
\includegraphics[width=0.24\textwidth]{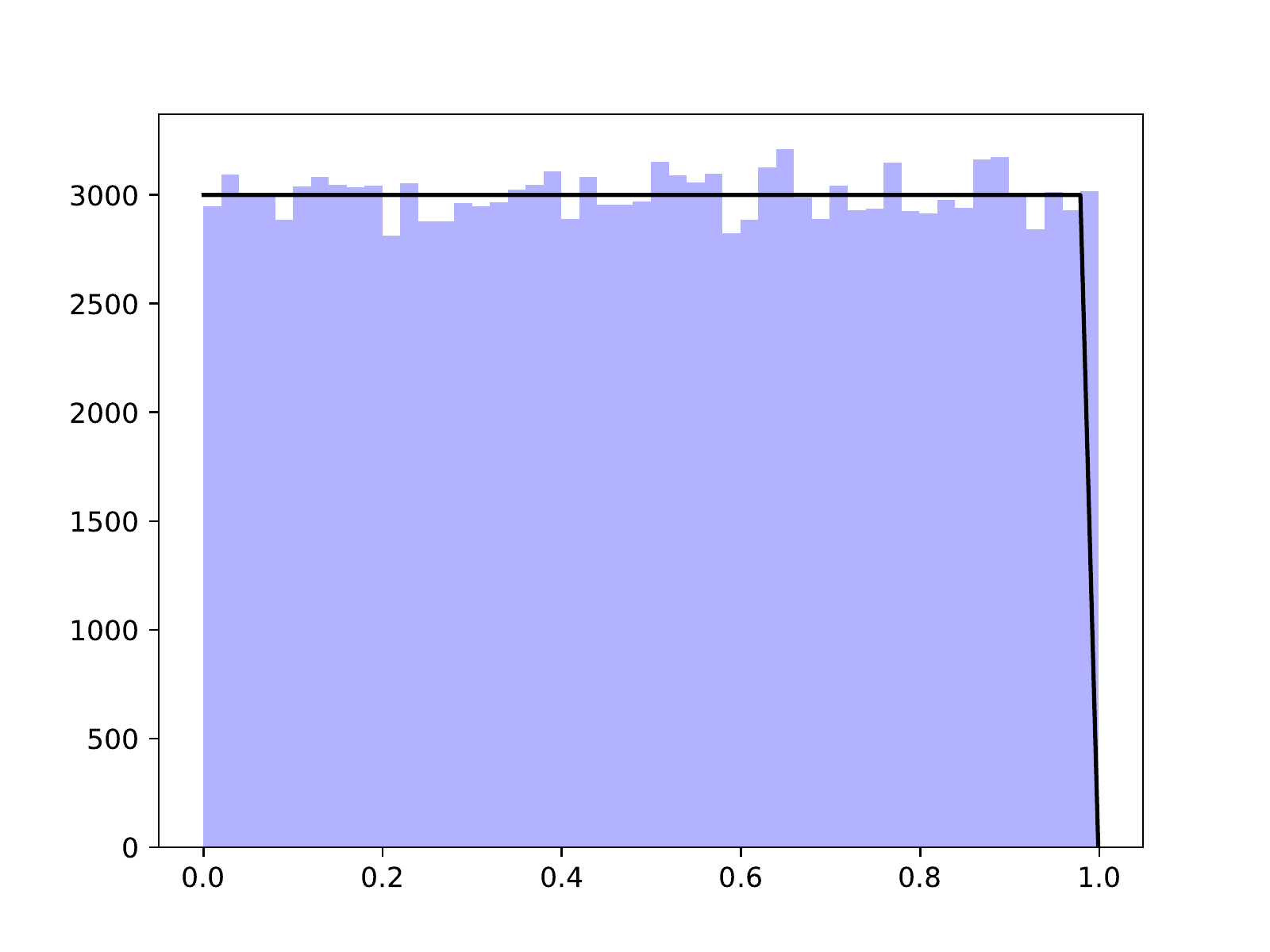}
\includegraphics[width=0.24\textwidth]{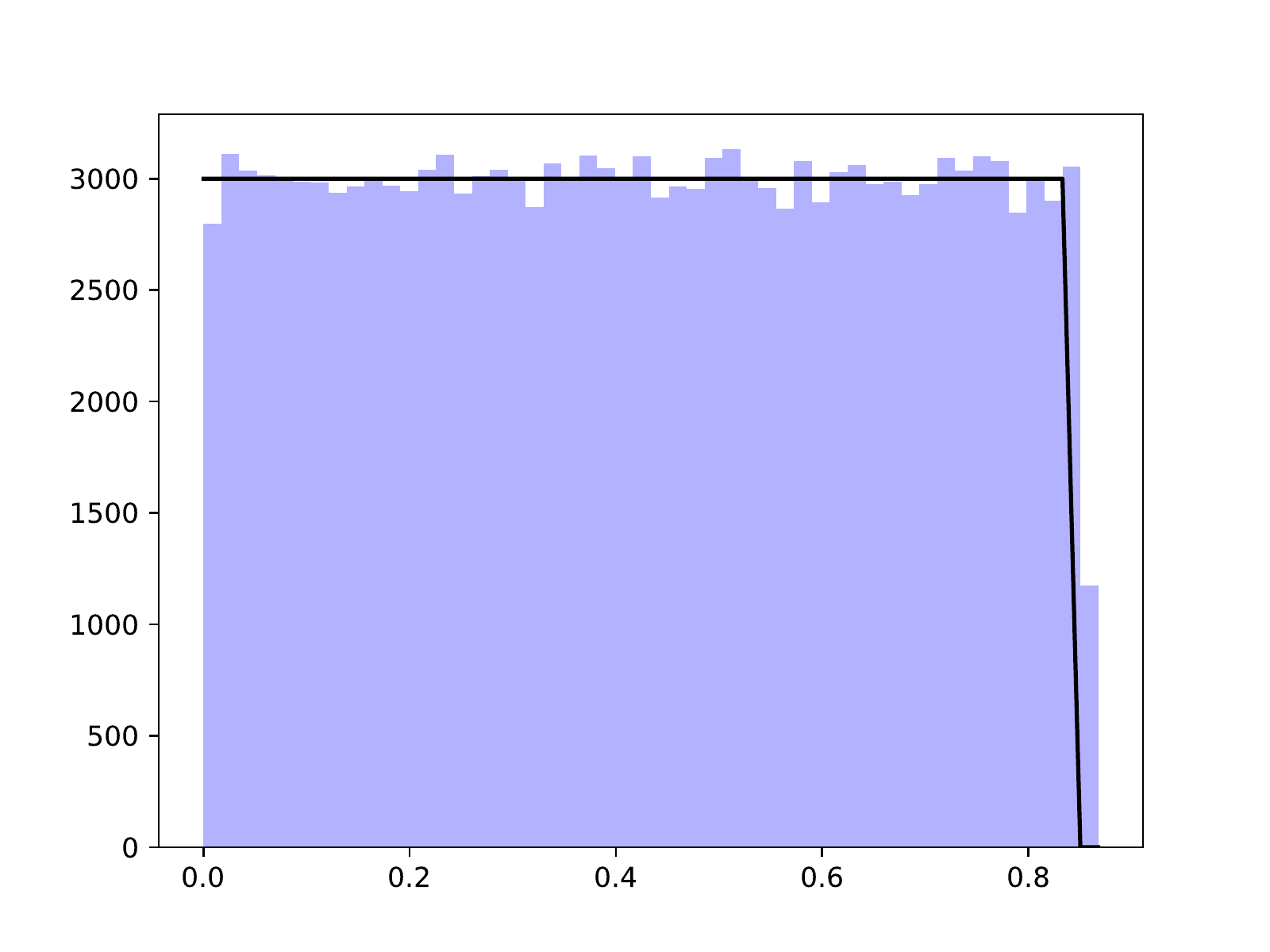}
\includegraphics[width=0.24\textwidth]{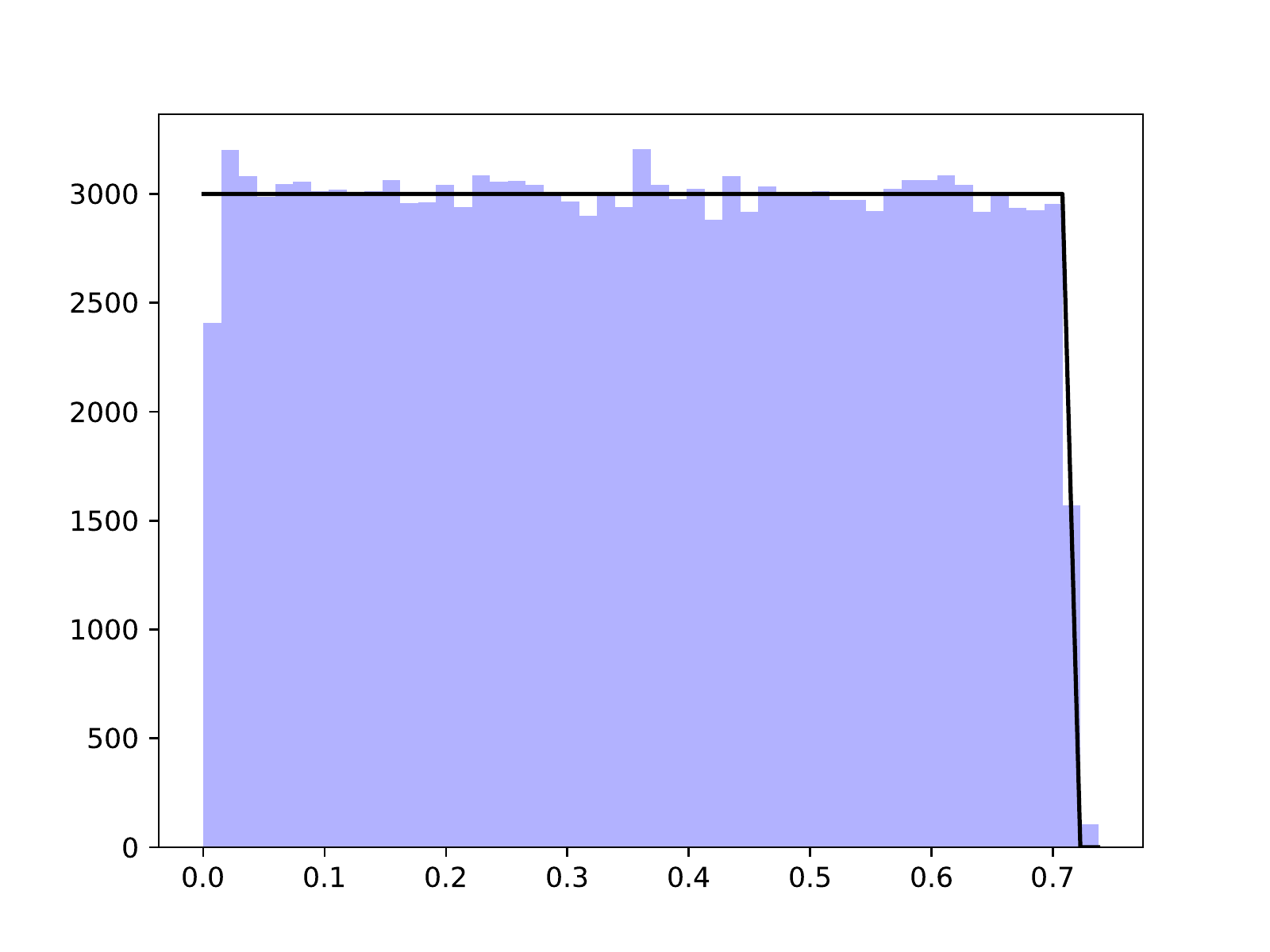}
\includegraphics[width=0.24\textwidth]{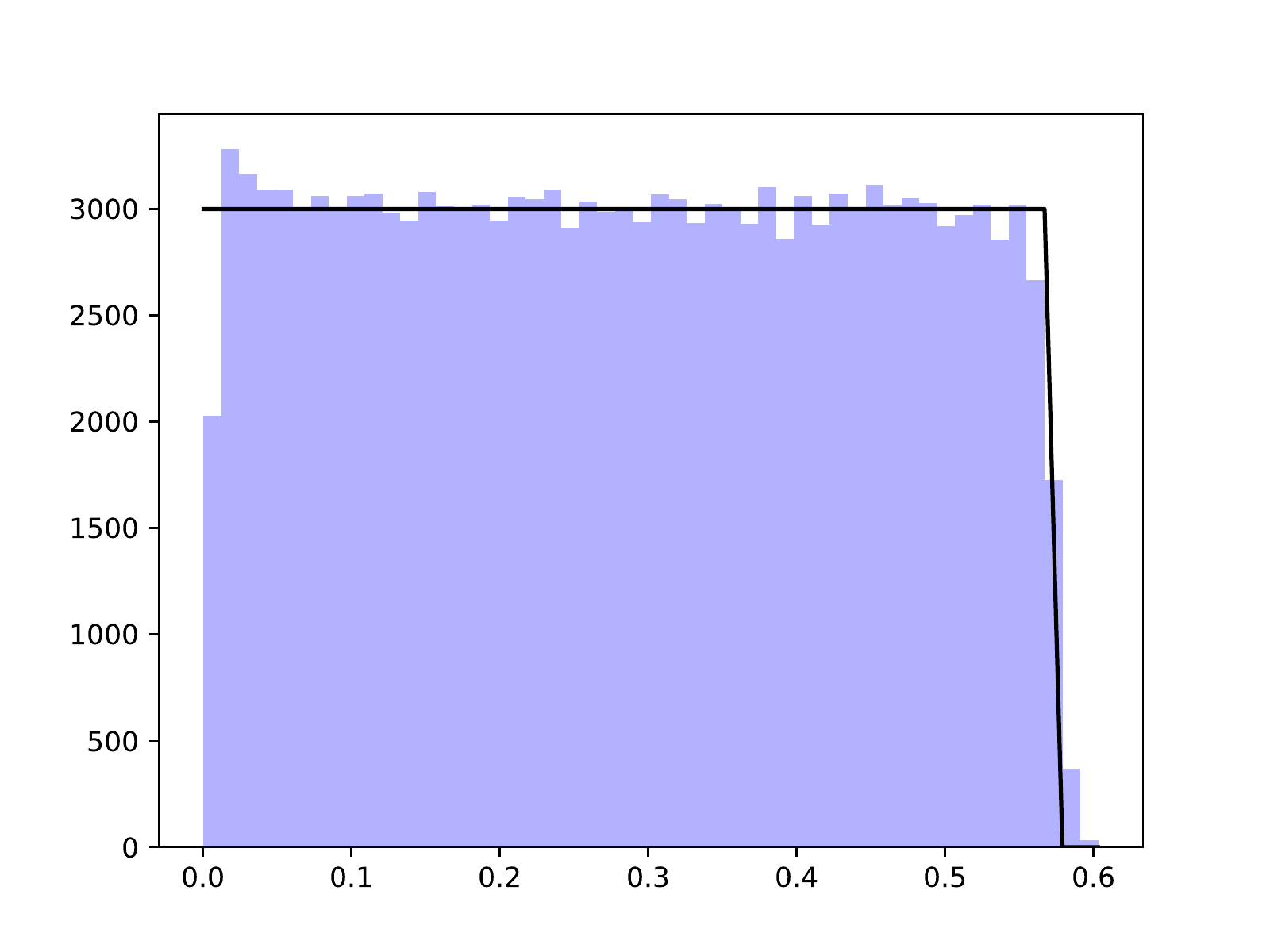}
\includegraphics[width=0.24\textwidth]{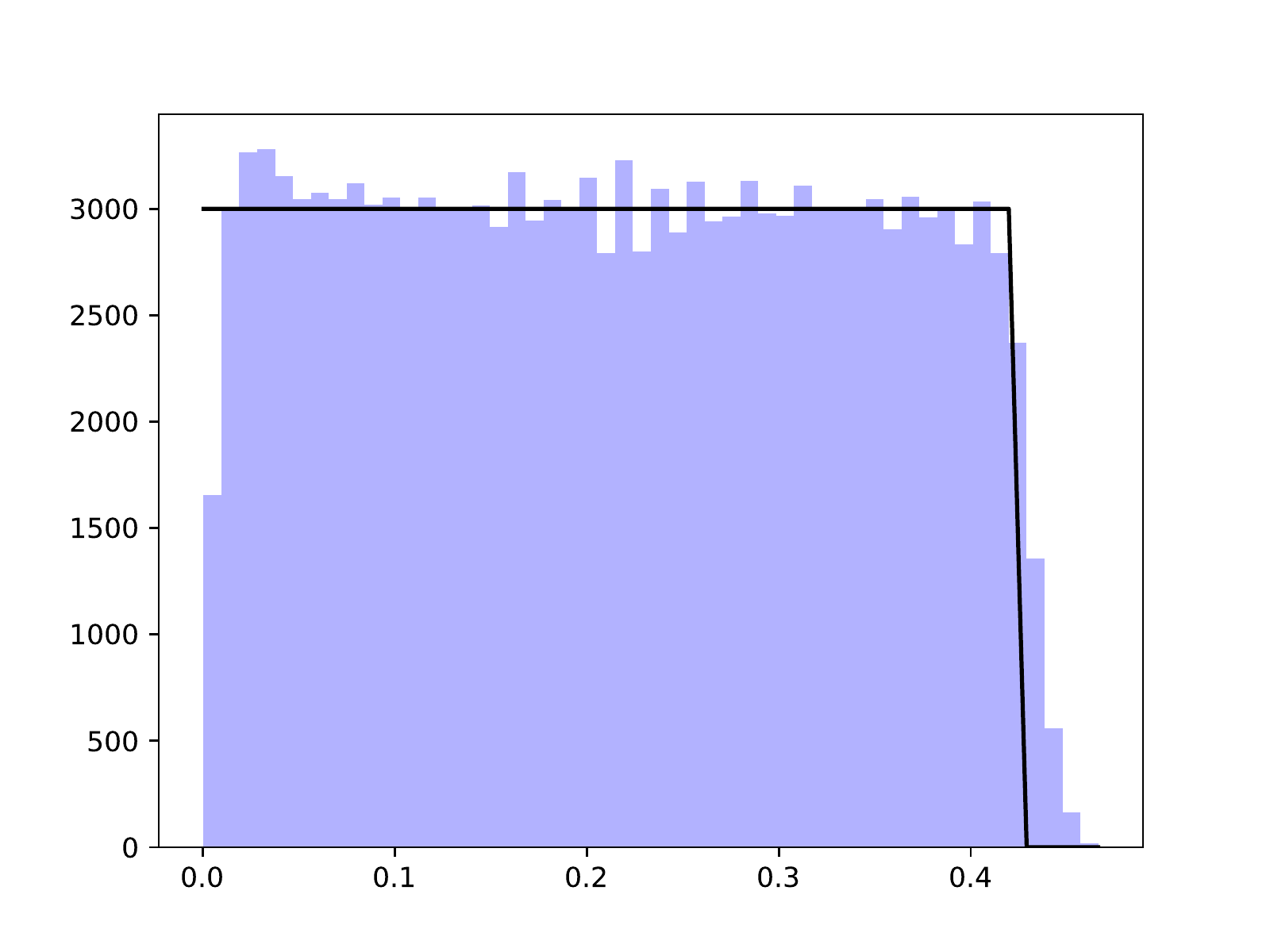}
\includegraphics[width=0.24\textwidth]{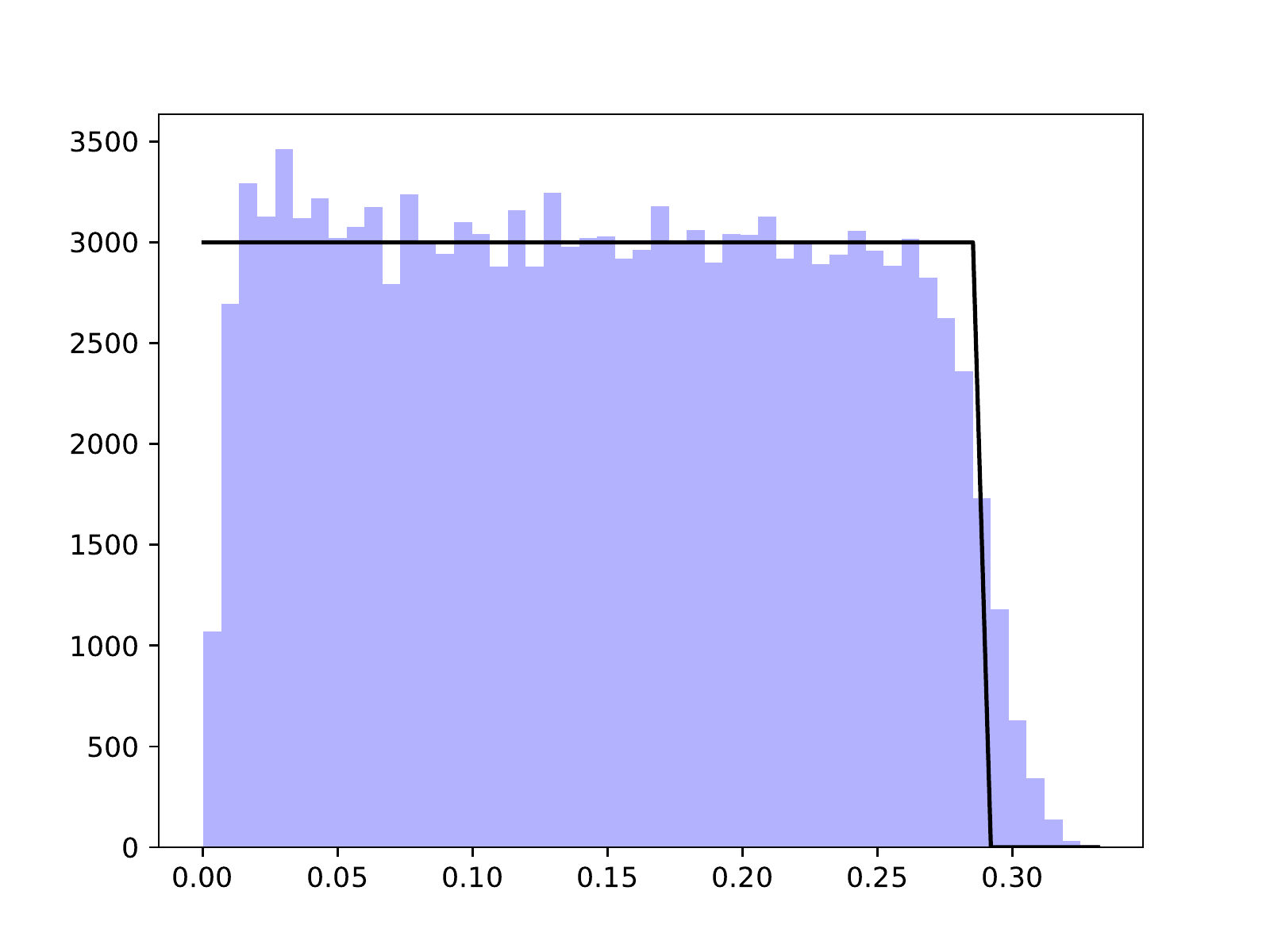}
\includegraphics[width=0.24\textwidth]{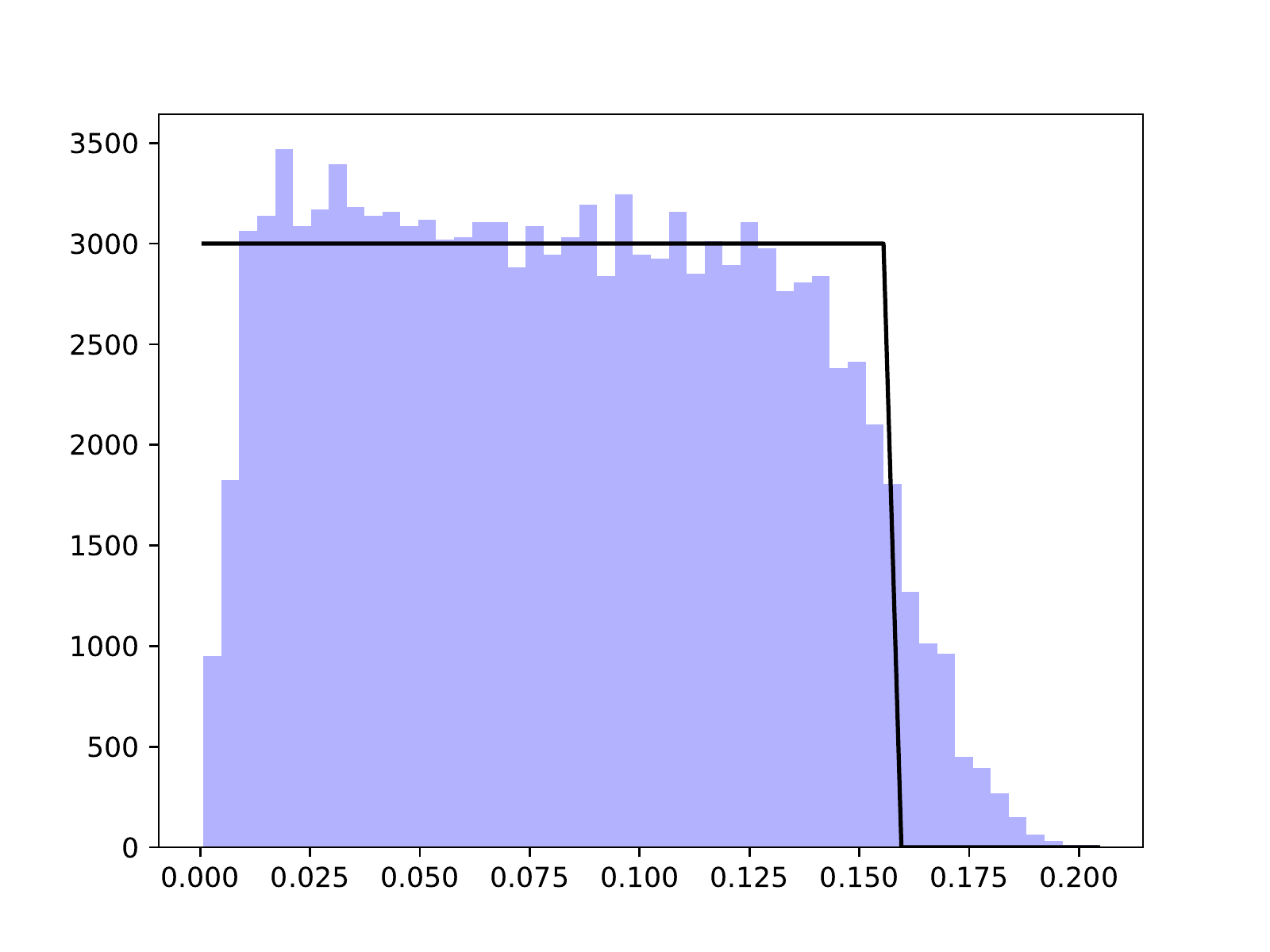}
\includegraphics[width=0.24\textwidth]{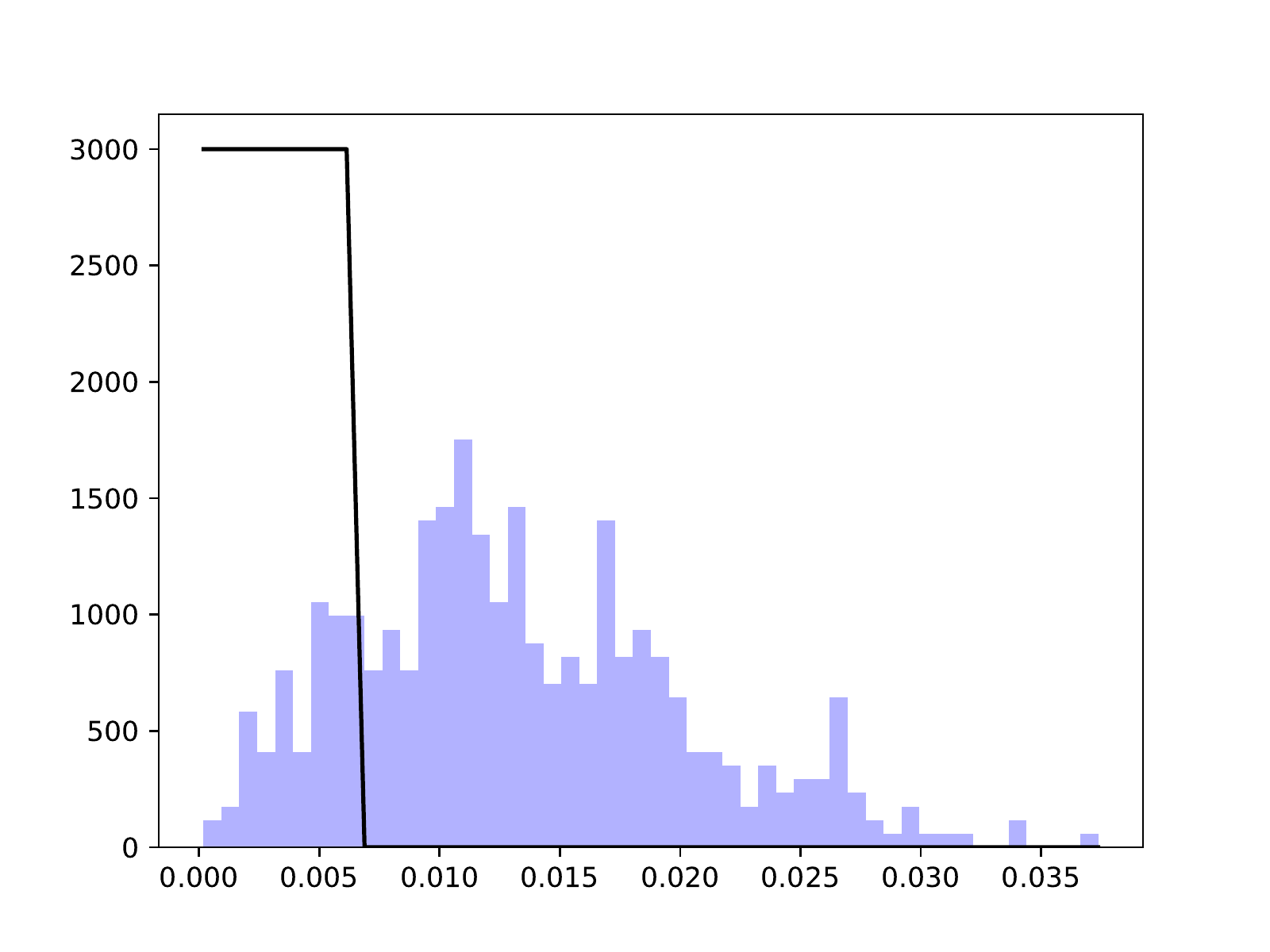}
\caption{
Histograms of the radial parts for the initial density $\psi(x,0)= 1$, $0<x<1$.
The zeroes of the initial polynomial are i.i.d.\ and the degree is $n=3000$.
The orders of the derivatives are $4 + 425 k$ with $k\in \{0,\ldots,7\}$.
The black curve shows the theoretical density given in~\eqref{eq:evolution_weyl_alpha_1}.
In these and similar histograms, an average over $14-20$ independent realizations of the polynomial is shown.
}
\label{pic:unif_zeroes_histogram}
\end{figure}

\vspace*{2mm}
\noindent
\textsc{Case $\alpha=1$} (Uniformly distributed radial parts).  In this case, the radial parts at time $t=0$ are uniformly distributed on the interval $[0,1]$, that is the initial condition is
\begin{equation}\label{eq:initial_cond_weyl_alpha_1}
\psi(x,0) = \ind_{\{0<x<1\}}, \qquad x\geq 0.
\end{equation}
The implicit equation~\eqref{eq:Psi_Weyl_implicit} is solved by
$$
\Psi(x,t) = x, \qquad 0 < x <  1-t.
$$
The density of the absolute values of the roots is thus
\begin{equation}\label{eq:evolution_weyl_alpha_1}
\psi(x,t) = \ind_{\{0<x<1-t\}}, \qquad x\geq 0.
\end{equation}
So, the absolute values of the roots of the $[tn]$-th derivative of $P_n$ are asymptotically  uniformly distributed on the interval $[0,1-t]$. This agrees with the results of the numerical simulation presented on Figure~\ref{pic:unif_zeroes_histogram}.
Note that the roots themselves are distributed according to a rotationally invariant subprobability measure on the disk of radius $1-t$ centered at the origin with the Lebesgue density
$$
u(z,t)
= \frac{1}{2\pi |z|}\ind_{\{|z|\leq 1-t\}}, \qquad z\in \C.
$$
This special solution has been already mentioned in~\cite{orourke_steinerberger_nonlocal}. It is interesting that the shape of the density does not change, just its support shrinks with growing $t$.  To explain this  phenomenon, consider the random polynomials
\begin{equation}\label{eq:poly_exp_Q_n}
Q_n(z):= \sum_{k=0}^n \xi_k \frac {(nz)^k}{k!},
\end{equation}
where $\xi_0,\xi_1,\ldots$ are i.i.d.\ random variables with values in $\C$ and such that $\E \log (1+|\xi_0|) <\infty$. In the large degree limit $n\to\infty$,  the radial parts of the roots become uniformly distributed on the interval $[0,1]$, see Theorem~2.3 in~\cite{kabluchko_zaporozhets12a}, meaning that these polynomials correspond to the initial condition~\eqref{eq:initial_cond_weyl_alpha_1}. Under repeated differentiation, the polynomials behave as follows:
$$
Q_n^{(m)}(z)
=
n^m \sum_{k=m}^n \xi_k \frac {(nz)^{k-m}}{(k-m)!}
=
n^m \sum_{k=0}^{n-m} \xi_{k+m} \frac {(nz)^{k}}{k!},
$$
which has the same distribution as $Q_{n-m}(zn/(n-m))$, up to a constant factor. Thus, at time $t$ (meaning that $m=[nt]$), the density of the radial parts is given by~\eqref{eq:evolution_weyl_alpha_1}.

\begin{figure}[!tbp]
\includegraphics[width=0.24\textwidth]{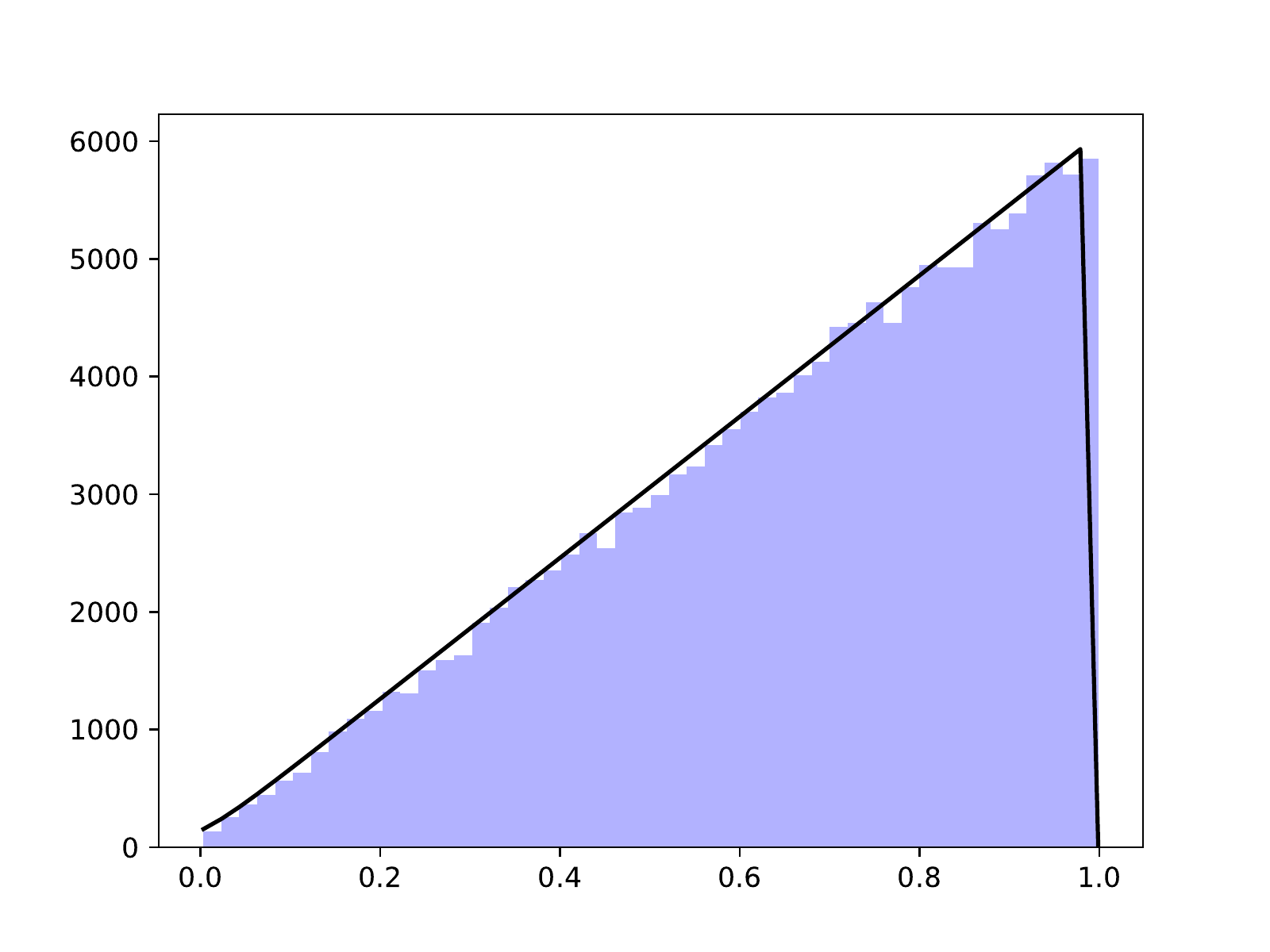}
\includegraphics[width=0.24\textwidth]{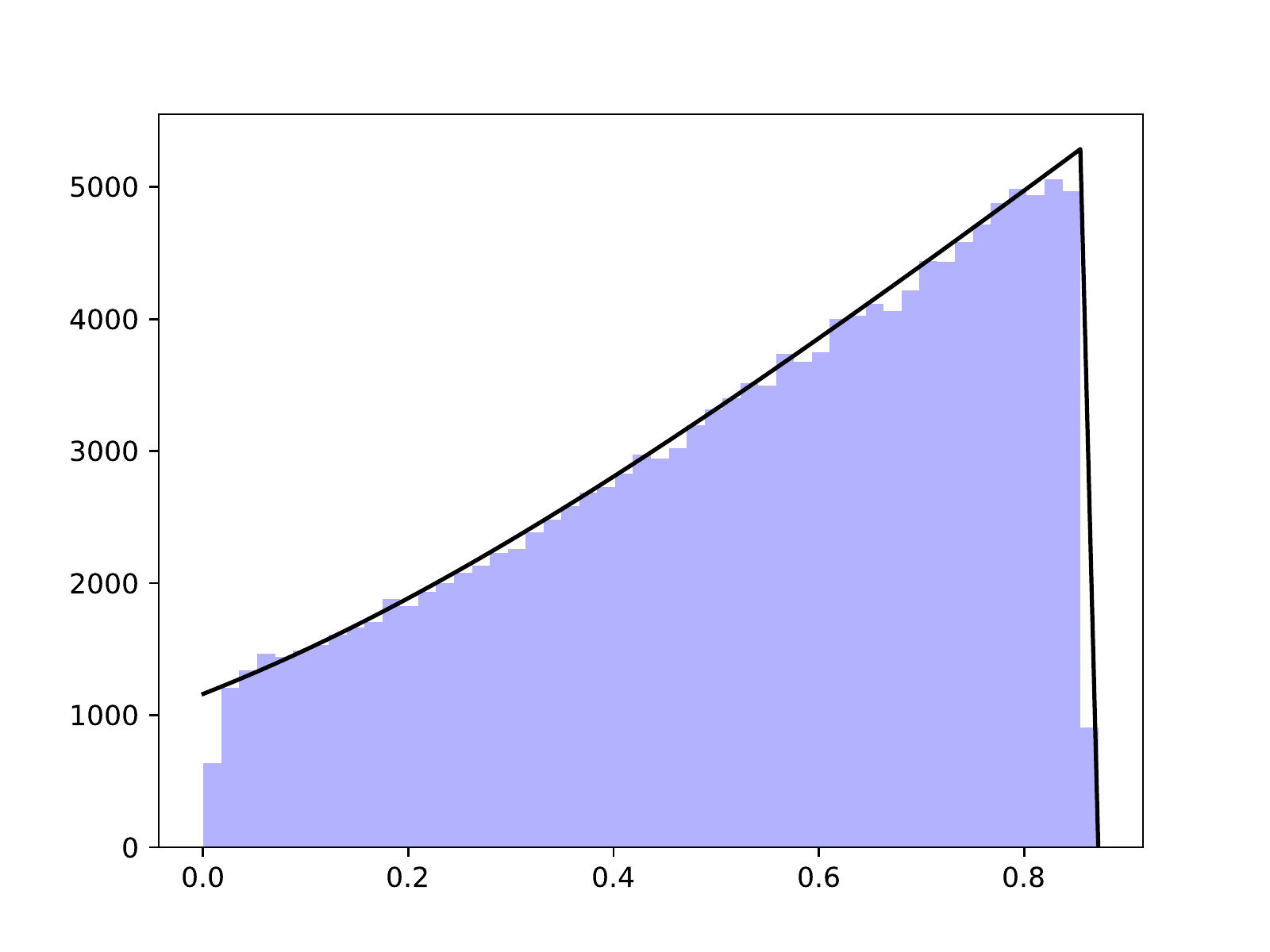}
\includegraphics[width=0.24\textwidth]{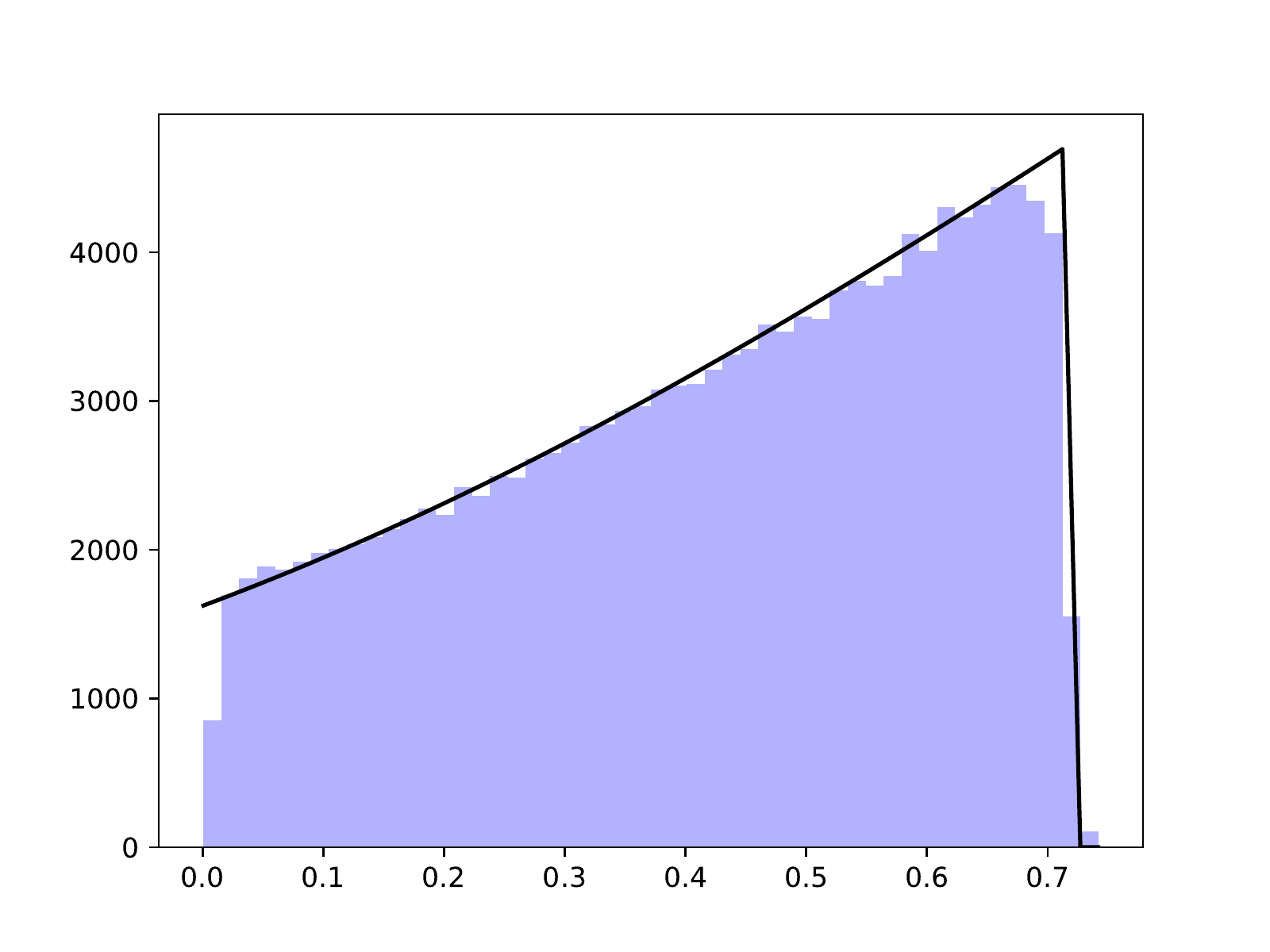}
\includegraphics[width=0.24\textwidth]{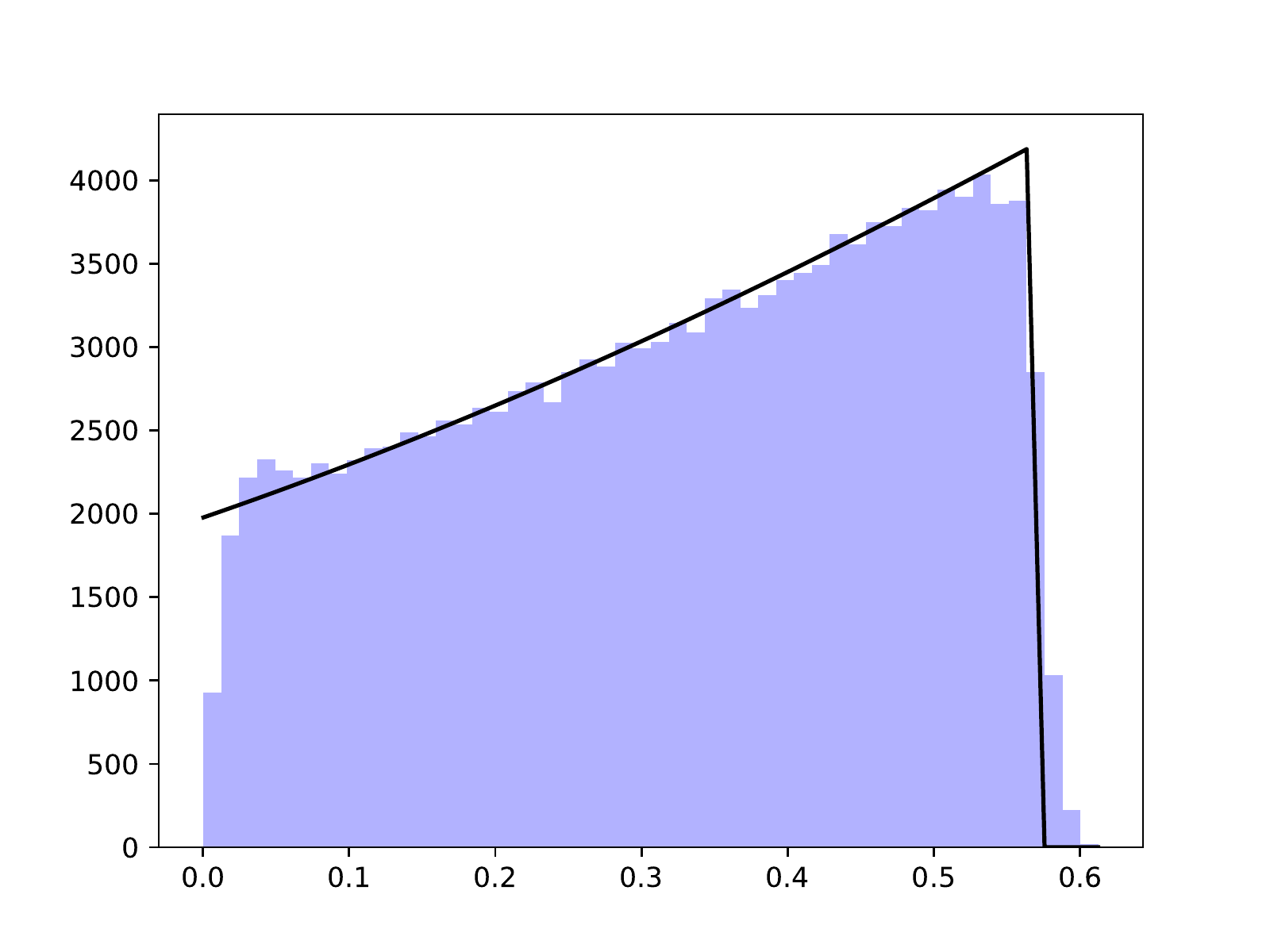}
\includegraphics[width=0.24\textwidth]{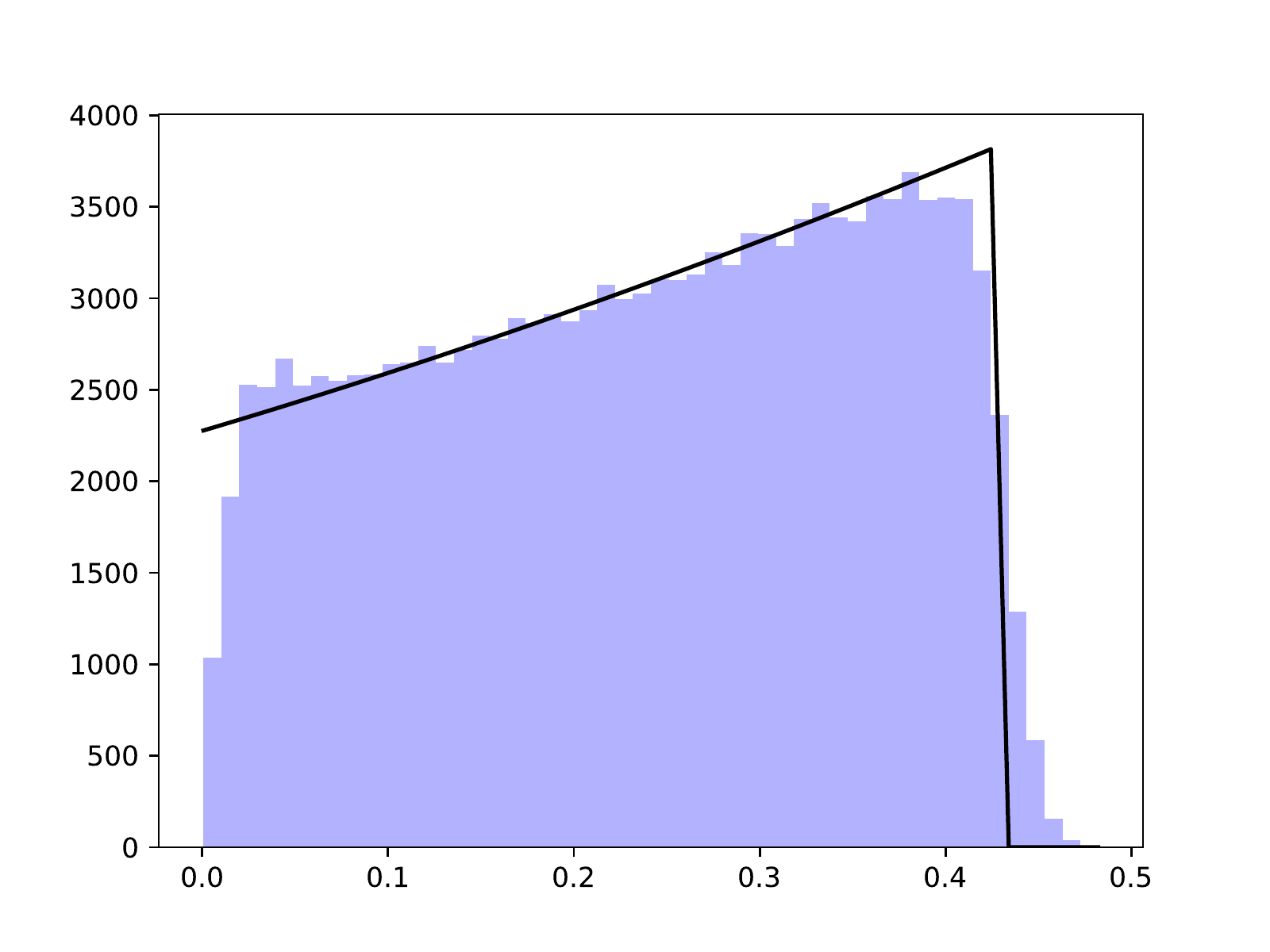}
\includegraphics[width=0.24\textwidth]{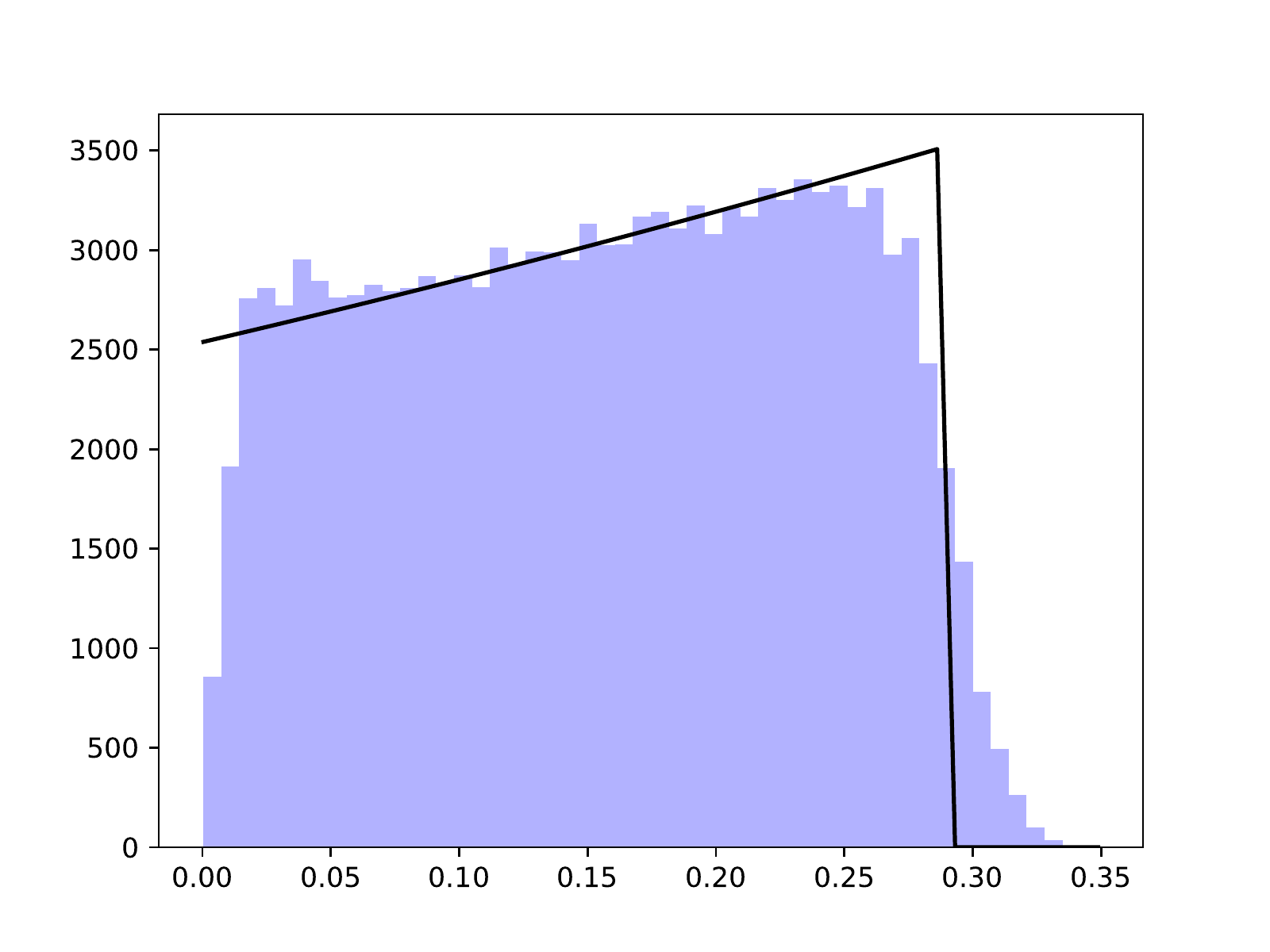}
\includegraphics[width=0.24\textwidth]{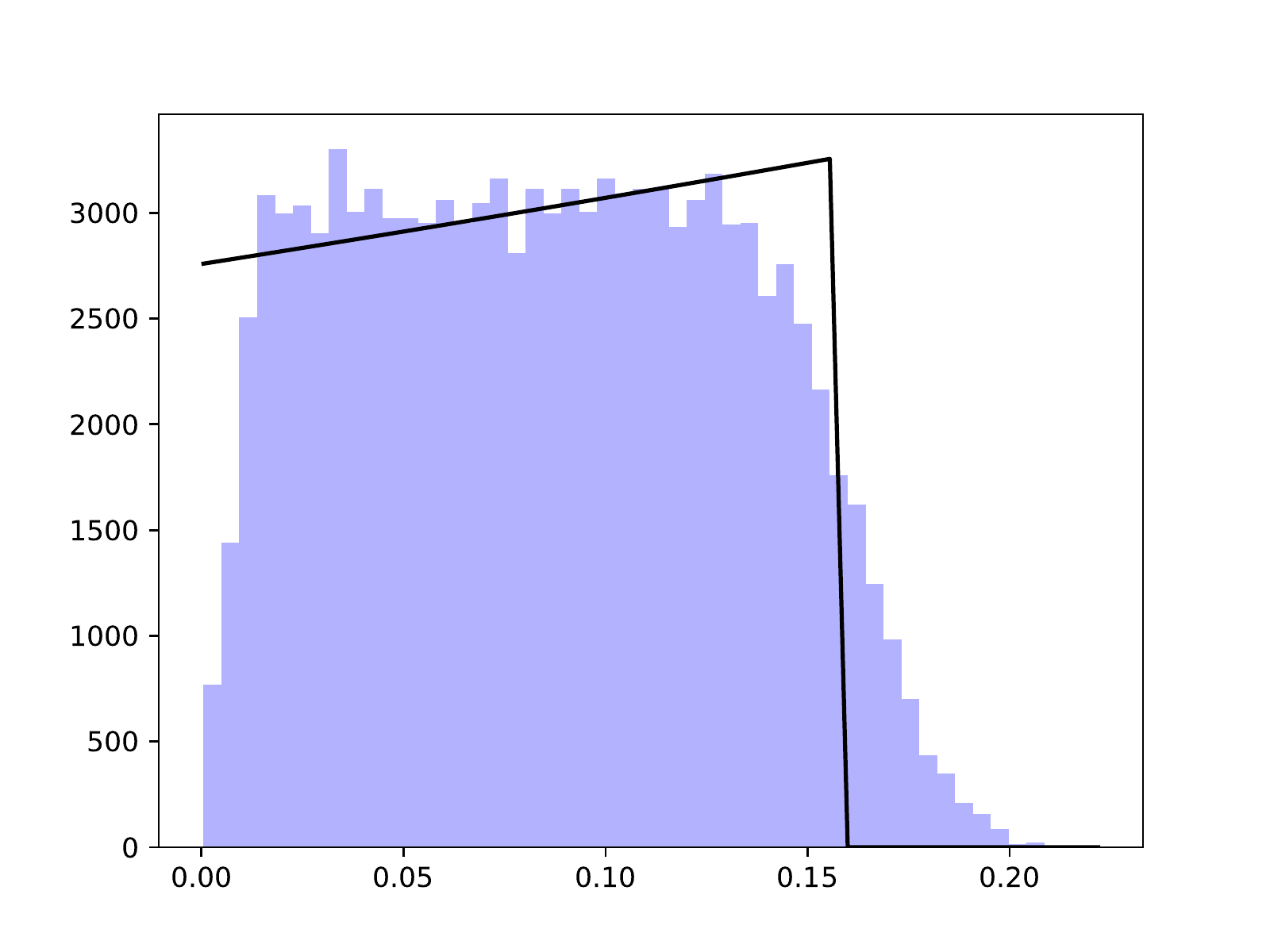}
\includegraphics[width=0.24\textwidth]{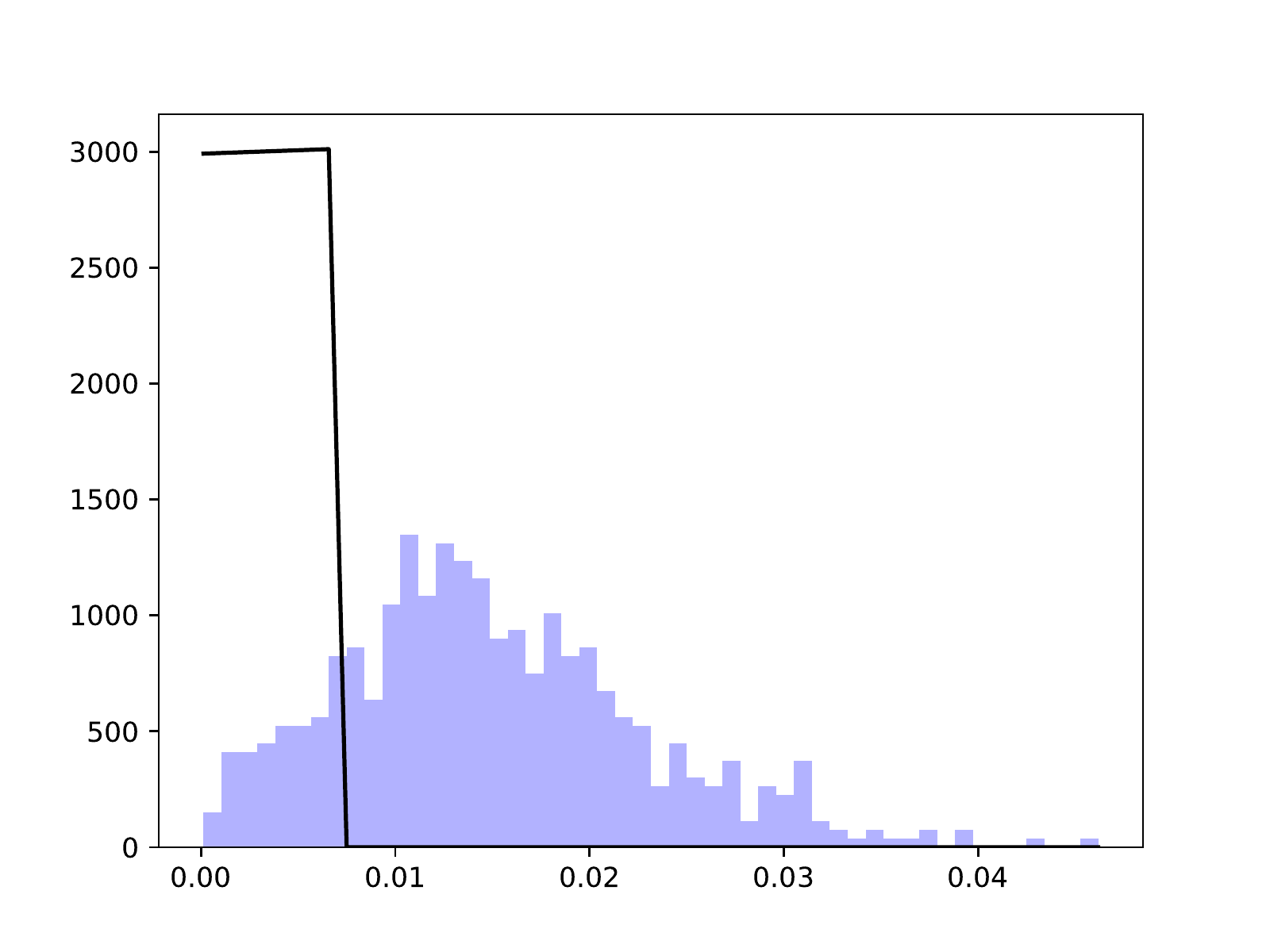}
\caption{
Histograms of the radial parts for the initial density $\psi(x,0)= 2 x$, $0<x<1$. The zeroes of the initial polynomial are i.i.d.\ and the degree is $n=3000$.  The orders of the derivatives are $4 + 425 k$ with $k\in \{0,\ldots,7\}$. The black curve shows the theoretical density given in~\eqref{eq:evolution_weyl_alpha_1/2}.
}
\label{pic:unif_disk_histogram}
\end{figure}

\vspace*{2mm}
\noindent
\textsc{Case $\alpha=1/2$} (Uniform distribution on the disk).   At time $t=0$, the zeroes are uniformly distributed on the unit disk meaning that their Lebesgue density is
$$
u(z,0) = \frac 1 \pi \ind_{\{|z|\leq 1\}}, \qquad z\in \C.
$$
The density of the radial parts is therefore
$$
\psi(x,0) = 2\pi x u(x,0) = 2 x \ind_{\{x\leq 1\}}, \qquad x\geq 0.
$$
This case corresponds to the Weyl polynomials defined by~\eqref{eq:weyl_poly_def} with $\alpha=1/2$.  Let us compute the asymptotic density of the radial parts of the zeroes of the  $[tn]$-th derivative. The implicit equation~\eqref{eq:Psi_Weyl_implicit} for $\Psi(x,t)$ takes the form
$$
-\frac 12 \log (\Psi(x,t) + t) + \log \Psi(x,t) = \log x, \qquad  -\infty < x <  1-t.
$$
Exponentiating and then squaring, we obtain
$$
x^2 = \frac{\Psi^2(x,t)}{\Psi(x,t) + t}, \qquad 0 < x <  1-t.
$$
This yields the following quadratic equation for $\Psi(x,t)$:
$$
\Psi^2(x,t) - x^2 \Psi(x,t)  - x^2 t = 0, \qquad  0 < x <  1-t.
$$
Solving it, we obtain
%
that the  distribution function of the radial parts is given by
$$
\Psi(x,t)  = \frac{x^2 + \sqrt{x^4 + 4x^2 t}}{2}, \qquad 0 < x <  1-t.
$$
The other solution of the quadratic equation is negative and can be discarded.
Taking the derivative  in $x$ yields the following density of the radial parts at time $t$:
\begin{equation}\label{eq:evolution_weyl_alpha_1/2}
\psi(x,t) = \left(x + \frac{x^2+2t}{\sqrt{x^2 + 4 t}}\right) \ind_{\{0 < x <  1-t\}}, \qquad x\geq 0,\;\; 0 \leq t < 1.
\end{equation}
One may conjecture that this formula applies to several cases in which the roots at time $0$ are asymptotically uniform on the unit disk, for example to the i.i.d.\ roots (see Figure~\ref{pic:unif_disk_histogram} for numerical simulations), the Weyl polynomials (for which the above derivation is essentially rigorous), or to the eigenvalues of a random Ginibre matrix; see Figure~\ref{pic:zeroes_weyl_unif_disk}.

\begin{figure}[!tbp]
\includegraphics[width=0.32\textwidth]{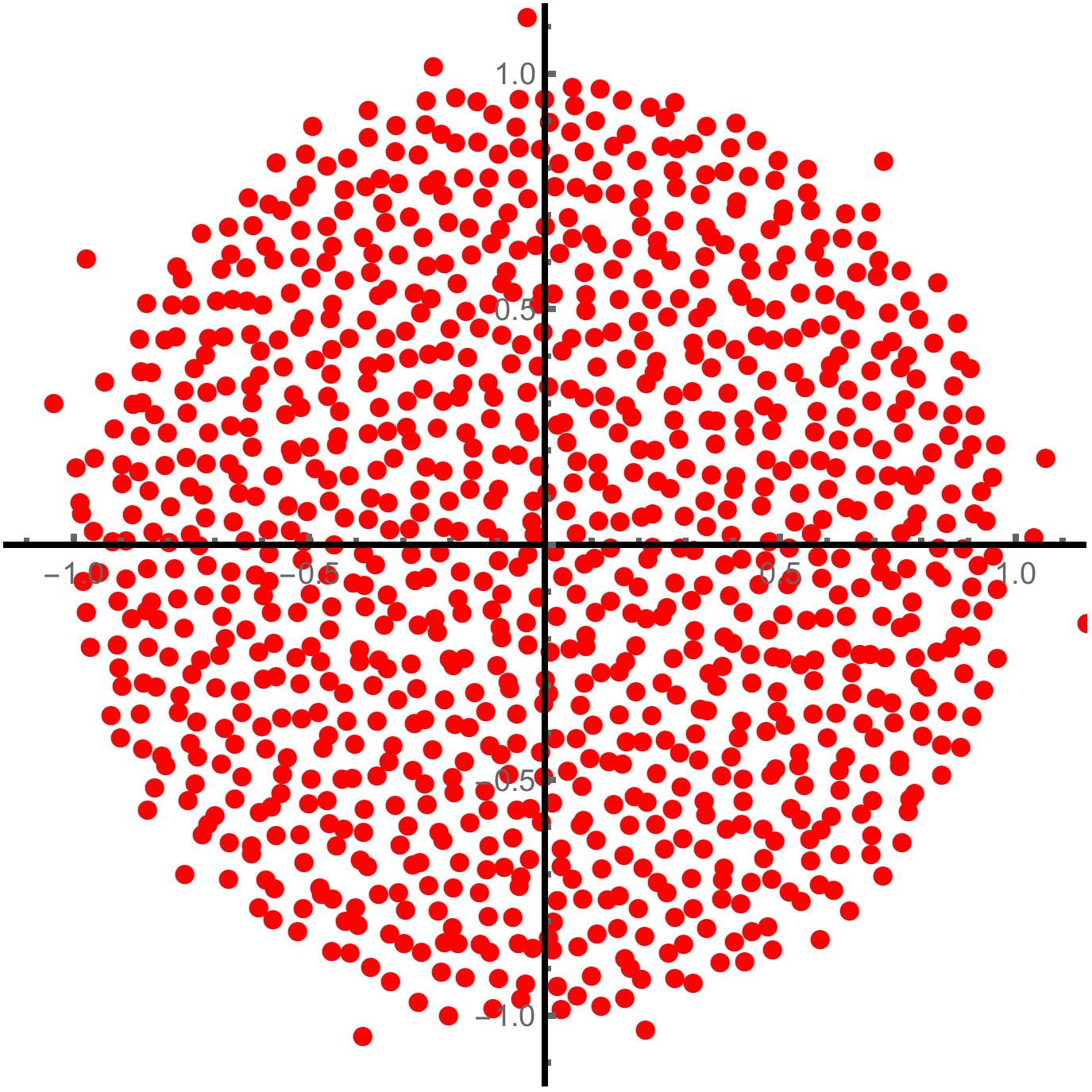}
\includegraphics[width=0.32\textwidth]{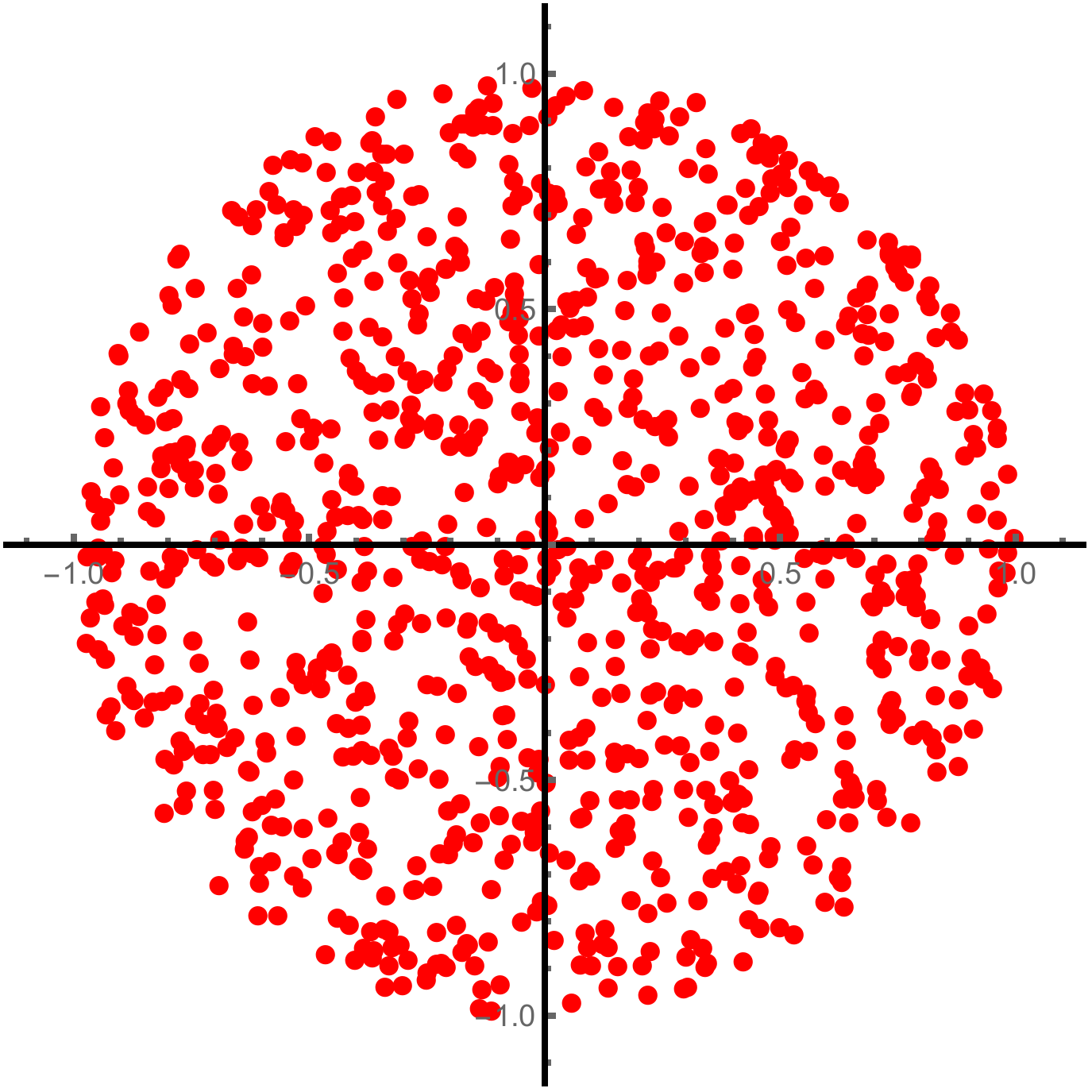}
\includegraphics[width=0.32\textwidth]{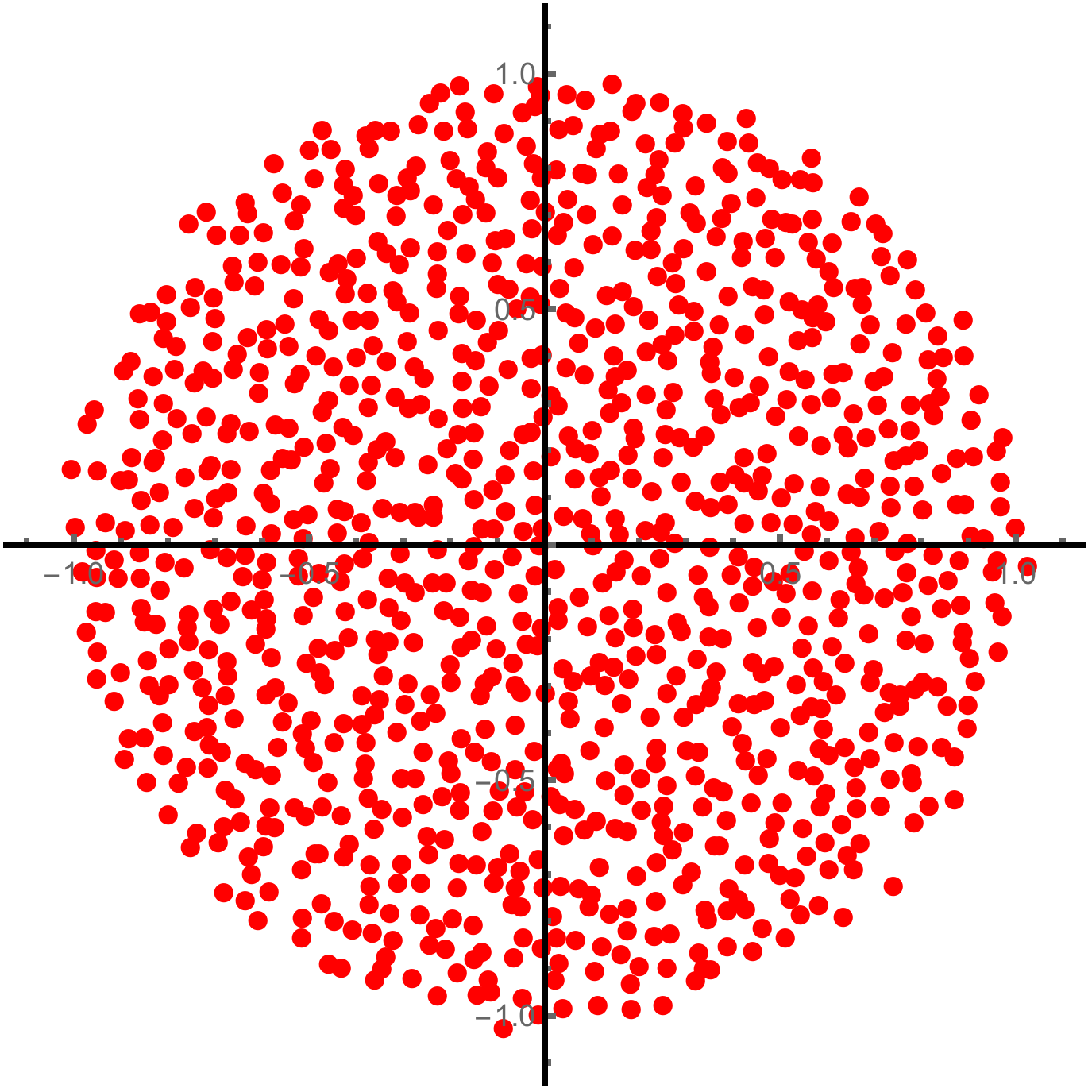}
\includegraphics[width=0.32\textwidth]{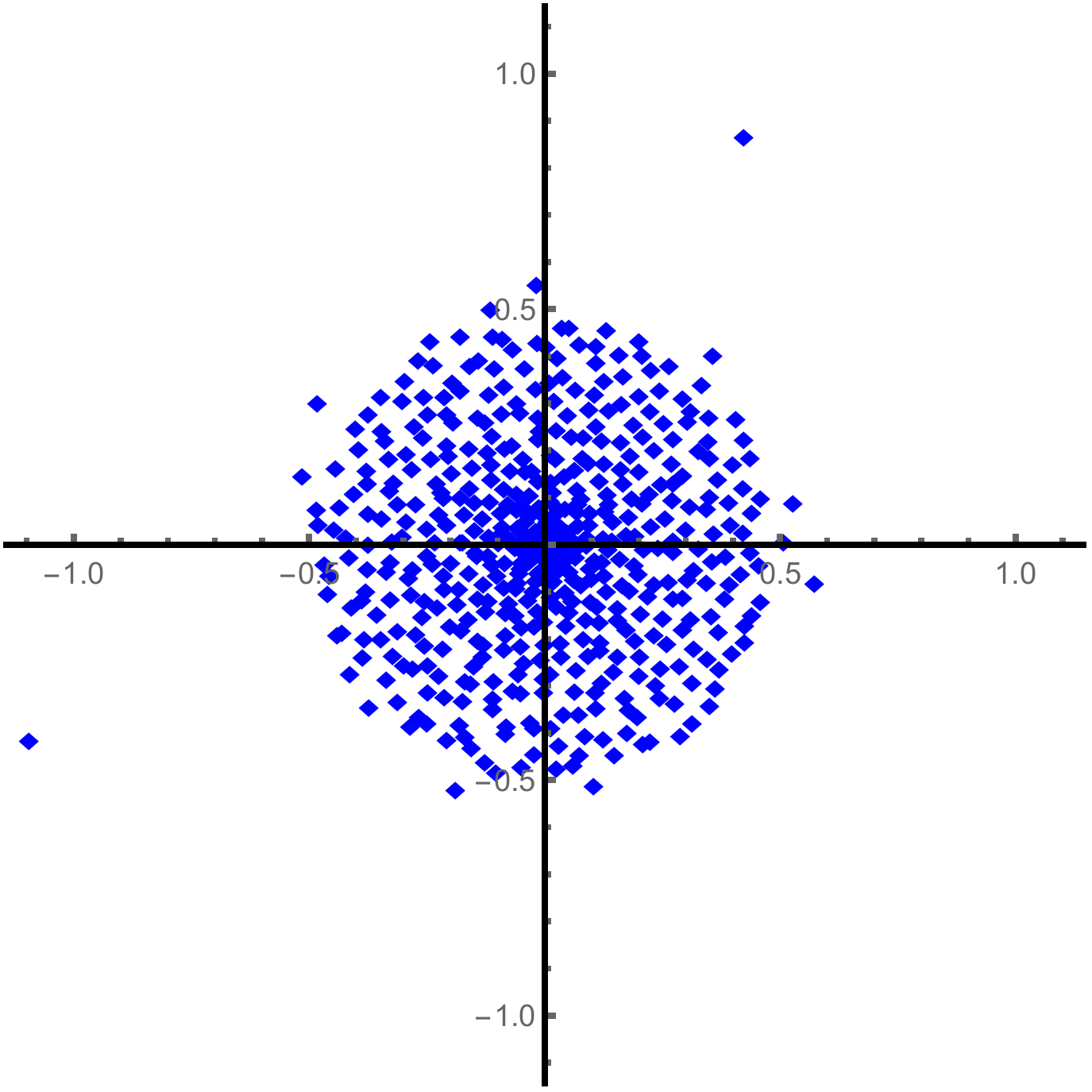}
\includegraphics[width=0.32\textwidth]{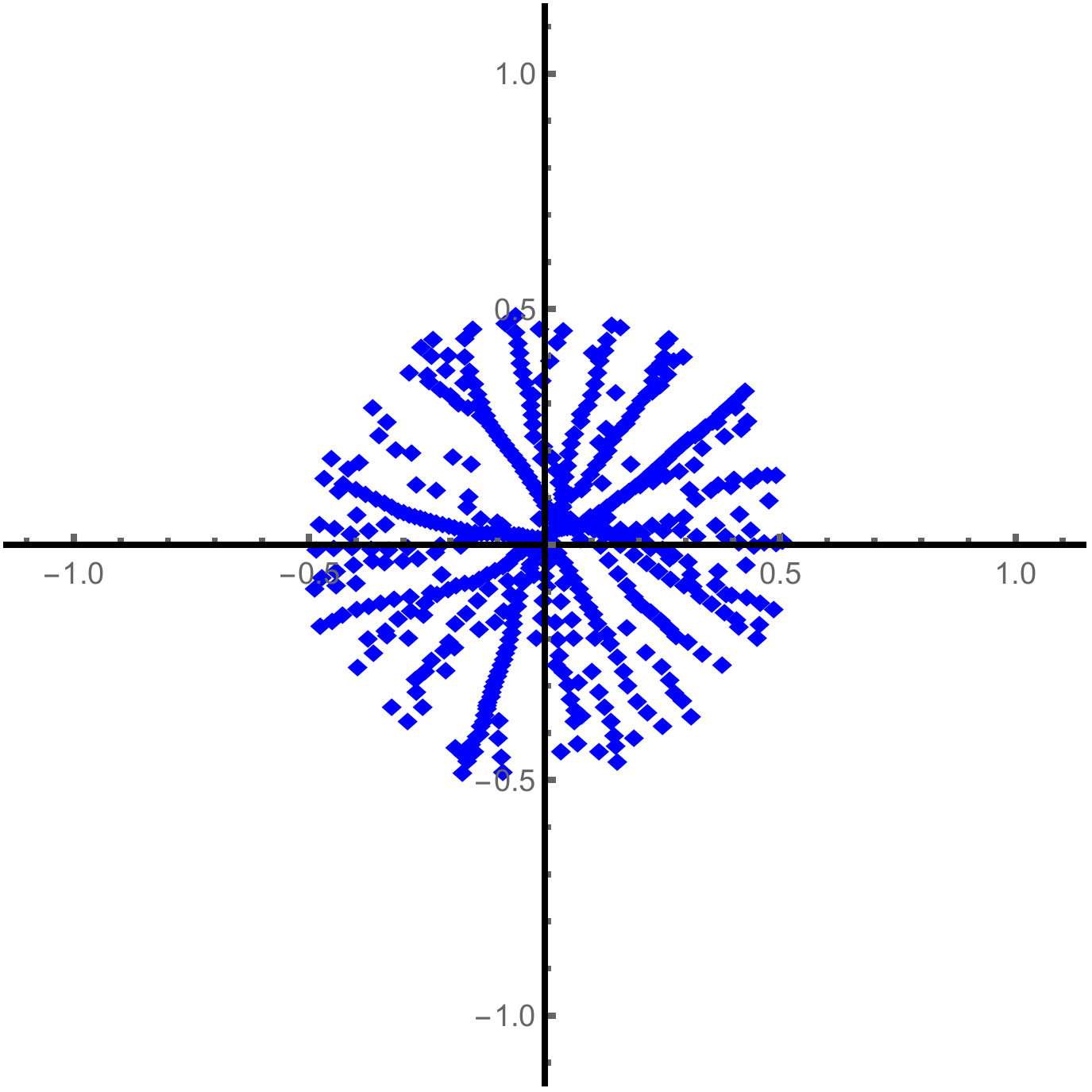}
\includegraphics[width=0.32\textwidth]{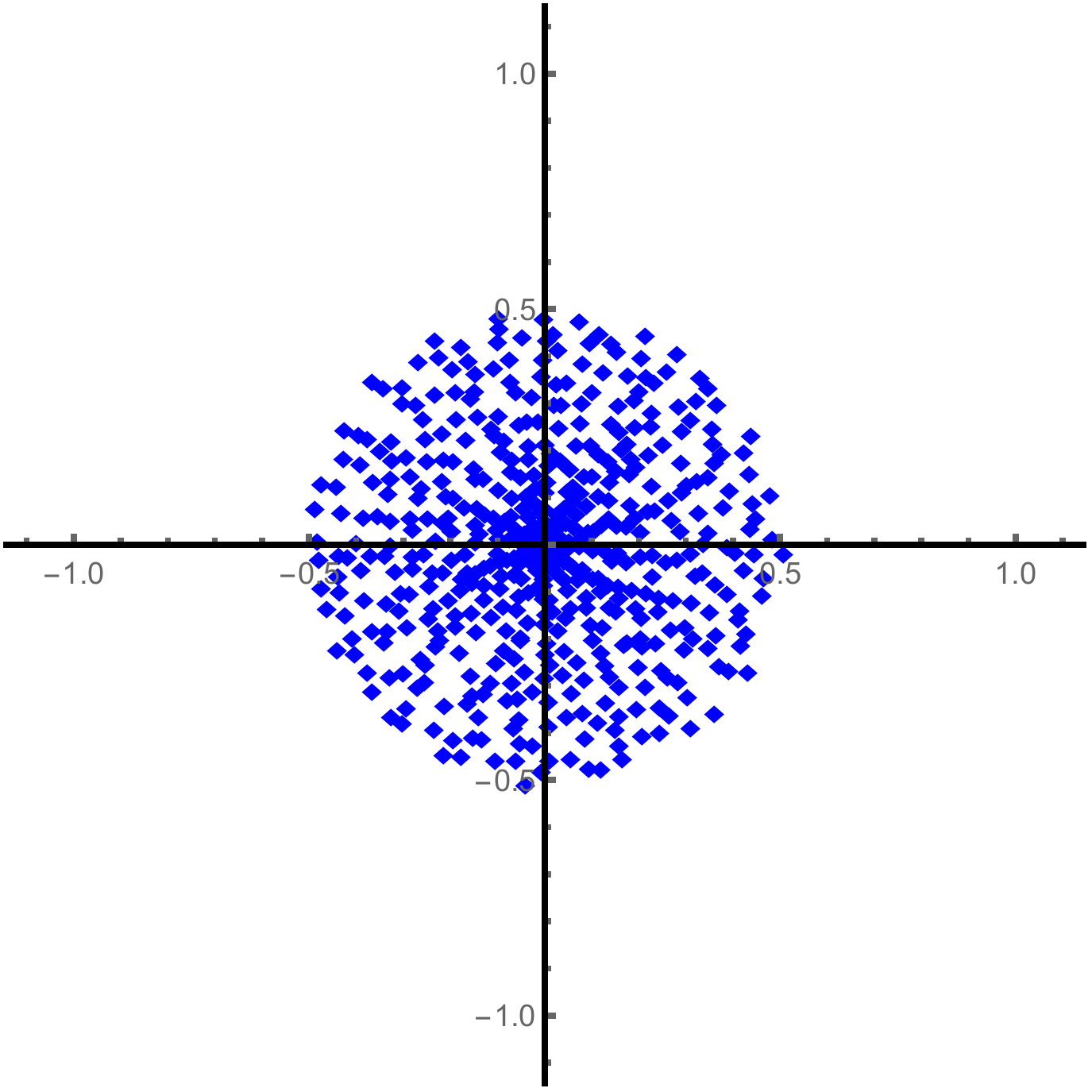}
\caption{
First column: Zeroes of the Weyl polynomial of degree $n=1000$  and the zeroes of its $500$-th derivative.
Second column: Same for $n=1000$ independent zeroes distributed uniformly on the unit disk.
Third column: Same for the characteristic polynomial of the Ginibre matrix of size $n=1000$.
}
\label{pic:zeroes_weyl_unif_disk}
\end{figure}

\begin{figure}[!tbp]
\includegraphics[width=0.24\textwidth]{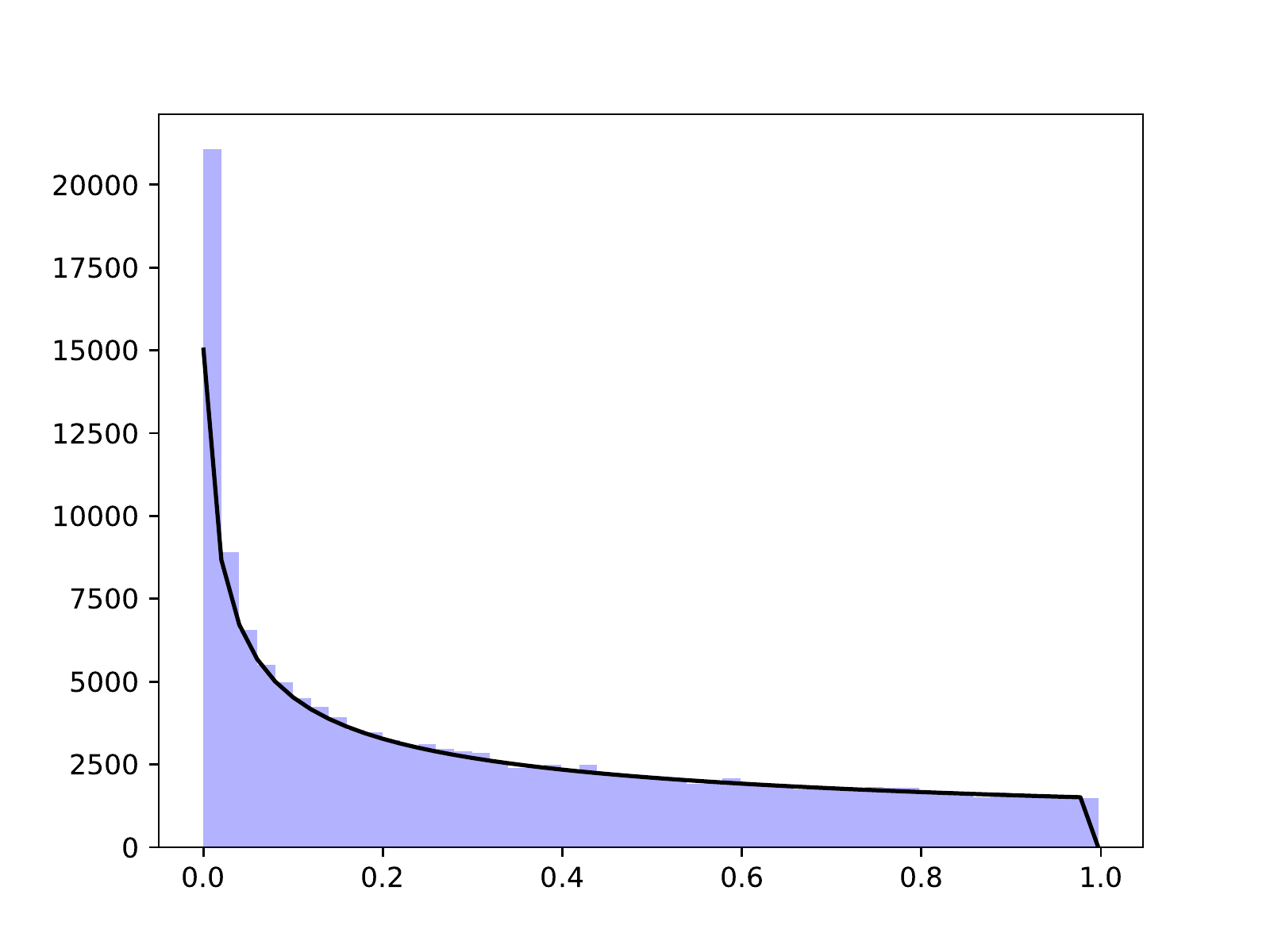}
\includegraphics[width=0.24\textwidth]{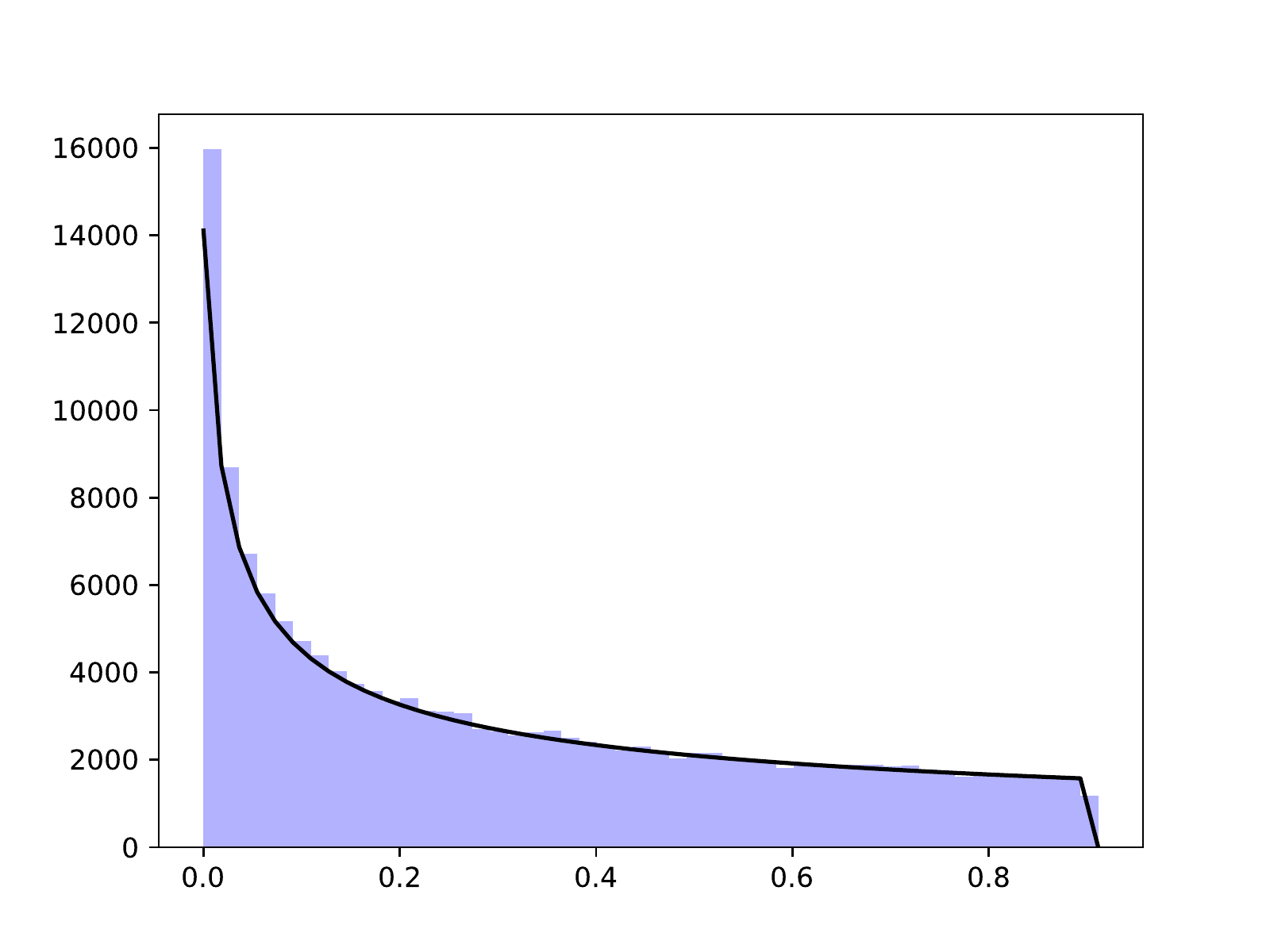}
\includegraphics[width=0.24\textwidth]{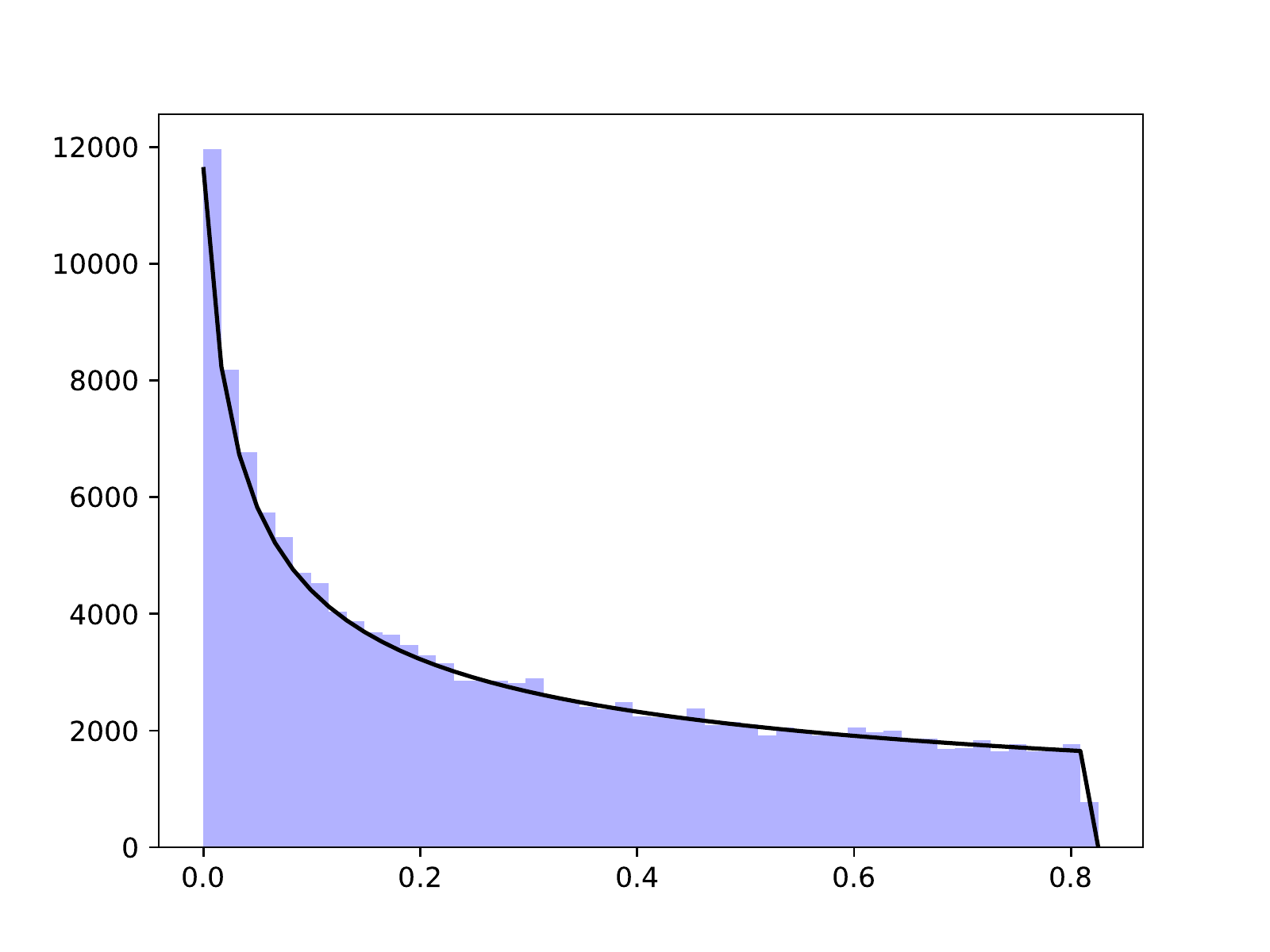}
\includegraphics[width=0.24\textwidth]{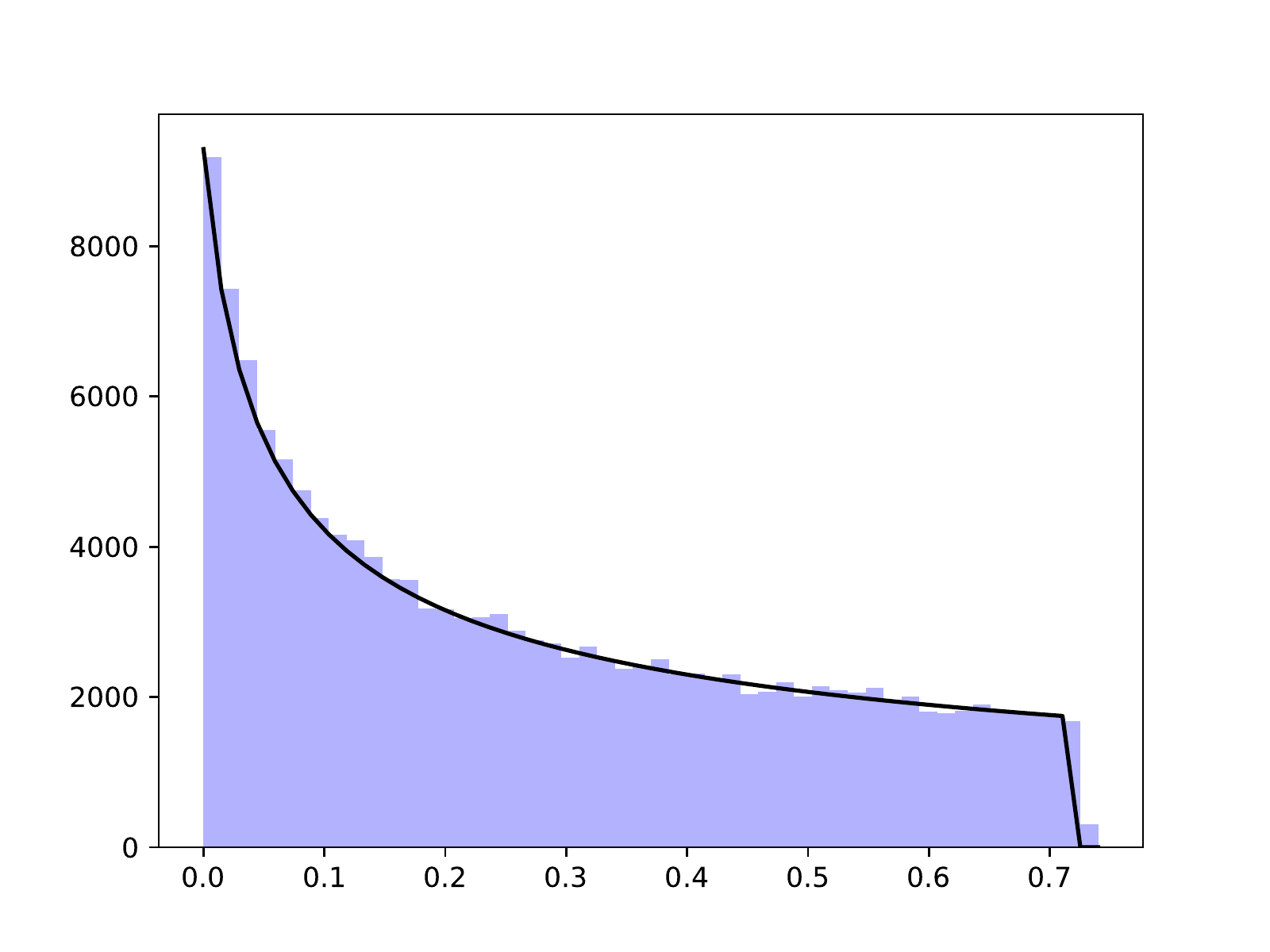}
\includegraphics[width=0.24\textwidth]{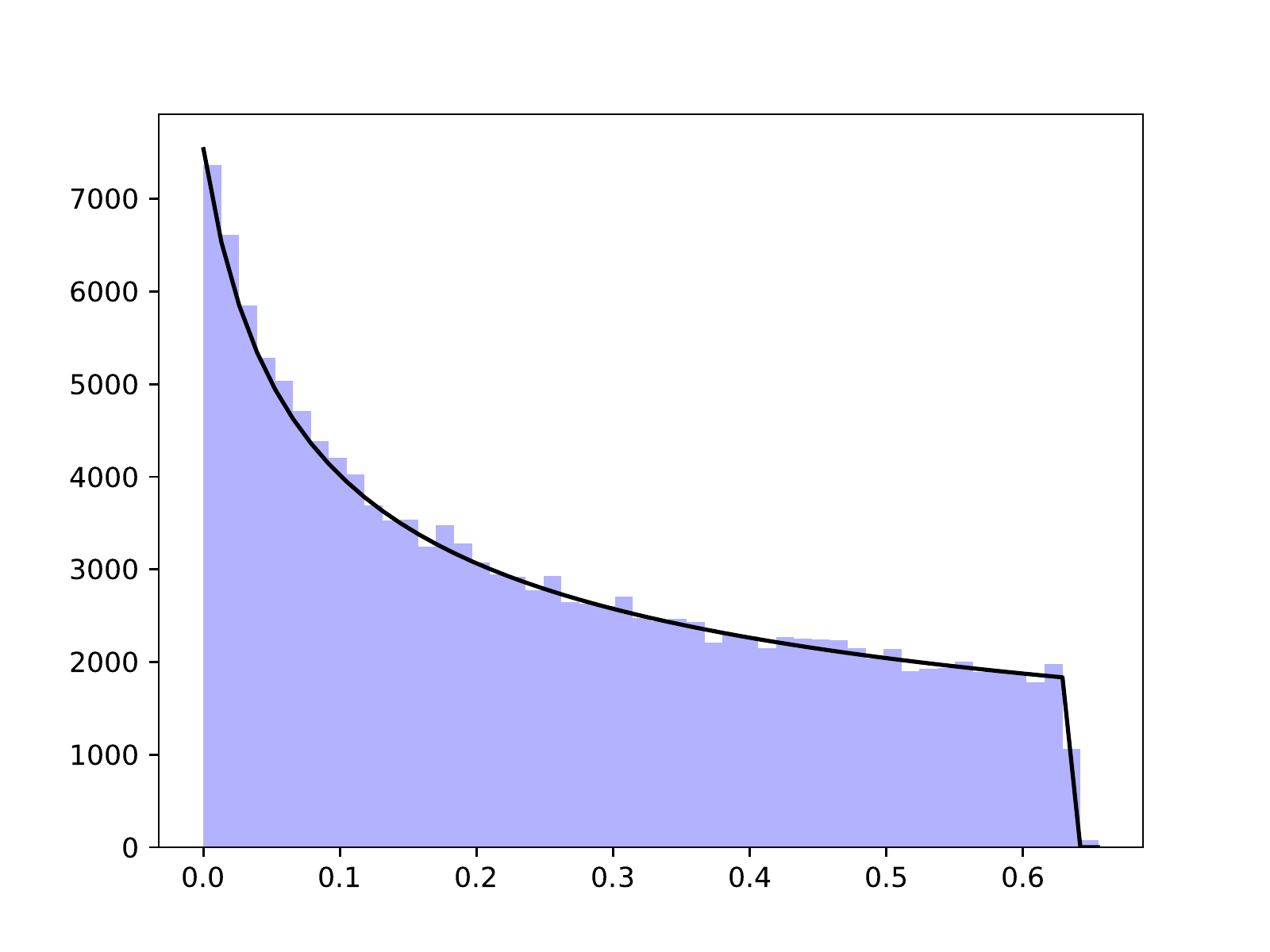}
\includegraphics[width=0.24\textwidth]{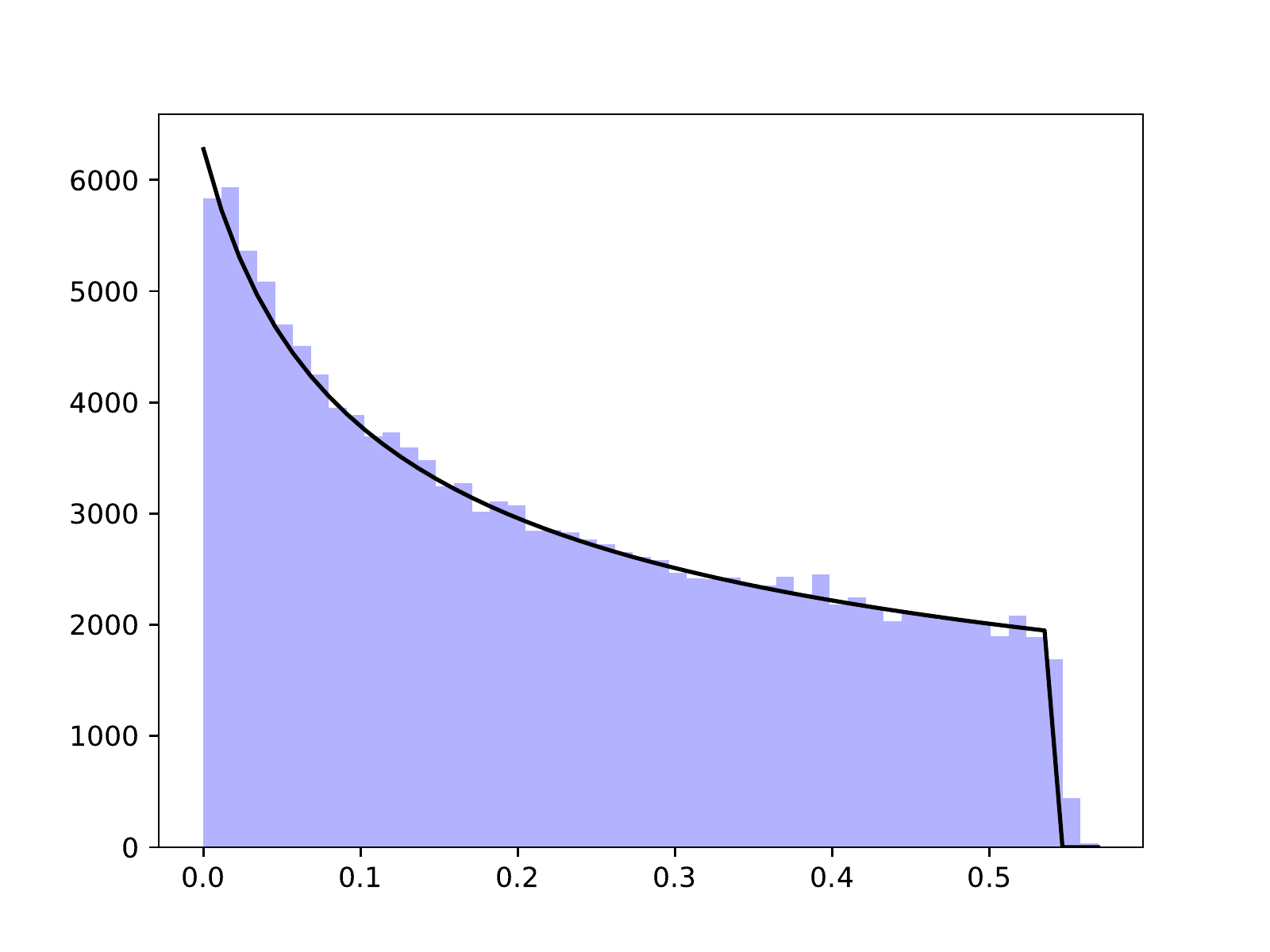}
\includegraphics[width=0.24\textwidth]{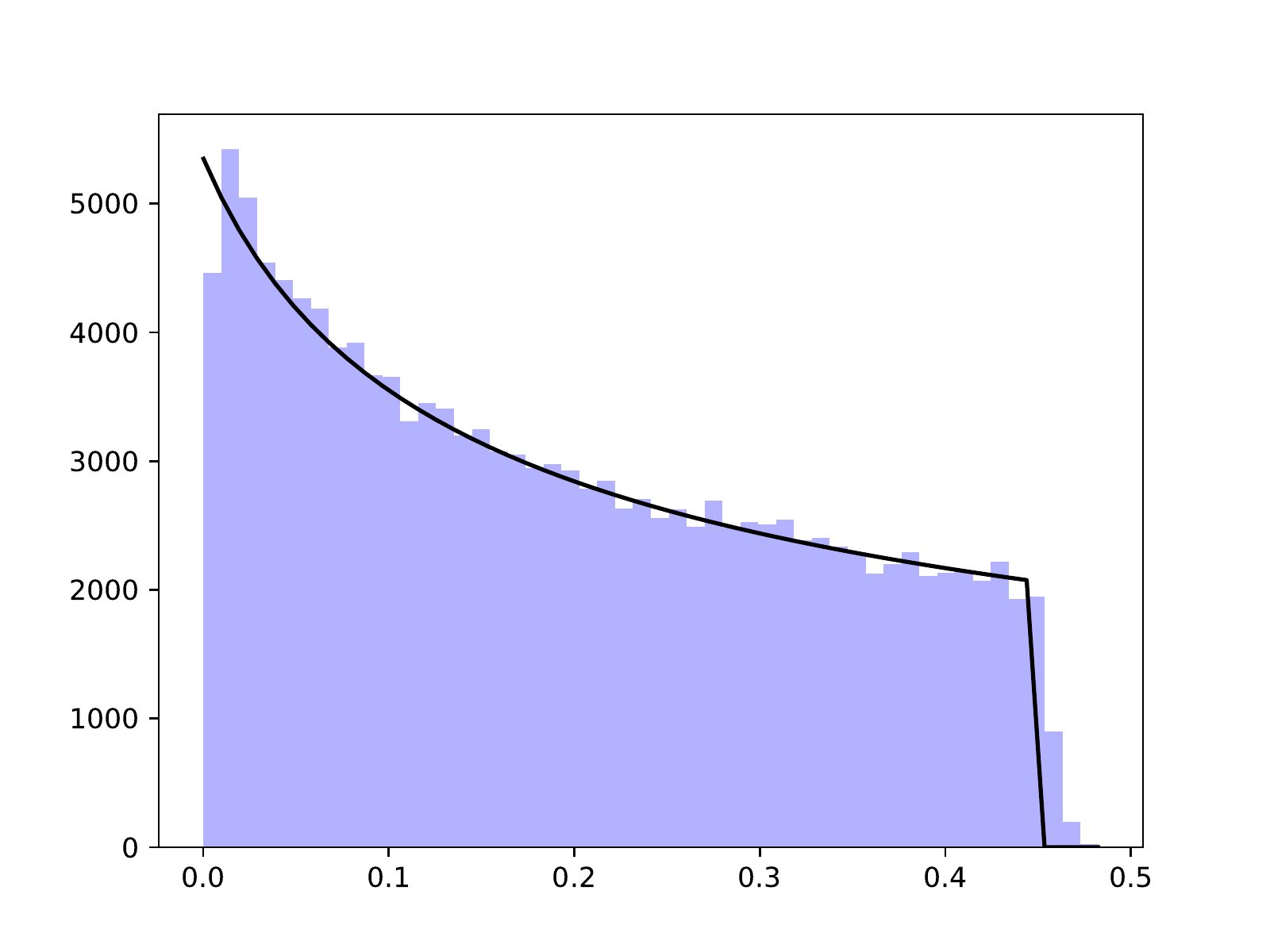}
\includegraphics[width=0.24\textwidth]{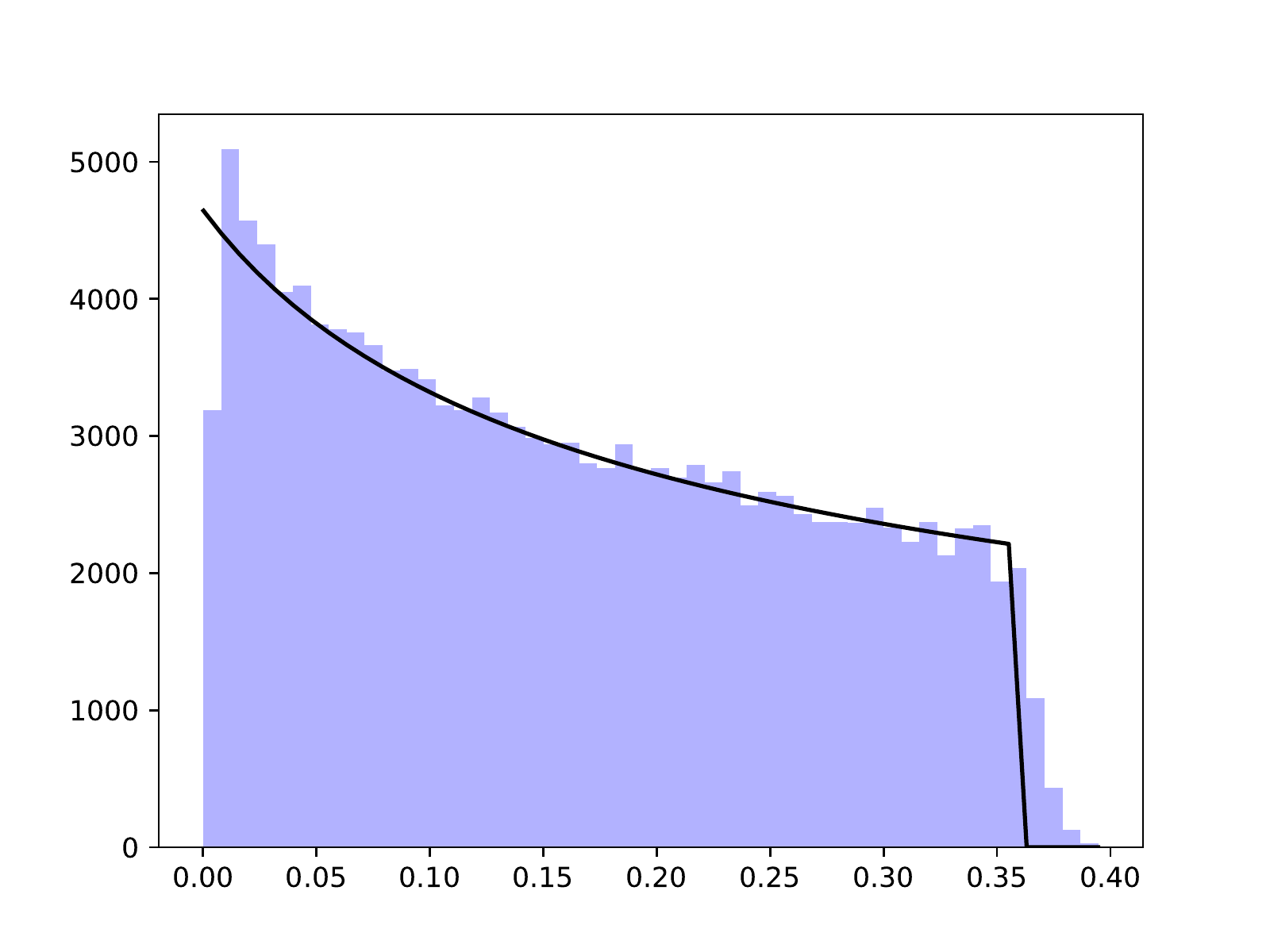}
\includegraphics[width=0.24\textwidth]{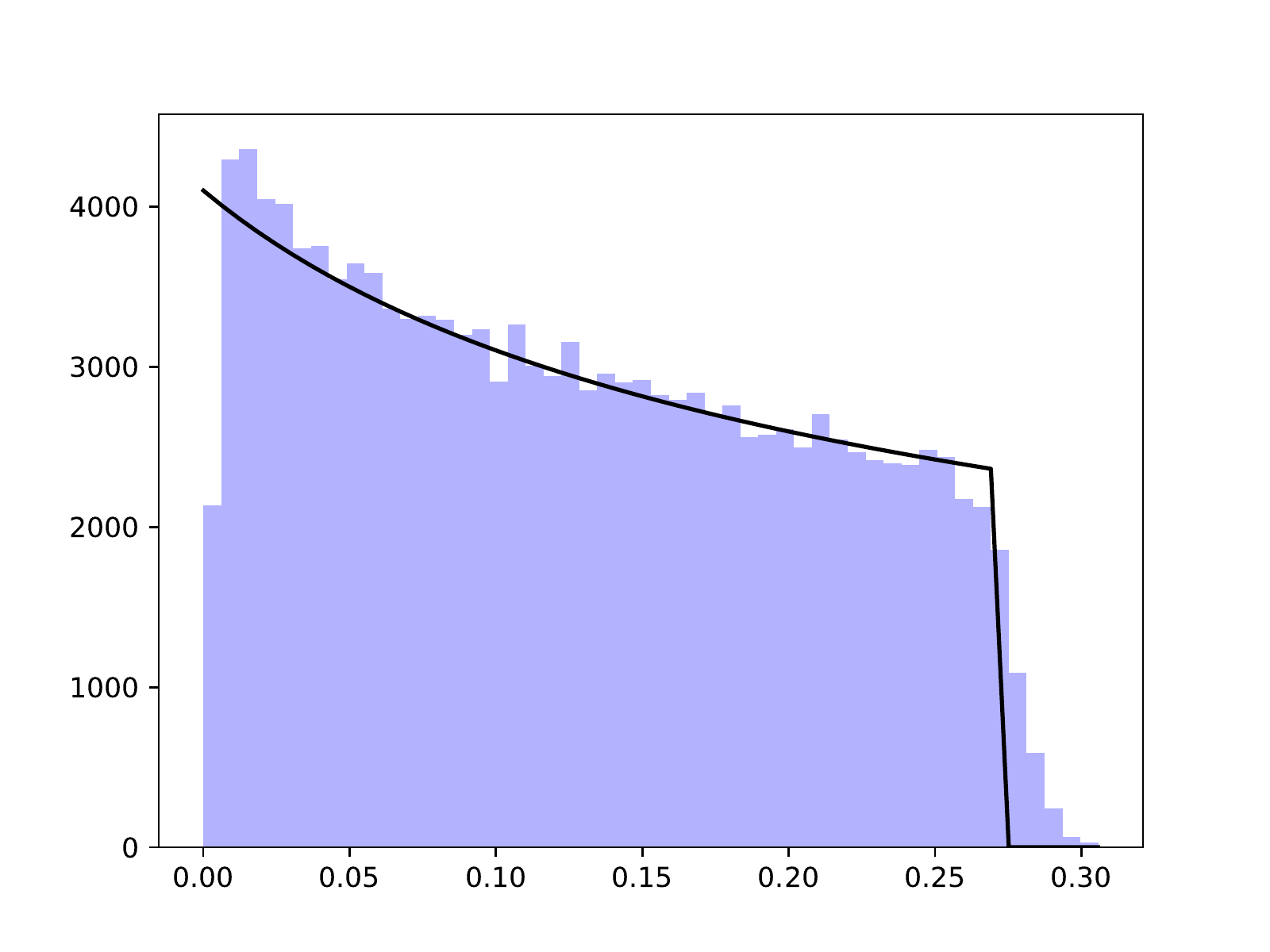}
\includegraphics[width=0.24\textwidth]{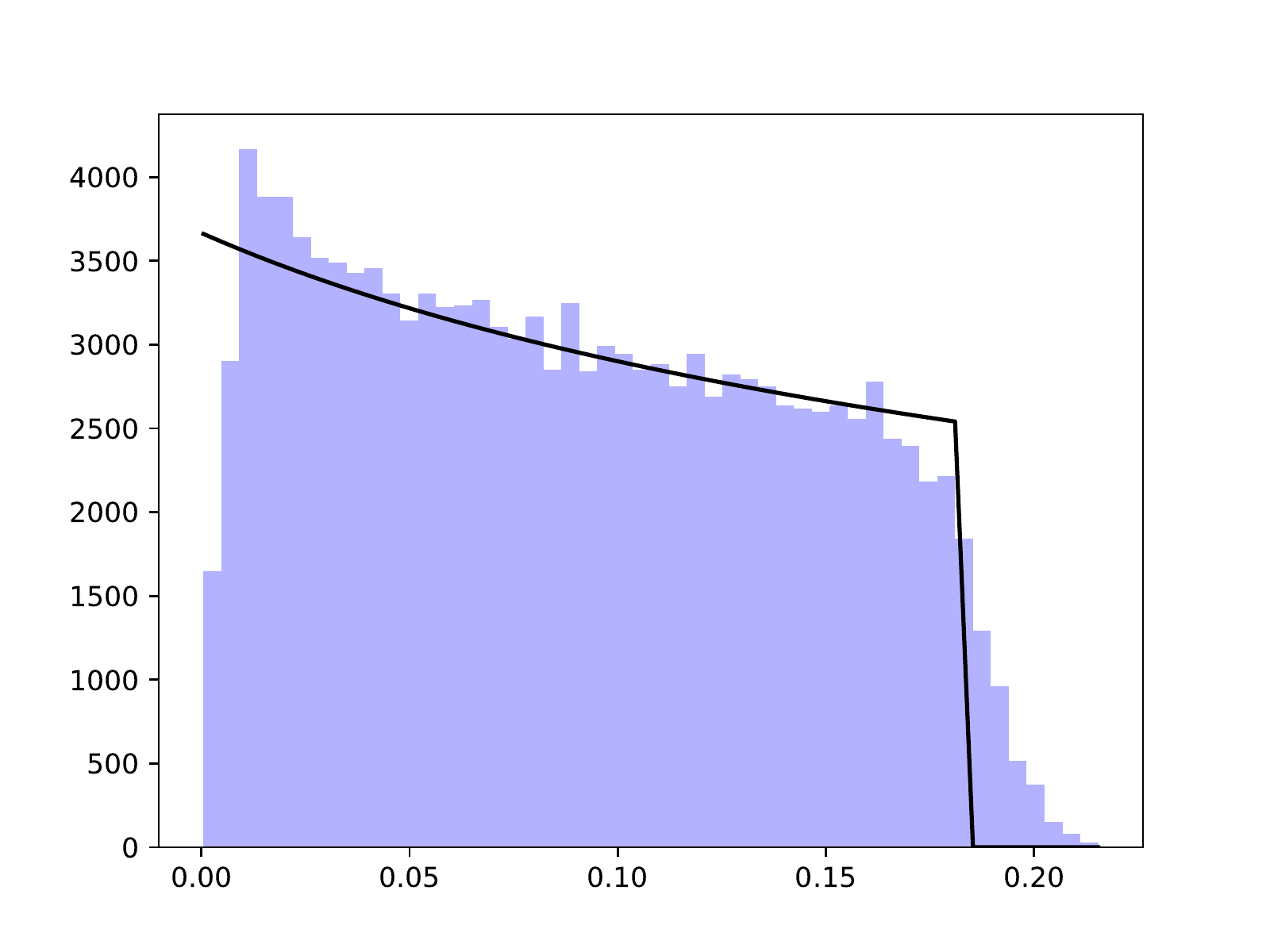}
\includegraphics[width=0.24\textwidth]{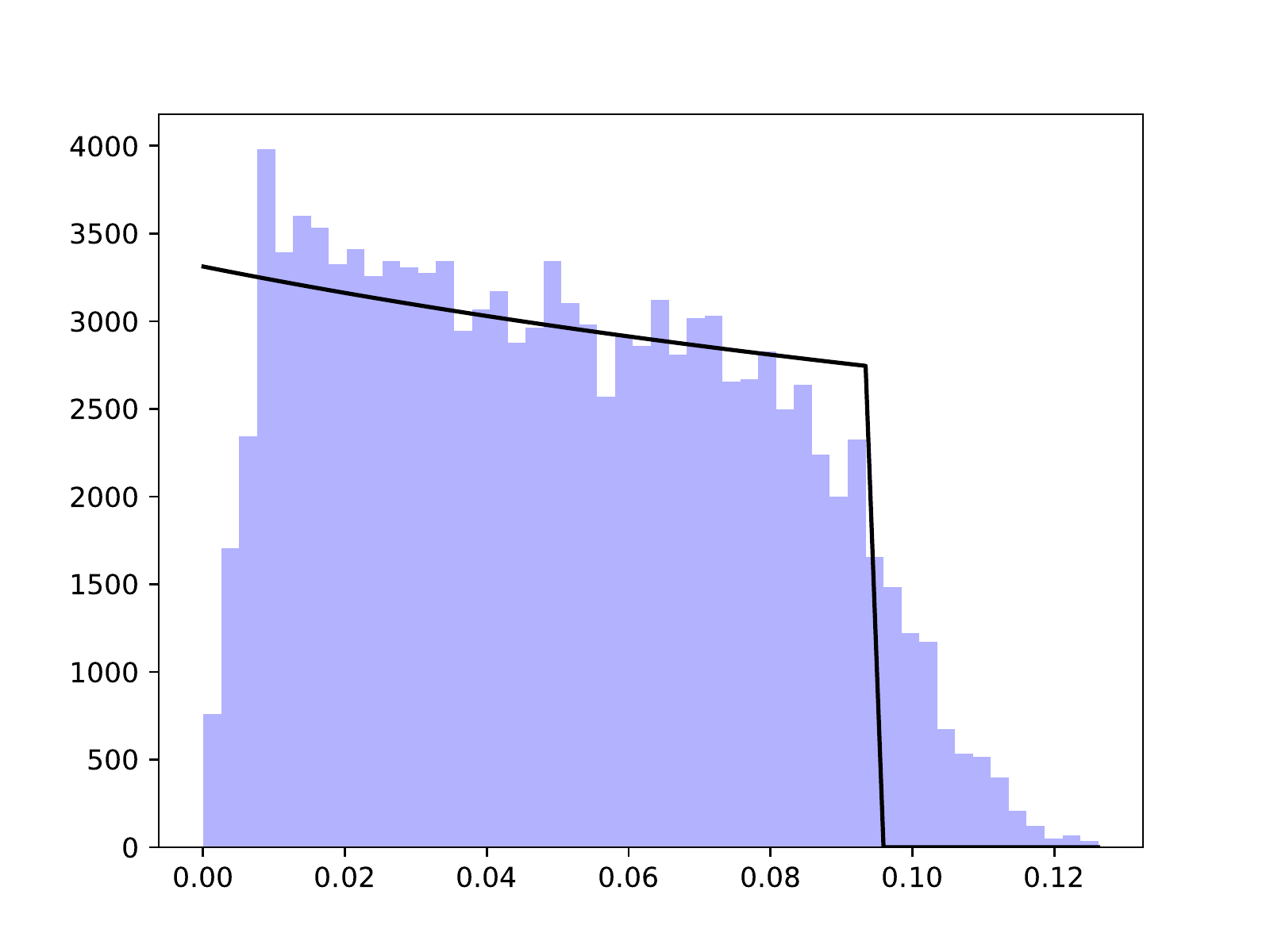}
\includegraphics[width=0.24\textwidth]{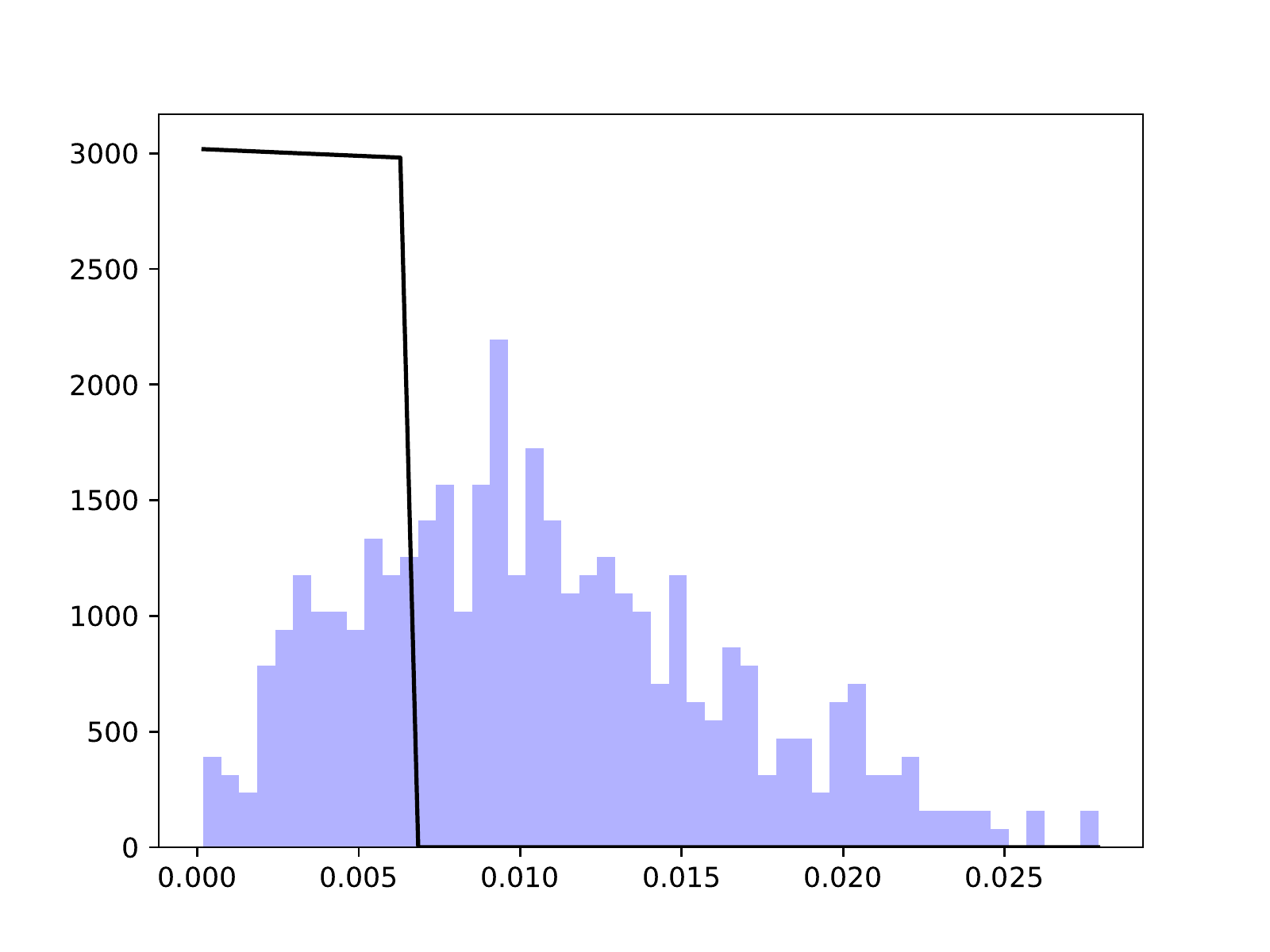}
\caption{
Histograms for the initial density $\psi(x,0)= 1/(2\sqrt x)$, $0<x<1$. The zeroes of the initial polynomial are i.i.d. The degree is $n=3000$ and the orders of the derivatives are $9+270 k$ with $k\in \{0,\ldots, 11\}$.  The black curve shows the theoretical density given in~\eqref{eq:psi_sqrt_density}.
}
\label{pic:sqrt_zeroes_histogram}
\end{figure}

\vspace*{2mm}
\noindent
\textsc{Case $\alpha=2$}. The initial condition is
$$
\psi(x,0)
=
\frac {\ind_{\{0<x<1\}}} {2\sqrt x},
\qquad x\geq 0.
$$
That is to say, the square roots of the radial parts are uniformly distributed on $[0,1]$. The implicit equation~\eqref{eq:Psi_Weyl_implicit} takes the form
$$
\log (\Psi(x,t) + t) + \log \Psi(x,t) = \log x, \qquad  0 < x < 1-t.
$$
This results in the following quadratic equation for $\Psi(x,t)$:
$$
\Psi^2(x,t) + t \Psi(x,t) - x = 0.
$$
Solving it yields
$$
\Psi(x,t) = \frac {-t + \sqrt{t^2+4x}}{2}, \qquad  0< x  < 1-t.
$$
Differentiating in $x$ we arrive at
\begin{equation}\label{eq:psi_sqrt_density}
\psi(x,t) = \frac {\ind_{\{0 < x <  1-t\}}} {\sqrt{t^2+4x}}, \qquad  x>0, \;\; 0 \leq t < 1.
\end{equation}
This result is in a very good agreement with numerical simulation shown on Figure~\ref{pic:sqrt_zeroes_histogram}.

\begin{remark}
It is also possible to perform similar computations for the initial density of zeroes of the form
$$
u(z,0)=\frac 1 {2\pi \alpha} |z|^{(1/\alpha)-2},
\qquad z\in \C.
$$
It should be stressed that the integral of this density is infinite, which means that it corresponds to zeroes of a random analytic function rather than a polynomial.
The previous calculations apply with the only difference that now the indicator function $\ind_{\{0 < x <  1-t\}}$ has to be removed from the final results and that these are valid for arbitrary $t>0$.
\end{remark}

\subsection{Uniform distribution of radial parts on an interval}
Consider now zeroes whose radial parts are uniformly distributed on the interval $[r_1,r_2]$, for some $0\leq r_1 < r_2 < \infty$. The initial condition is
$$
\psi(x,0) = (r_2-r_1)^{-1} \ind_{[r_1,r_2]}(x), \qquad x>0.
$$
Applying the usual scheme, one easily gets $v'(x) = \log (r_1 + (r_2-r_1)x)$ for $x\in [0,1]$ and hence
$$
\partial_1(x,t) = \log\left((r_1 + (r_2-r_1)(x+t))\cdot \frac{x}{x+t}\right), \qquad  0 < x < 1-t.
$$
Inverting this function, we obtain
$$
\Psi(x,t) = J(\log x,t) = \frac{x - (r_2-r_1)t -r_1 + \sqrt{(r_1+(r_2-r_1)t - x)^2 + 4t(r_2-r_1)x}}{2(r_2-r_1)},
$$
for $0 <  t <1$ and $0\leq x \leq (1-t)r_2$.
Differentiating in $x$, we arrive at the general solution
$$
\psi(x,t) = \frac 1 {2 (r_2-r_1)} + \frac{x + (r_2-r_1)t - r_1}{2 (r_2-r_1)\sqrt{(r_1 + (r_2-r_1)t-x)^2 + 4 t (r_2-r_1) x}},
$$
for $0 < t <1$ and $0\leq x \leq (1-t)r_2$.  Note that the void disk of radius $r_1$ present in the initial condition disappears instantaneously for every $t\neq 0$.  This solution is easily seen to reduce to~\eqref{eq:evolution_weyl_alpha_1} if $r_1=0$ and $r_2=1$, and  to~\eqref{eq:kac_evolution} if $r_1=1$ and $r_2=1+\eps$ with $\eps\downarrow 0$.

\subsection{Elliptic polynomials}
Another family of initial densities of complex zeroes for which the complete dynamics can be written down explicitly has the form
$$
u(z,0)
= \frac{1}{2\pi \alpha} \cdot \frac{|z|^{(1/\alpha) - 2}}{(1+|z|^{1/\alpha})^2}, \qquad z\in \C,
$$
where the parameter $\alpha$ satisfies $\alpha \geq 0$. The corresponding density of the radial parts has the form
$$
\psi(x,0) = 2 \pi x   u(x,0) = \frac{1}{\alpha} \cdot \frac{x^{(1/\alpha) - 1}}{(1+x^{1/\alpha})^2},\qquad x\geq 0.
$$
An example of a random polynomials with stochastically independent coefficients having this asymptotic distribution of zeroes is given by
\begin{equation}\label{eq:elliptic_poly_def}
E_n(z)
=
\sum_{k=0}^n \xi_k \left(\frac{n(n-1)\ldots (n-k+1)}{k!}\right)^\alpha z^k,
\qquad
|z|<1;
\end{equation}
see~\cite[Theorem~2.1]{kabluchko_zaporozhets12a}. The case with $\alpha=1/2$ and complex Gaussian $\xi_k$'s plays a special role~\cite[Chapter~2]{peres_book} and corresponds to the zeroes distributed in a $\text{SO}(3)$-invariant way  on the Riemann sphere, after identifying it with the complex plane using the stereographic projection.

The corresponding function $v$ can be determined either using~\eqref{eq:mu_0_I'} and~\eqref{eq:I_def} or by computing the exponential profile of the coefficients in~\eqref{eq:elliptic_poly_def}; see~\cite[p.~1385]{kabluchko_zaporozhets12a}. It is given by
$$
v(x) = \alpha (x\log x + (1-x)\log (1-x)), \qquad 0\leq x \leq 1.
$$
It follows from~\eqref{eq:v_x_t_def} that for all $0\leq t <1$,
$$
v(x,t) = (\alpha-1) (x+t)\log (x+t) + \alpha (1 - x - t)\log (1 - x - t) +  x\log x, \quad 0\leq x \leq 1-t.
$$
The derivative in $x$ is given by
$$
\partial_1 v(x,t) = (\alpha-1) \log (x+t) - \alpha \log (1 - x - t)  + \log x, \qquad 0 \leq x \leq 1-t.
$$
The function $x\mapsto \partial_1 v(x,t)$  is monotone increasing and its range is the whole real line. The implicit equation for the corresponding inverse function $y\mapsto J(y,t)$ takes the form
$$
(\alpha-1) \log (J(y,t) + t) - \alpha \log (1 - J(y,t) - t)  + \log J(y,t) = y, \qquad y\in \R.
$$
Taking $y=\log x$, we obtain the following implicit equation for $\Psi(x,t) = J(\log x, t)$:
\begin{equation}\label{eq:Psi_elliptic_implicit}
(\alpha-1) \log (\Psi(x,t) + t) - \alpha \log (1 - \Psi(x,t) - t)  + \log \Psi(x,t) = \log x, \qquad x>0.
\end{equation}

\vspace*{2mm}
\noindent
\textsc{Case $\alpha=1$.}
Equation~\eqref{eq:Psi_elliptic_implicit} takes the form
$$
 -  \log (1 - \Psi(x,t) - t)  + \log \Psi(x,t) = \log x, \qquad x>0.
$$
Solving it yields
$$
\Psi(x,t) = \frac {x(1-t)}{1+x}, \qquad x>0.
$$
Differentiating in $x$, we arrive at the following solution:
$$
\psi(x,t) = \frac {1-t}{(1+x)^2}, \qquad x>0, \;\; 0\leq t<1.
$$
This solution is stable in the sense that the function $\psi(x,t)/(1-t)$ (which is a \textit{probability} density) does not depend on $t$.

\vspace*{2mm}
\noindent
\textsc{Case $\alpha=1/2$} (Uniform distribution on the Riemann sphere).
The implicit equation~\eqref{eq:Psi_elliptic_implicit} takes the form
$$
-\frac 12  \log (\Psi(x,t) + t) - \frac 12 \log (1-\Psi(x,t)-t)  + \log \Psi(x,t) = \log x, \qquad x>0.
$$
Multiplying by $2$ and exponentiating yields
$$
\Psi^2(x,t) = x^2 (\Psi(x,t) + t)(1 - \Psi(x,t) -  t),  \qquad x>0.
$$
Multiplying out, we arrive at the following quadratic equation for $\Psi(x,t)$:
$$
(1 + x^2) \Psi^2(x,t) + (2t-1) x^2 \Psi(x,t) +  x^2 t(t-1) = 0.
$$
Solving it gives the distribution function of the radial parts
$$
\Psi(x,t)
=
\frac{-(2t-1)x^2 + \sqrt{x^4 - 4 x^2 (t-1)t}}{2(1+x^2)},
\qquad
x>0.
$$
Differentiating, we obtain the following solution
$$
\psi(x,t) = -\frac{x(2t-1)}{(1+x^2)^2} + \frac{x^2-2t^2+2t^2x^2+2t-2tx^2}{(1+x^2)^2  \sqrt{x^2 - 4  (t-1)t}},
\qquad
x>0,\;\; 0\leq t\leq 1.
$$
Repeated derivatives of the elliptic polynomials with $\alpha=1/2$ have been studied by Feng and Yao~\cite[Theorem~6]{feng_yao} who did not give an explicit formula for $\psi(x,t)$. In part~(2) of their theorem they considered the regime when $t\to 1$ and therefore obtained a different limit distribution for the zeroes.

Let us also mention that the case $\alpha=2$ also corresponds to a quadratic equation and can be solved explicitly. We omit the details.

\subsection{Hyperbolic functions}
The last family of initial densities of complex zeroes for which we are able to explicitly determine the complete dynamics has the form
$$
u(z,0)
= \frac{1}{2\pi \alpha} \cdot \frac{|z|^{(1/\alpha) - 2}}{(1-|z|^{1/\alpha})^2}, \qquad |z|<1,
$$
where the parameter $\alpha$ satisfies $\alpha \geq 0$. The corresponding density of the radial parts has the form
$$
\psi(x,0) = 2 \pi x  u(x,0) = \frac{1}{\alpha} \cdot \frac{x^{(1/\alpha) - 1}}{(1-x^{1/\alpha})^2} \ind_{\{x<1\}},\qquad x\geq 0.
$$
In the special case $\alpha=1/2$, the zeroes are distributed according to the hyperbolic area measure on the unit disk. An example of a random analytic function with stochastically independent Taylor coefficients having this asymptotic distribution of zeroes is given by
\begin{equation}\label{eq:hyperbloic_poly_def}
F_n(z)
=
\sum_{k=0}^\infty \xi_k \left(\frac{n(n+1)\ldots (n+k-1)}{k!}\right)^\alpha z^k,
\qquad
|z|<1;
\end{equation}
see Theorem~2.1 in~\cite{kabluchko_zaporozhets12a}. The special case $\alpha=1/2$  (and when the $\xi_k$'s are standard complex Gaussian) is known under the name hyperbolic Gaussian analytic function; see~\cite[Chapters~2 and~5]{peres_book}.

Although the function $\psi(x,0)$ is not a probability density (since $\int_0^1 \psi(x,0)\dd x = +\infty$) and the function $F_n(z)$ is not a polynomial, the recipe described in Section~\ref{subsec:log_potential} can be applied with minor modifications. The main difference is that now the function $v:[0,\infty) \to\R$ describing the exponential profile of the coefficients in~\eqref{eq:hyperbloic_poly_def} is defined on the interval $[0,\infty)$ rather than on $[0,1]$.   It  is given by
$$
v(x) = \alpha (x\log x - (1+x)\log (1+x)), \qquad x\geq 0;
$$
see~\cite[p.~1385]{kabluchko_zaporozhets12a}.
It follows by~\eqref{eq:v_x_t_def} that
$$
v(x,t) = (\alpha-1) (x+t)\log (x+t) - \alpha (1+x+t)\log (1+x+t)  +  x\log x, \qquad x\geq 0.
$$
The derivative in $x$ is given by
$$
\partial_1 v(x,t) = (\alpha-1) \log (x+t) - \alpha \log (1+x+t)  + \log x, \qquad x\geq 0.
$$
The function $x\mapsto \partial_1 v(x,t)$  is monotone increasing and its range is the interval $(-\infty, 0)$. The implicit equation for the corresponding inverse function $y\mapsto J(y,t)$ takes the form
$$
(\alpha-1) \log (J(y,t) + t) - \alpha \log (J(y,t) + 1 + t) + \log J(y,t) = y, \qquad y<0.
$$
Taking $y=\log x$, we obtain the following implicit equation for $\Psi(x,t) = J(\log x, t)$:
\begin{equation}\label{eq:Psi_hyperbolic_implicit}
(\alpha-1) \log (\Psi(x,t) + t) - \alpha \log (\Psi(x,t) + 1 + t)  + \log \Psi(x,t) = \log x, \qquad 0<x<1.
\end{equation}

\vspace*{2mm}
\noindent
\textsc{Case $\alpha=1$.}
Equation~\eqref{eq:Psi_hyperbolic_implicit} takes the form
$$
 -  \log (\Psi(x,t) + 1 + t)  + \log \Psi(x,t) = \log x, \qquad 0<x<1.
$$
Solving it yields
$$
\Psi(x,t) = \frac {x(t+1)}{1-x}, \qquad 0<x<1.
$$
Differentiating, we arrive at the following simple solution:
$$
\psi(x,t) = \frac {t+1}{(1-x)^2}, \qquad 0<x<1, \;\; t\geq 0.
$$
Note that this solution makes sense for arbitrary $t\geq 0$.  The density increases under repeated differentiation, but this is not a contradiction since the total number of zeroes is infinite.

\vspace*{2mm}
\noindent
\textsc{Case $\alpha=1/2$} (Hyperbolic area measure).
The initial condition has the form
$$
u(z,0)
=
\frac{1}{\pi \,(1-|z|^{2})^2}, \qquad |z|<1,
\qquad
\psi(x,0) = \frac{2 x \ind_{\{x<1\}}}{(1-x^{2})^2},
\qquad x\geq 0,
$$
and corresponds to zeroes distributed according to  the hyperbolic area measure on the unit disk.
The implicit equation~\eqref{eq:Psi_hyperbolic_implicit} takes the form
$$
-\frac 12  \log (\Psi(x,t) + t) - \frac 12 \log (\Psi(x,t) + 1 + t)  + \log \Psi(x,t) = \log x, \qquad 0<x<1.
$$
Multiplying by $2$ and exponentiating yields
$$
\Psi^2(x,t) = x^2 (\Psi(x,t) + t)(\Psi(x,t) + 1 + t),  \qquad 0<x<1.
$$
Multiplying out, we arrive at the following quadratic equation for $\Psi(x,t)$:
$$
(1- x^2) \Psi^2(x,t) - (2t+1) x^2 \Psi(x,t) -  x^2 t(t+1) = 0.
$$
Solving it gives the distribution function of the radial parts
$$
\Psi(x,t)
=
\frac{(2t+1)x^2 + \sqrt{x^4 + 4 x^2 (t+1)t}}{2(1-x^2)},
\qquad
0 \leq  x <1,
$$
because the second solution, being negative, can be discarded.  
Differentiating, we obtain the following solution
$$
\psi(x,t) = \frac{x(2t+1)}{(1-x^2)^2} + \frac{x^2+2t^2+2t^2x^2+2t+2tx^2}{(1-x^2)^2  \sqrt{x^2 + 4 (t+1)t}},
\qquad
0 \leq  x <1.
$$
Note that this solution makes sense for arbitrary $t\geq 0$.

We omit the details in the case $\alpha=2$ which also can be solved explicitly.

\section{Real zeroes of repeated derivatives}\label{sec:real_zeroes}
\subsection{Recipe for real zeroes}\label{subsec:recipe_real}
We now explain how the above methods can be modified to treat polynomials with real zeroes only. As already mentioned in Section~\ref{subsec:real_zeroes}, the method proposed below need not be optimal and an approach based on finite free probability may be more natural.

Consider a sequence of monic (deterministic) polynomials $(Q_n)_{n\in\N}$ whose zeroes are real and belong to the interval $[-C, 0]$ for some constant $C>0$. Moreover, assume that the zeroes are distributed according to some finite measure $\mu_0$ concentrated on $[-C,0]$, that is
\begin{equation}\label{eq:real_zeroes_weak_conv}
\frac 1n\sum_{z\in \R: Q_n(z) = 0} \delta_z \toweak \mu_0,
\end{equation}
where $\toweaknon$ denotes weak convergence of finite measures.
One of the special cases we have in mind is when $\mu_0$ is a probability measure and the zeroes of $Q_n$ are $n$ i.i.d.\ random variables sampled according to $\mu_0$.  Then, \eqref{eq:real_zeroes_weak_conv} holds for a.e.\ realization of these random variables.   In general, $\mu_0$ need not be a probability measure and the degree of $Q_n$ is $mn + o(n)$, where $m= \mu_0 ([-C, 0])$ is the total mass of measure $\mu_0$.

In the following we shall derive a formula for the asymptotic distribution of zeroes of the $[tn]$-th derivative of $Q_n$, where $0\leq t<m$.  Note that by Rolle's theorem, all zeroes of all derivatives of $Q_n$ stay real and do not leave the interval $[-C,0]$.
The first step is to relate the distribution of zeroes of $Q_n$ to the exponential profile of the coefficients of the polynomial $Q_n$, defined (up to sign) by
\begin{equation}\label{eq:real_profile_v_def}
v(\alpha) = -\lim_{n\to\infty} \frac 1n \log \left([x^{[\alpha n]}] Q_n(x)\right), \qquad 0 < \alpha < m.
\end{equation}
Here, $[x^k]Q_n(x)$ is the coefficient of $x^k$ in $Q_n(x)$. First of all, observe that for $x>0$ we have
$$
\frac 1n \log Q_n(x)
=
\frac 1n \log \sum_{\alpha\in \{0,\frac 1n,\frac 2n, \ldots\}} \eee^{-n v(\alpha) + o(n)} x^{\alpha n}
=
\frac 1n \log \sum_{\alpha\in \{0,\frac 1n,\frac 2n, \ldots\}} \eee^{n (\alpha \log x - v(\alpha) +o(1))}.
$$
Letting $n\to\infty$ and assuming that the Laplace asymptotics can be justified, we obtain
\begin{equation}\label{eq:lim_log_Q_n_1}
\lim_{n\to\infty} \frac 1n \log Q_n(x)   = I(\log x), \qquad x>0,
\end{equation}
where $I$ is the Legendre transform of $v$ defined by
\begin{equation}
I(s) = \sup_{\alpha\in [0,m]} (s\alpha - v(\alpha)), \qquad s\in \R.
\end{equation}
On the other hand, writing $Q_n(x)$ as a product of the terms  $(x-z)$, where $z$ runs through all zeroes of $Q_n$ counting multiplicities, and passing  to logarithmic potentials by means of~\eqref{eq:real_zeroes_weak_conv}, we have that for all $x>0$,
\begin{equation}\label{eq:lim_log_Q_n_2}
\frac 1n \log Q_n(x) =  \frac 1n \sum_{z\leq 0: Q_n(z) = 0}  \log (x - z) \ton \int_{-C}^0 \log (x-z) \mu_0(\dd z).
\end{equation}
Comparing~\eqref{eq:lim_log_Q_n_1} and~\eqref{eq:lim_log_Q_n_2}, we get
$$
I(\log x) = \int_{-C}^0 \log (x-z) \mu_0(\dd z), \qquad x>0.
$$
Differentiating in $x$, we arrive at
\begin{equation}\label{eq:I'_G}
I'(\log x) = x G_0(x), \qquad x>0,
\end{equation}
where $G_0$ is the Cauchy-Stieltjes transform of $\mu_0$ given by
\begin{equation}\label{eq:def_f_0_real}
G_0(x) = \int_{-C}^0  \frac {\mu_0 (\dd u)} {x-u}, \qquad x\in \C\backslash[-C,0].
\end{equation}
It is convenient to put $x = w_0(y)$ with $w_0(y):= \eee^{v'(y)}$. Since  the functions $I'$ and $v'$ are inverse to each other, we arrive at the identity
\begin{equation}\label{eq:w_0_inverse_real}
y = w_0(y) G_0(w_0(y)), \qquad \mu_0(\{0\}) < y < m.
\end{equation}
The above derivation was non-rigorous, but the same equation has been established rigorously by Van Assche, Fano and Ortolani~\cite{van_assche_fano_ortolani}, see also~\cite[Theorem~5.1]{van_assche_book}. These authors  also established that the function $v(\alpha)$ is convex (which also follows from Newton's real roots theorem), differentiable and finite in the range $\mu_0(\{0\}) < \alpha <  m$. Since $\lim_{x\to +\infty} x G_0(x) = m$, it follows from~\eqref{eq:I'_G} that $\lim_{\alpha \uparrow m} v'(\alpha)=+\infty$ and hence $\lim_{\alpha \uparrow m} w_0(\alpha) =  +\infty$.
For $0 < \alpha < \mu_0(\{0\})$ the definition~\eqref{eq:real_profile_v_def} suggests to put  $v(\alpha) = +\infty$.  A more refined, distributional  result for the coefficients of $Q_n$  in the setting of i.i.d.\ zeroes has been obtained by Major~\cite[Theorem~1]{major}. For a closely related work see also~\cite{fano_gallavotti}.

Consider now the $[tn]$-th derivative of $Q_n$, where $0\leq t <m$. Its exponential profile is defined (up to sign) by
$$
v(\alpha,t) = - \lim_{n\to\infty} \frac 1n \log \left([x^{[\alpha n]}] Q_n^{([tn])}(x)\right), \qquad 0< \alpha < m-t.
$$
As we already know from~\eqref{eq:v_x_t_def}, $v(\alpha,t)$ is related to the exponential profile $v(\alpha) = v(\alpha,0)$ of $Q_n$ by
$$
v(\alpha, t) = v(\alpha+t) -  (\alpha+t) \log (\alpha+t)  + \alpha \log \alpha, \qquad 0< \alpha < m-t.
$$
Taking the derivative in $\alpha$ and then the exponential, we arrive at
$$
\eee^{\partial_1 v(\alpha, t)} = \eee^{v'(\alpha + t)} \cdot \frac{\alpha}{\alpha + t}, \qquad 0 < \alpha <  m-t.
$$
With the notation $w_t(\alpha):= \eee^{\partial_1 v(\alpha, t)}$, this takes the form
\begin{equation}\label{eq:w_t_w_0}
w_t(\alpha) = w_0(\alpha + t) \cdot \frac{\alpha}{\alpha + t}.
\end{equation}
Finally, repeating the above argument backwards, we arrive at the following analogue of~\eqref{eq:w_0_inverse_real} at time $t$:
\begin{equation}\label{eq:w_t_inverse_real}
y = w_t(y) G_t(w_t(y)).
\end{equation}
where $G_t$ is the Cauchy-Stieltjes transform of $\mu_t$. This identity holds in the range $\mu_0(\{0\})-t < y < m-t$, $y>0$.

We can now state the recipe for computing the distribution of zeroes at time $t$. Compute the Cauchy-Stieltjes transform $G_0$ of the initial distribution  $\mu_0$ by means of~\eqref{eq:def_f_0_real}. Compute $w_0$ by means of~\eqref{eq:w_0_inverse_real}, $w_t$ by means of~\eqref{eq:w_t_inverse_real}, and finally $G_t$ by means of~\eqref{eq:w_t_inverse_real}.  The Cauchy-Stieltjes transform  $G_t$ is an analytic function on $\C\backslash\supp \mu_t$, and the distribution $\mu_t$ of zeroes at time $t$ can be computed by means of the Stieltjes inversion formula~\cite[Section~3.1, p.~93]{hiai_petz_book}, namely as the following weak*-limit:
\begin{equation}\label{eq:stieltjes}
\mu_t(\dd x) =  -\frac 1 \pi \wlim\lim_{y\downarrow 0 } \Im G_t(x+\ii y).
\end{equation}
In the next section we shall apply this recipe to the special case when $\mu_0$ is a  combination of two delta-measures.
As we shall argue in a moment, the above derivation was, in fact, rigorous. We summarize our findings in the following
\begin{theorem}\label{theo:real_recipe}
Consider a sequence of deterministic monic  polynomials $(Q_n)_{n\in\N}$ whose zeroes belong to the interval $[-C, 0]$ for some constant $C>0$, and satisfy
\begin{equation*}
\frac 1n \sum_{z\in \R: Q_n(z) = 0} \delta_z \toweak \mu_0
\end{equation*}
in the sense of weak convergence of finite measures on $\R$,  for some finite measure $\mu_0$ on $[-C,0]$ with total mass $m>0$. Then, for every $0\leq t <m$ we have
$$
\frac 1n \sum_{z\in \R: Q_n^{([tn])}(z) = 0} \delta_z \toweak \mu_t,
$$
where $\mu_t$ is a finite measure on $[-C,0]$ whose Cauchy-Stieltjes transform $G_t$ satisfies~\eqref{eq:w_t_inverse_real}, \eqref{eq:w_t_w_0} and~\eqref{eq:w_0_inverse_real}.
\end{theorem}
\begin{proof}
The zeroes of $Q_n^{([tn])}$ are contained in $[-C,0]$ by Rolle's theorem. By Helly's compactness theorem,  there is a subsequential limit of the empirical measures of its zeroes. By the result of~\cite{van_assche_fano_ortolani}, the Cauchy-Stieltjes transform $G_t$ of any such subsequential limit satisfies~\eqref{eq:w_t_inverse_real}. Since $\lim_{y\uparrow m-t}w_t(y) = +\infty$ (which follows from~\eqref{eq:w_t_w_0} and the similar property of $w_0$ established above), condition~\eqref{eq:w_t_inverse_real} defines $G_t(x)$ uniquely if $x>0$ is sufficiently large. By the uniqueness of analytic continuation, the Cauchy-Stieltjes transforms of all subsequential limits coincide on $\C \backslash [-A,0]$. By the Stieltjes inversion formula~\eqref{eq:stieltjes}, this allows to conclude that all subsequential limits are equal and satisfy~\eqref{eq:w_t_inverse_real}.
\end{proof}

\subsection{Example: Polynomials with two zeroes of high multiplicity}
Fix some parameters $m_1>0$ and $m_2>0$ and consider a polynomial of the form
$$
Q_n(x) = x^{m_1 n +o(n)} (x+1)^{m_2 n + o(n)}.
$$
We claim that the asymptotic density  $u(x,t)$ of zeroes of the $[tn]$-th derivative of this polynomial is given by
\begin{multline}\label{eq:u_x_t_real}
u(x,t)
=
\frac{(m_1+m_2) \sqrt{(x_+(t)-x)(x-x_-(t))}}{2\pi |x| (1+x)} \ind_{\{x_-(t) < x < x_+(t)\}}
\\+
(m_1-t) \ind_{\{0\leq t<m_1\}} \delta(x)  + (m_2-t) \ind_{\{0\leq t<m_2\}} \delta(x+1),
\end{multline}
for all $0\leq t < m_1+m_2$, where
\begin{equation}\label{eq:x_+-}
x_\pm(t) = \frac{(t-m_1)m_1 - m_2 (t+m_1) \pm 2\sqrt{m_1m_2 t(m_1+m_2 -t)}}{(m_1+m_2)^2}.
\end{equation}
The normalization is chosen such that $\int_{\R} u(x,t) \dd x = m_1+m_2 -t$.

Before deriving this formula for $u(x,t)$, let us discuss some of its properties; see Figures~\ref{pic:two_zeroes} and~\ref{pic:arcsine}  which show plots of $u(x,t)$ in two special cases.
The presence of atoms  at $0$, respectively, $1$, in~\eqref{eq:u_x_t_real} as long as $t<m_1$, respectively, $t<m_2$,  is not surprising and is due to the fact that  $Q_n(x)$ has multiple zeroes at these points. The multiplicities of these zeroes decrease under repeated differentiation until the zeroes disappear. Regarding the behavior of the interval  $[x_-(t), x_+(t)]$ on which the continuous part of $u(x,t)$ is supported, we can make the following remarks. At times $t=0$ and $t=m_1+m_2$, we have
\begin{equation}\label{eq:x_+-0_x+-final}
x_-(0) = x_+(0) = -\frac{m_1}{m_1+m_2},
\qquad
x_-(m_1+m_2) = x_+(m_1+m_2) = - \frac{m_2}{m_1+m_2}.
\end{equation}
Both for $t\approx 0$ and $t\approx m_1+m_2$, the absolutely continuous part of $u(x,t)$ can be approximated by a Wigner semicircle law on the small interval $[x_-(t), x_+(t)]$; see Figure~\ref{pic:two_zeroes}. To explain this, note that the function $|x|(1+x)$ appearing in the denominator of~\eqref{eq:u_x_t_real} can be approximated by a constant on this small interval. The appearance of the Wigner law is discussed in~\cite{hoskins_steinerberger}.  At time $t=m_1$, we have $x_+(m_1) = 0$, and $u(x,m_1)\sim \text{const}/\sqrt {-x}$ has a singularity as $x\uparrow 0$. On the other hand, with some effort it is possible to check that $x_+(t) <0$ for $t\neq m_1$. Similarly, we have $x_-(m_2) = -1$ with $u(x,m_2)\sim \text{const}/\sqrt{1+x}$ becoming singular as $x\downarrow -1$,   and $x_-(t) >-1$ for $t\neq m_2$; see Figure~\ref{pic:two_zeroes}.  In general, we have $-1\leq x_-(t) <  x_+(t) \leq 0$, which is not surprising in view of Rolle's theorem implying that all zeroes are contained in $[-1,0]$. Also, one can check that a given $x\in (-1,0)$ belongs to the support of the absolutely continuous part of $u(x,t)$ if and only if $t\in [t_-(x), t_+(x)]$, where
$$
t_\pm(x) = m_1 + x(m_1-m_2) \pm 2\sqrt {m_1m_2 |x| (x+1)}.
$$
For example, the points $-\frac{m_1}{m_1+m_2}$ and $-\frac{m_2}{m_1+m_2}$ appearing in~\eqref{eq:x_+-0_x+-final} belong to the support if and only if $0\leq t \leq \frac{4m_1m_2}{(m_1+m_2)}$, respectively $\frac{(m_1-m_2)^2}{m_1+m_2} \leq t \leq m_1+m_2$.

Finally, let us mention that the solution $u(x,t)$ has the following symmetry  modulo delta-functions:
$$
u(x,t) = u(-x-1, m_1+m_2-t) \pm \text{$\delta$-functions}.
$$

\begin{figure}[!tbp]
\includegraphics[width=0.32\textwidth]{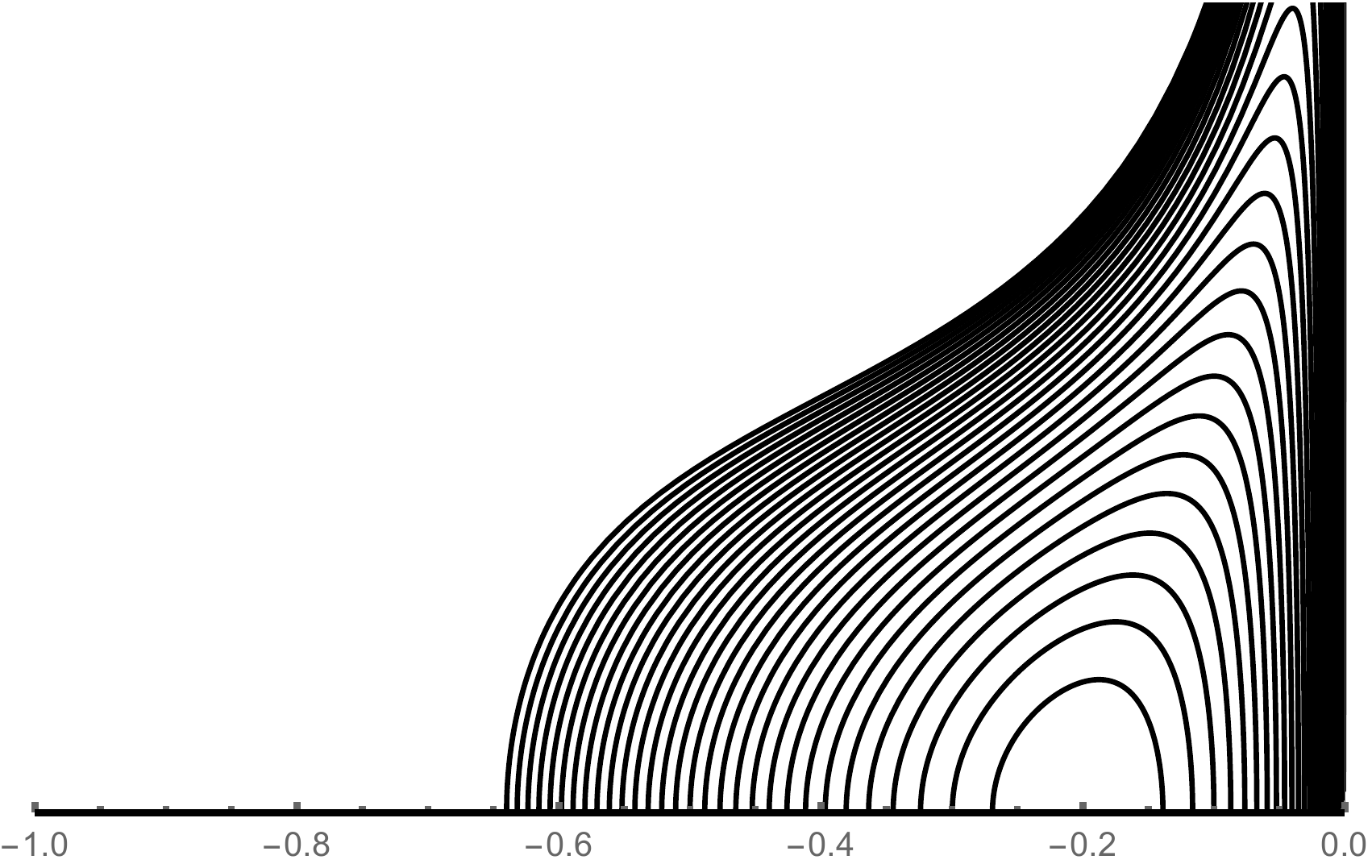}
\includegraphics[width=0.32\textwidth]{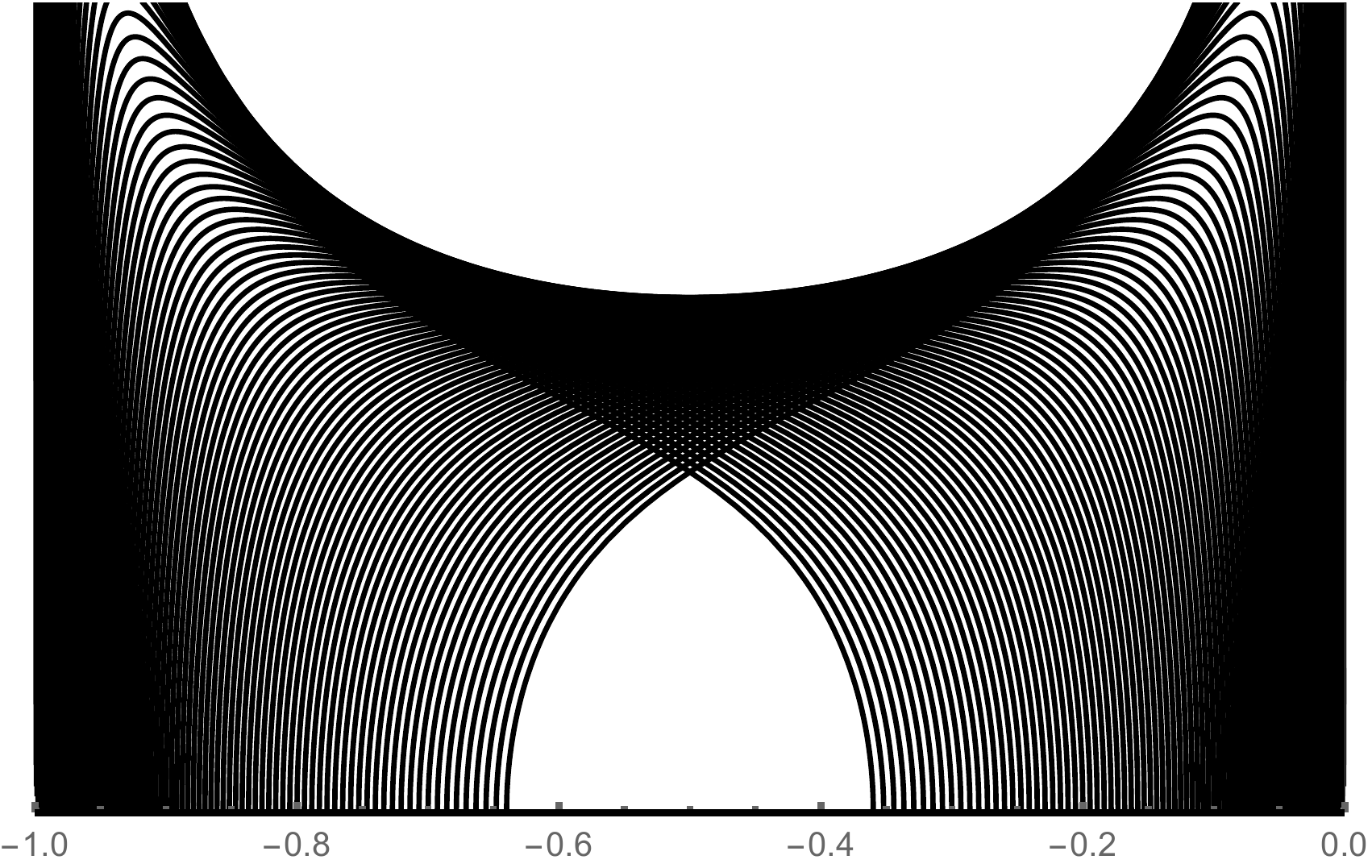}
\includegraphics[width=0.32\textwidth]{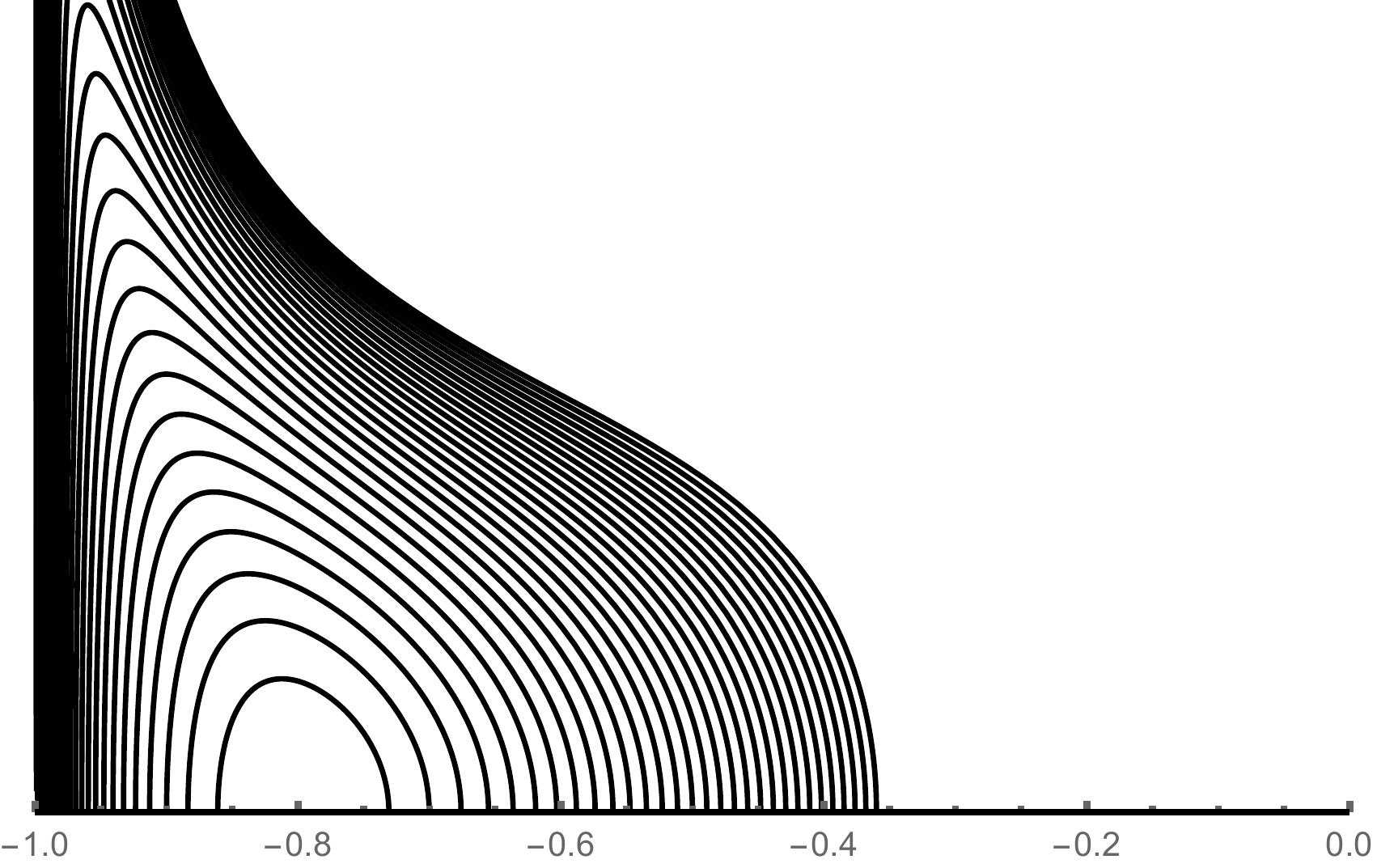}
\caption{The function $u(x,t)$ given by~\eqref{eq:u_x_t_real} with $m_1=1$, $m_2=4$ and $t\in \{\frac 1 {30}, \frac 2 {30}, \ldots, 5-\frac 1 {30}\}$.  The atoms at $0$ and $1$ are not shown.
Left: $0< t\leq 1$. Middle: $1\leq t \leq 4$. Right: $4\leq t <5$.
}
\label{pic:two_zeroes}
\end{figure}

\begin{figure}[!tbp]
\includegraphics[width=0.24\textwidth]{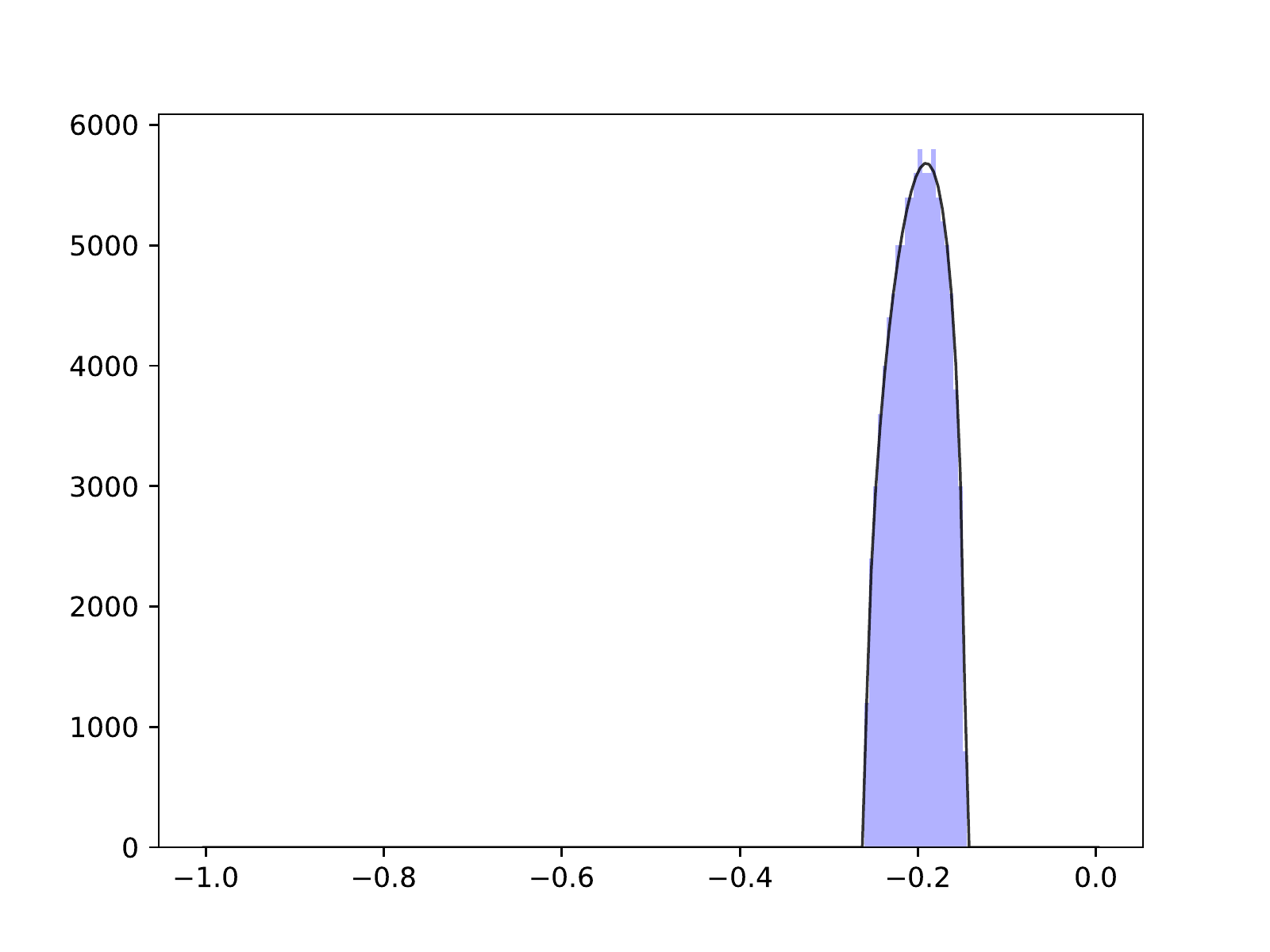}
\includegraphics[width=0.24\textwidth]{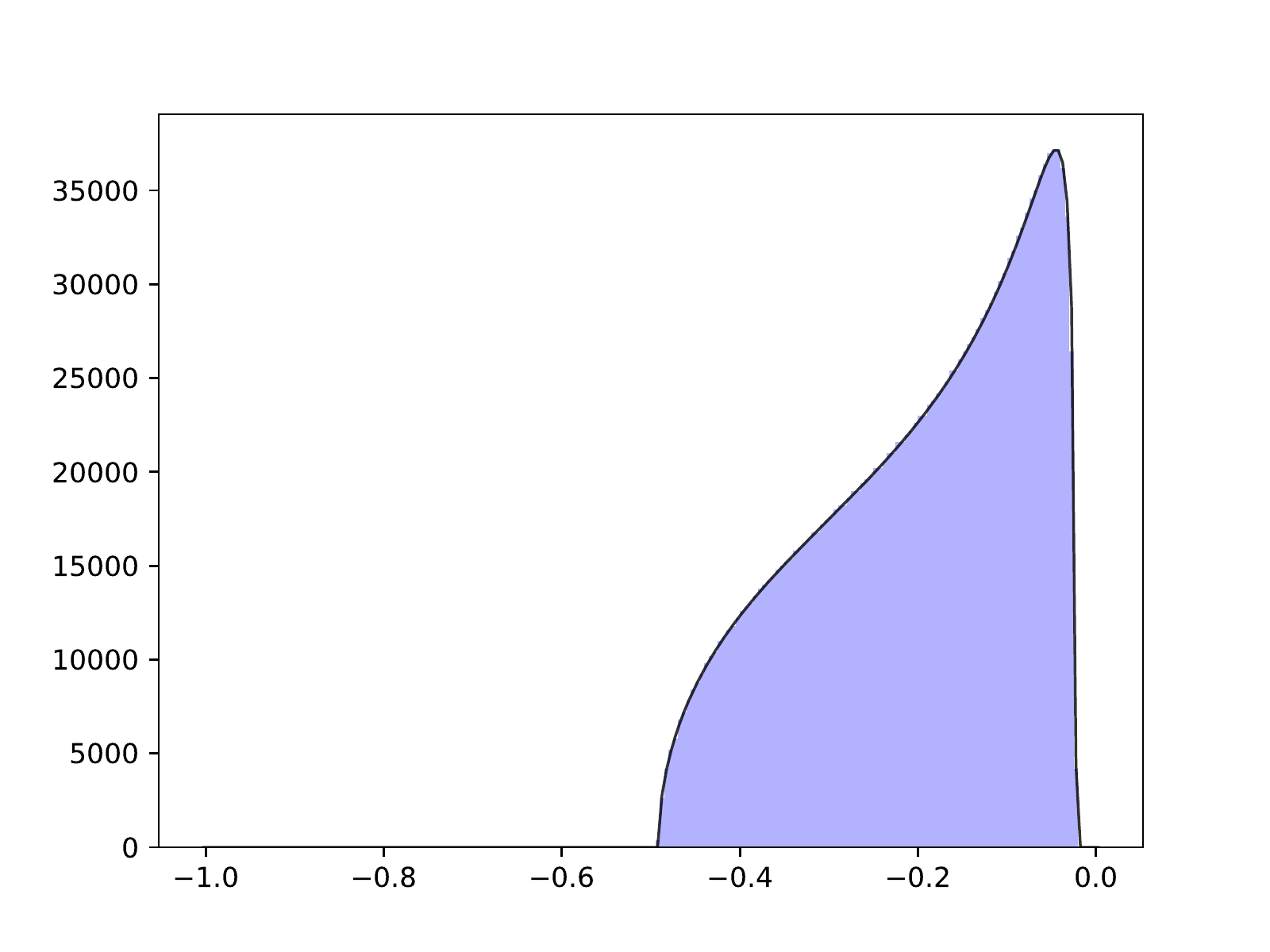}
\includegraphics[width=0.24\textwidth]{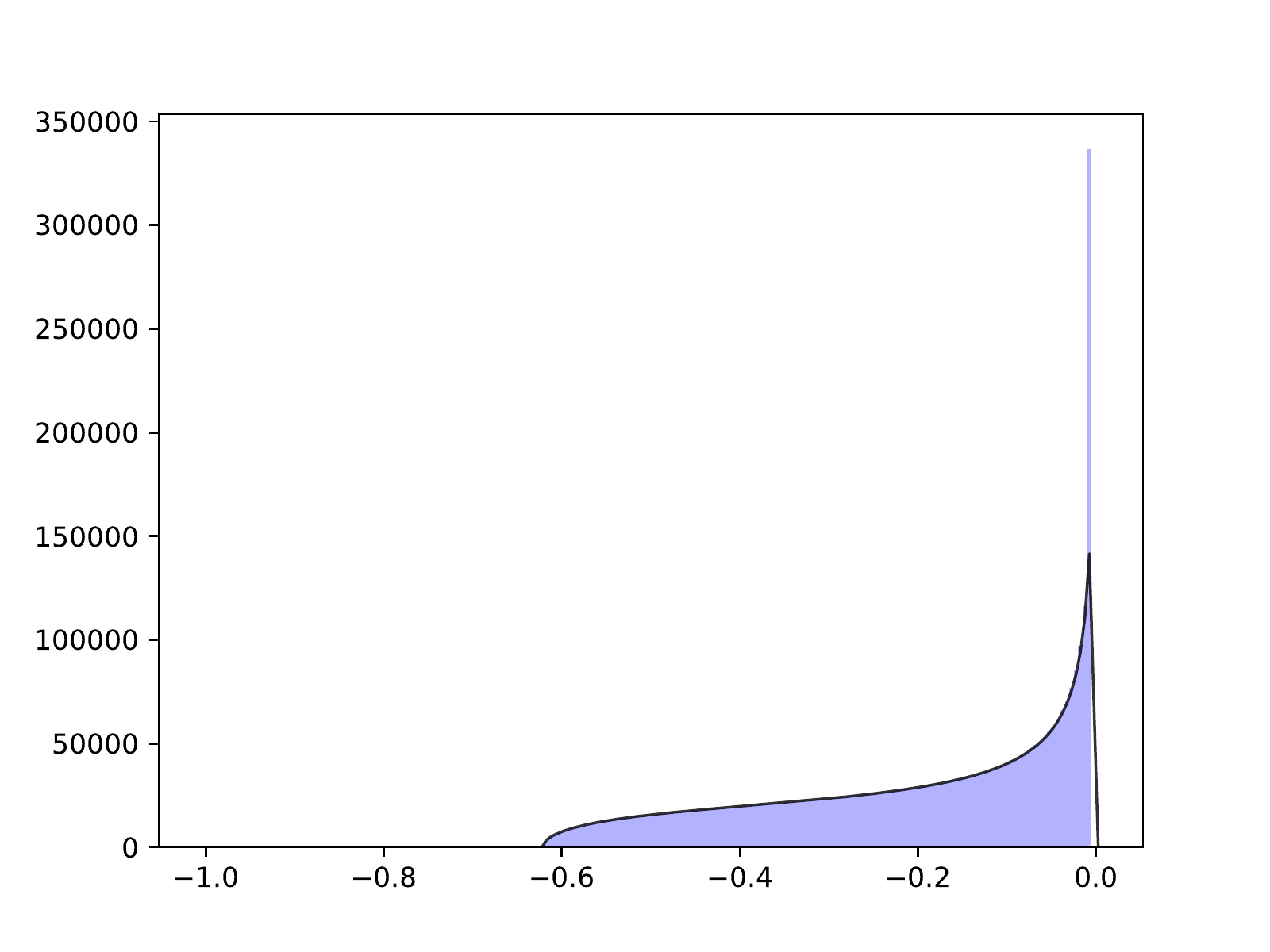}
\includegraphics[width=0.24\textwidth]{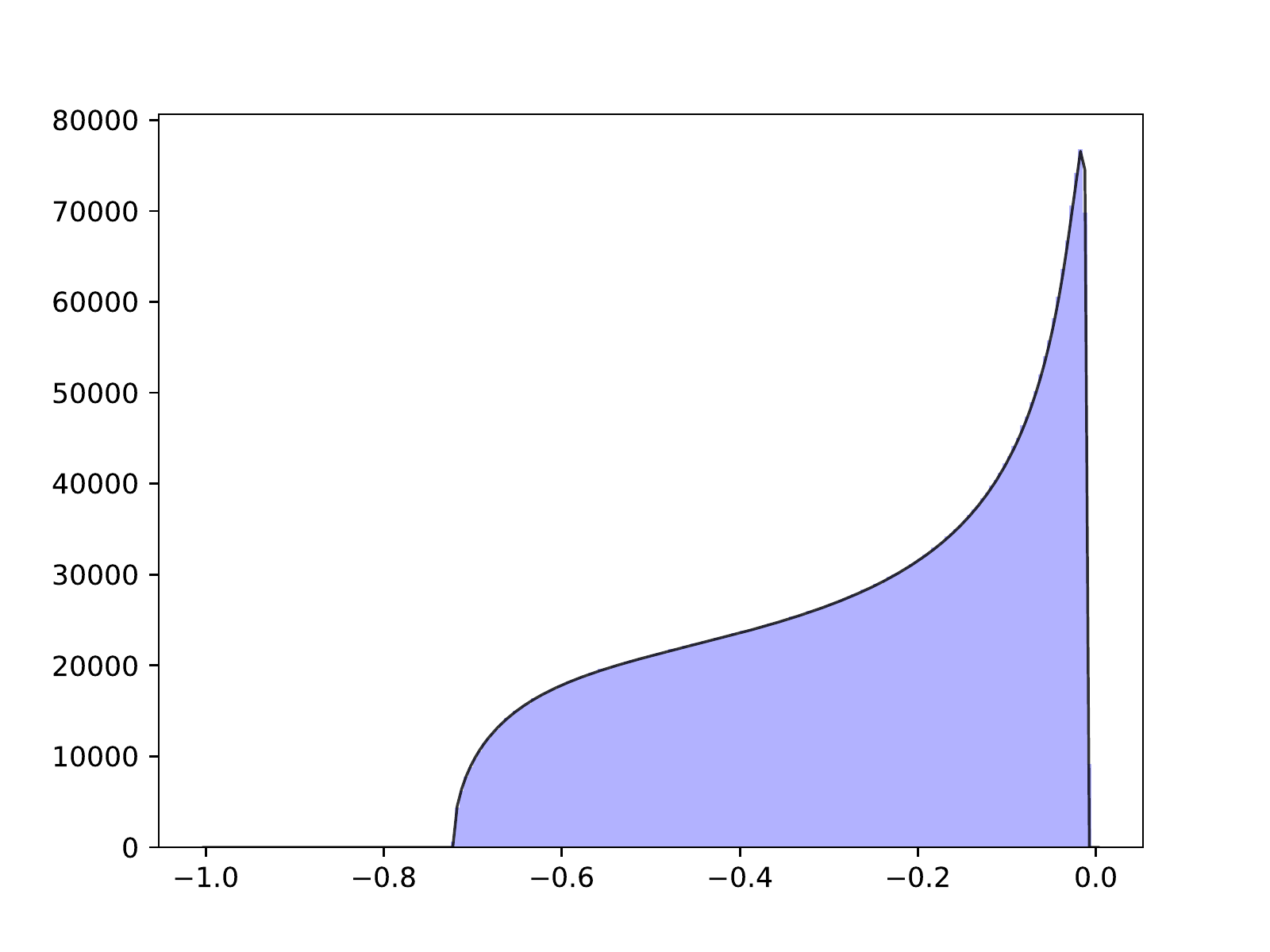}
\includegraphics[width=0.24\textwidth]{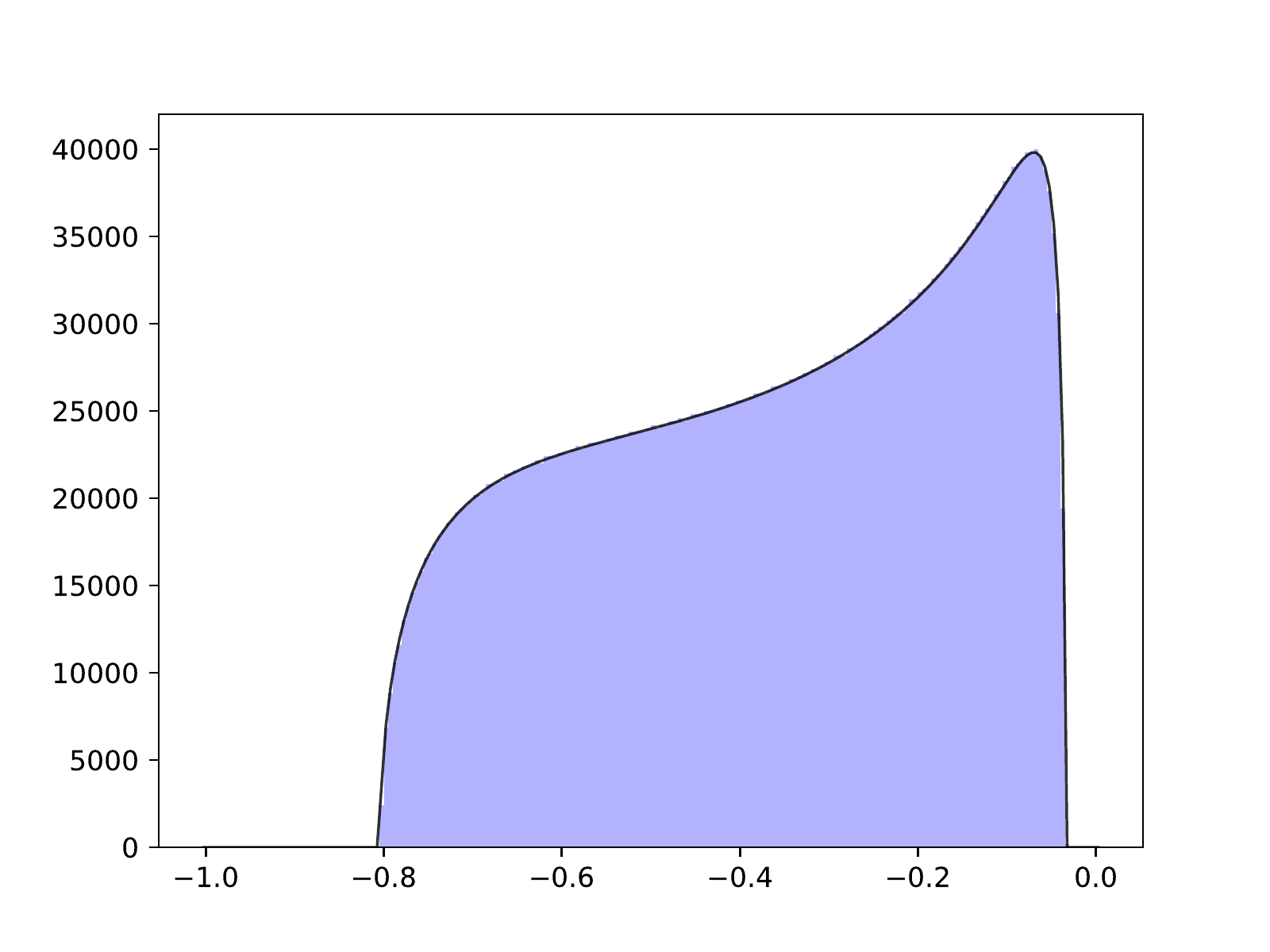}
\includegraphics[width=0.24\textwidth]{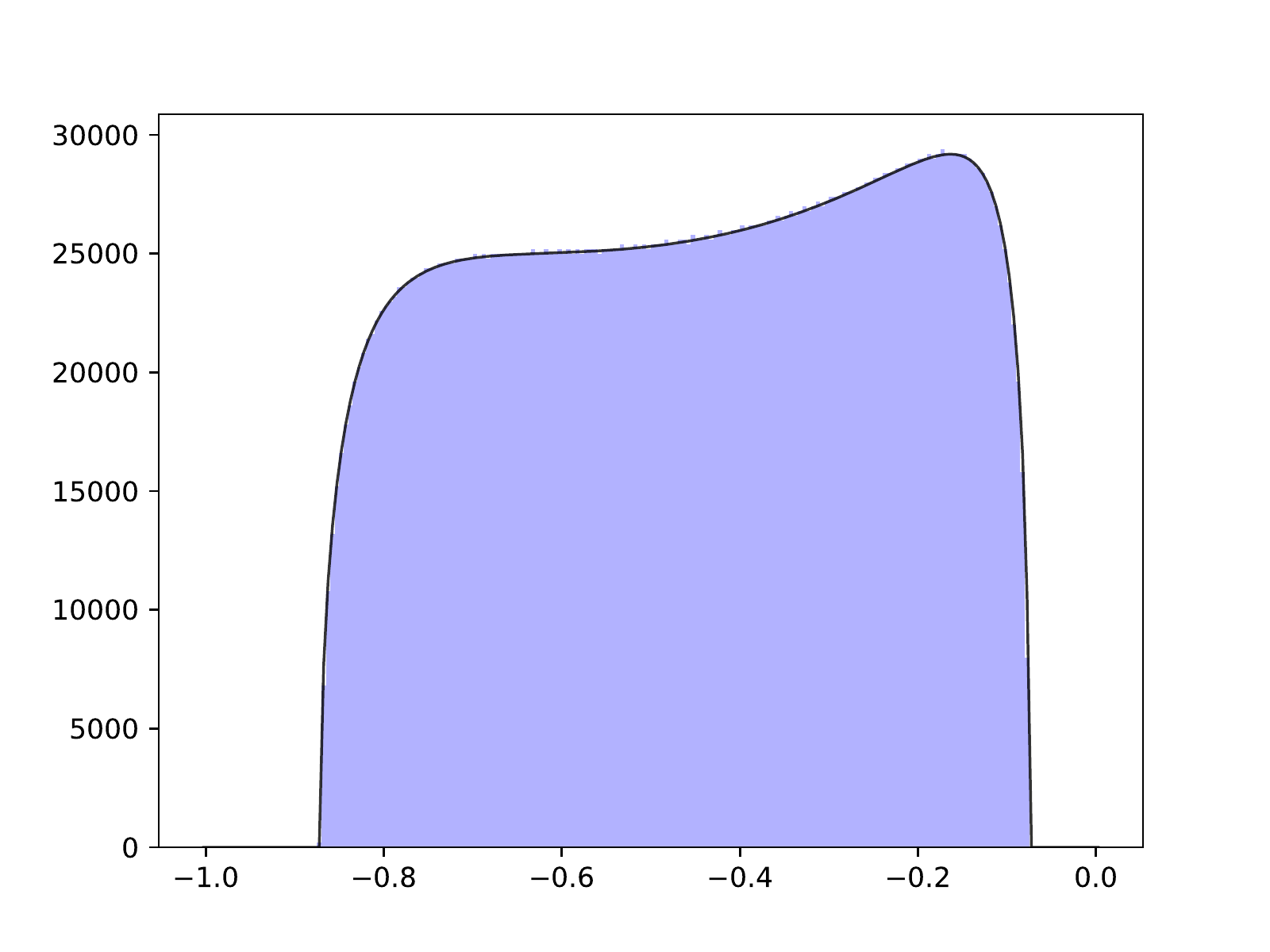}
\includegraphics[width=0.24\textwidth]{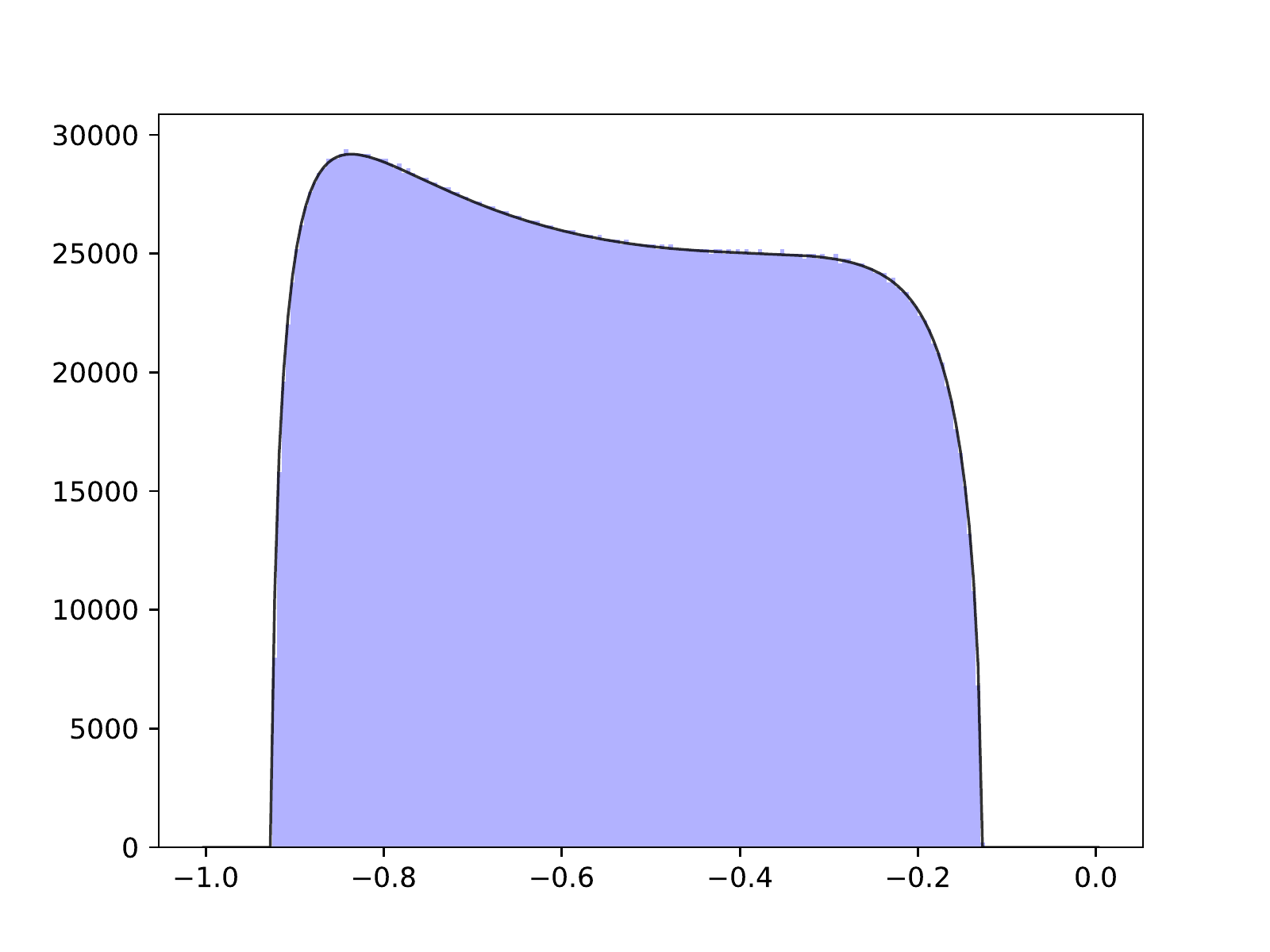}
\includegraphics[width=0.24\textwidth]{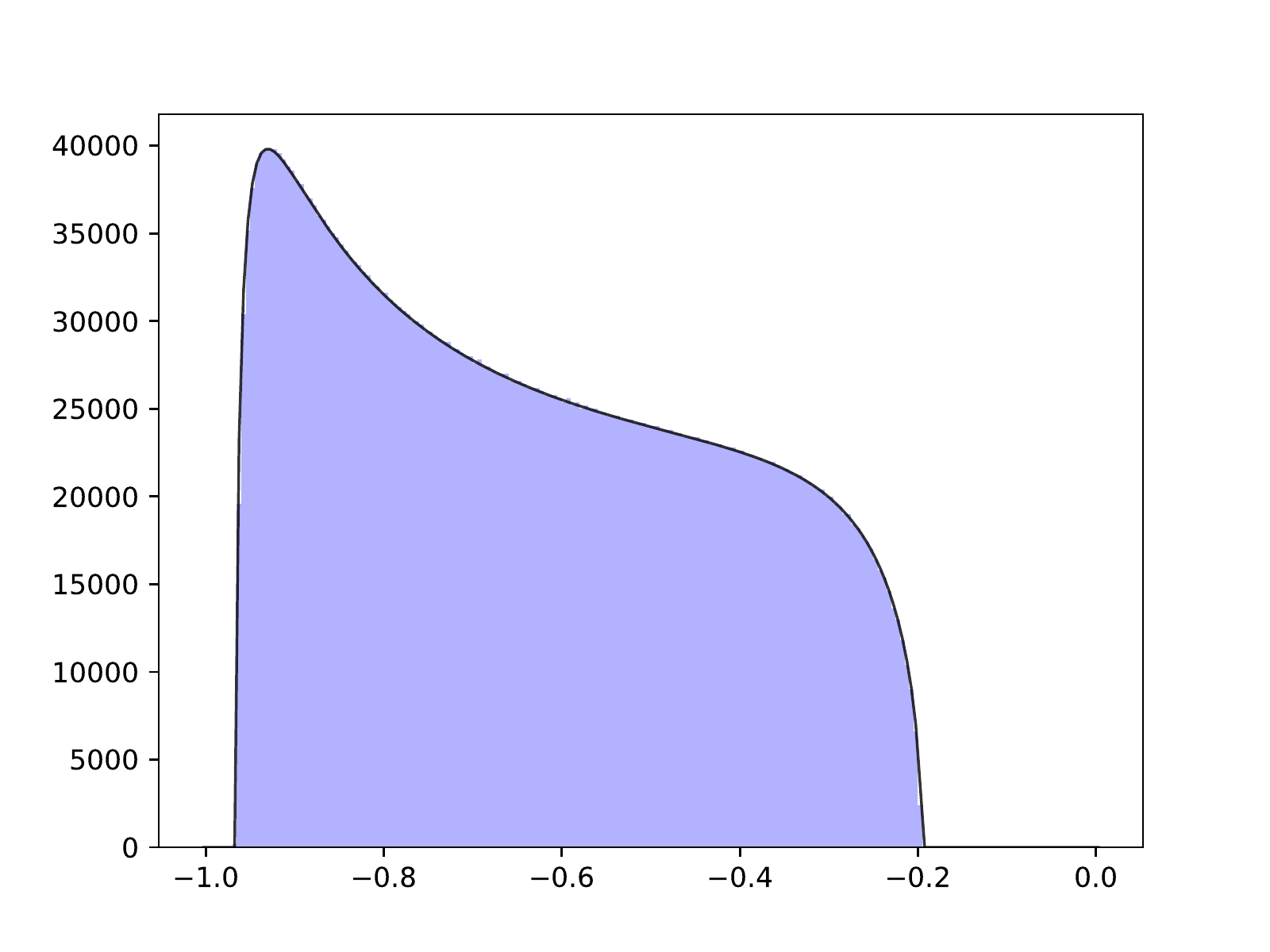}
\includegraphics[width=0.24\textwidth]{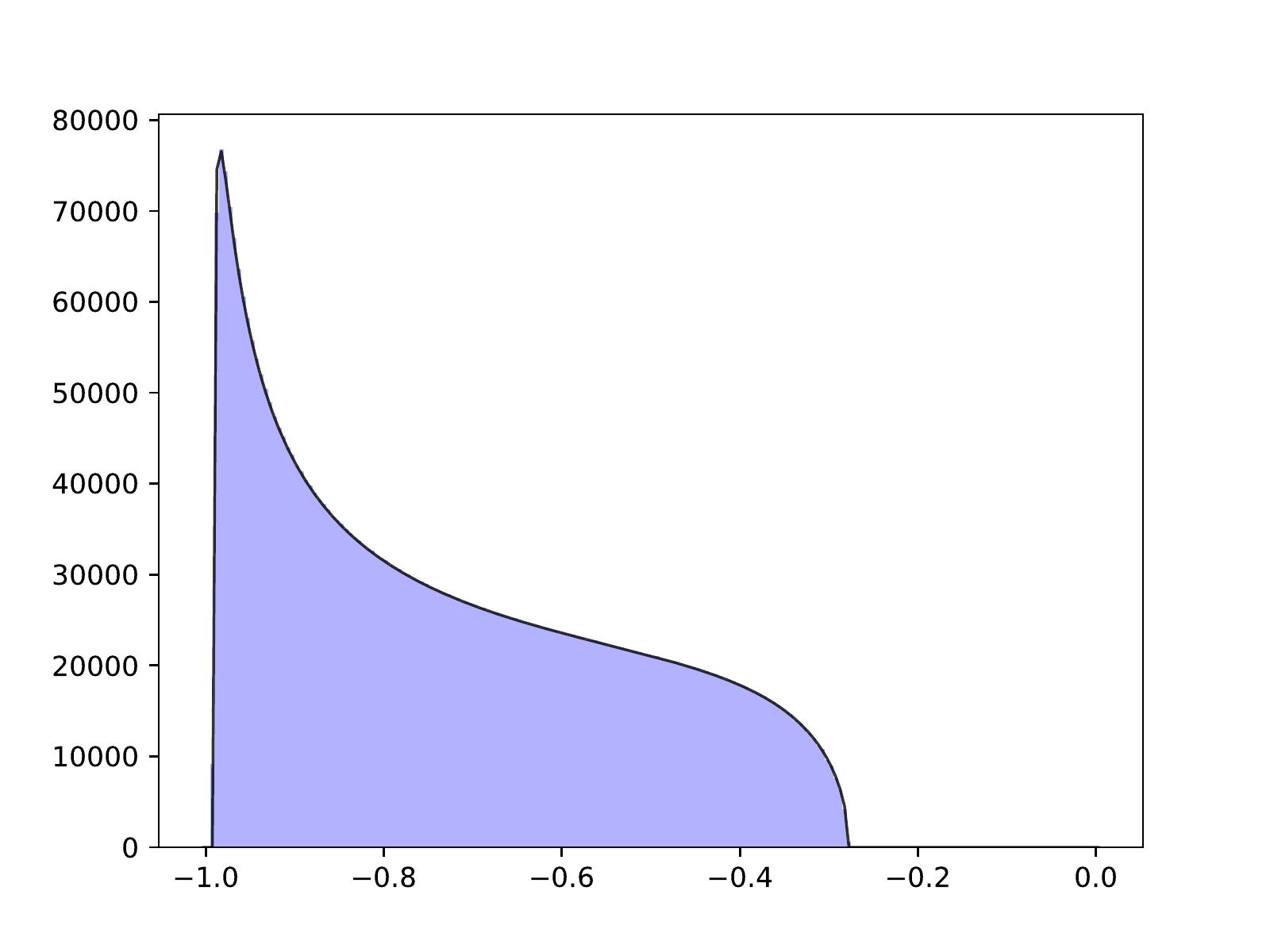}
\includegraphics[width=0.24\textwidth]{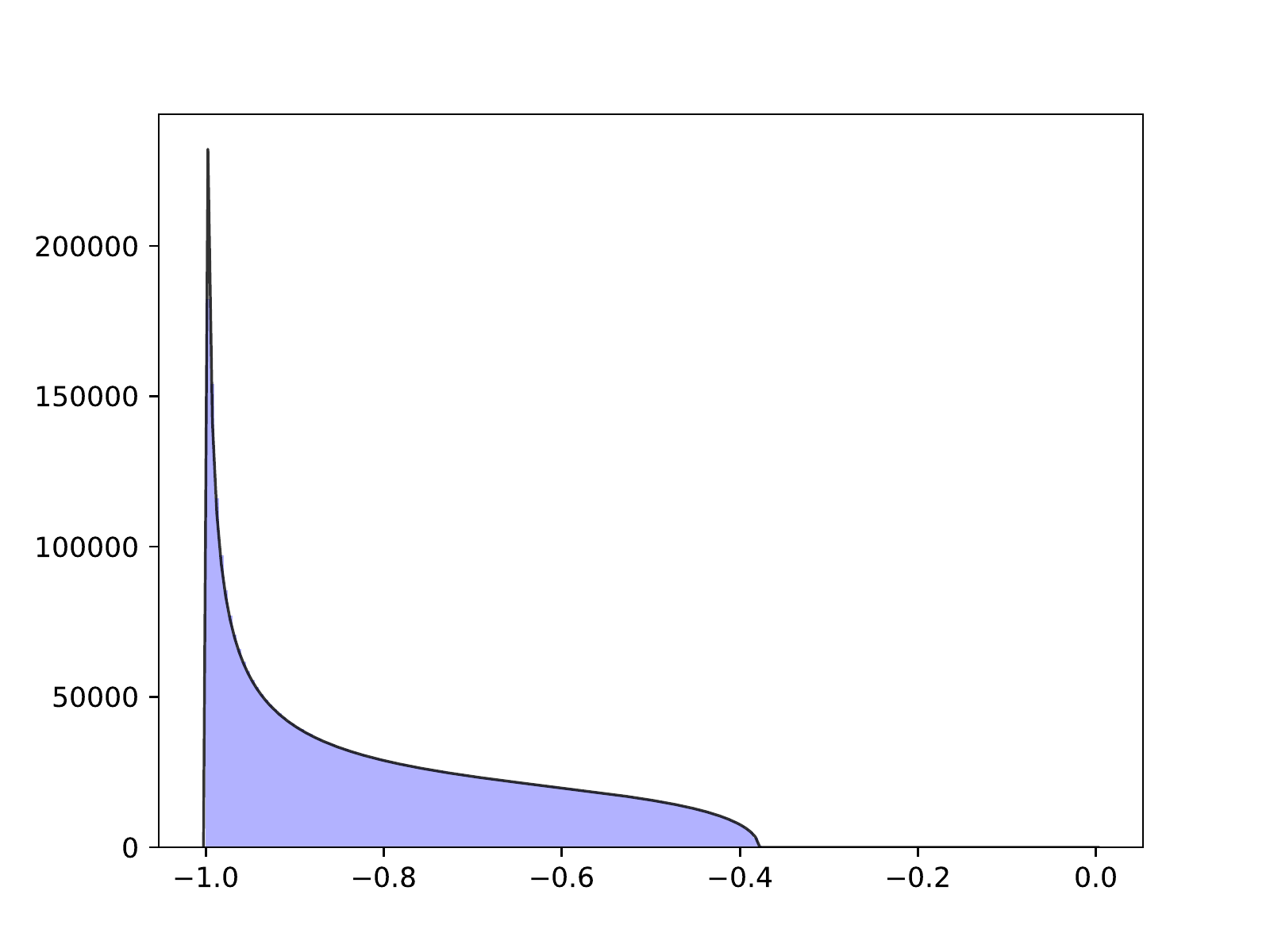}
\includegraphics[width=0.24\textwidth]{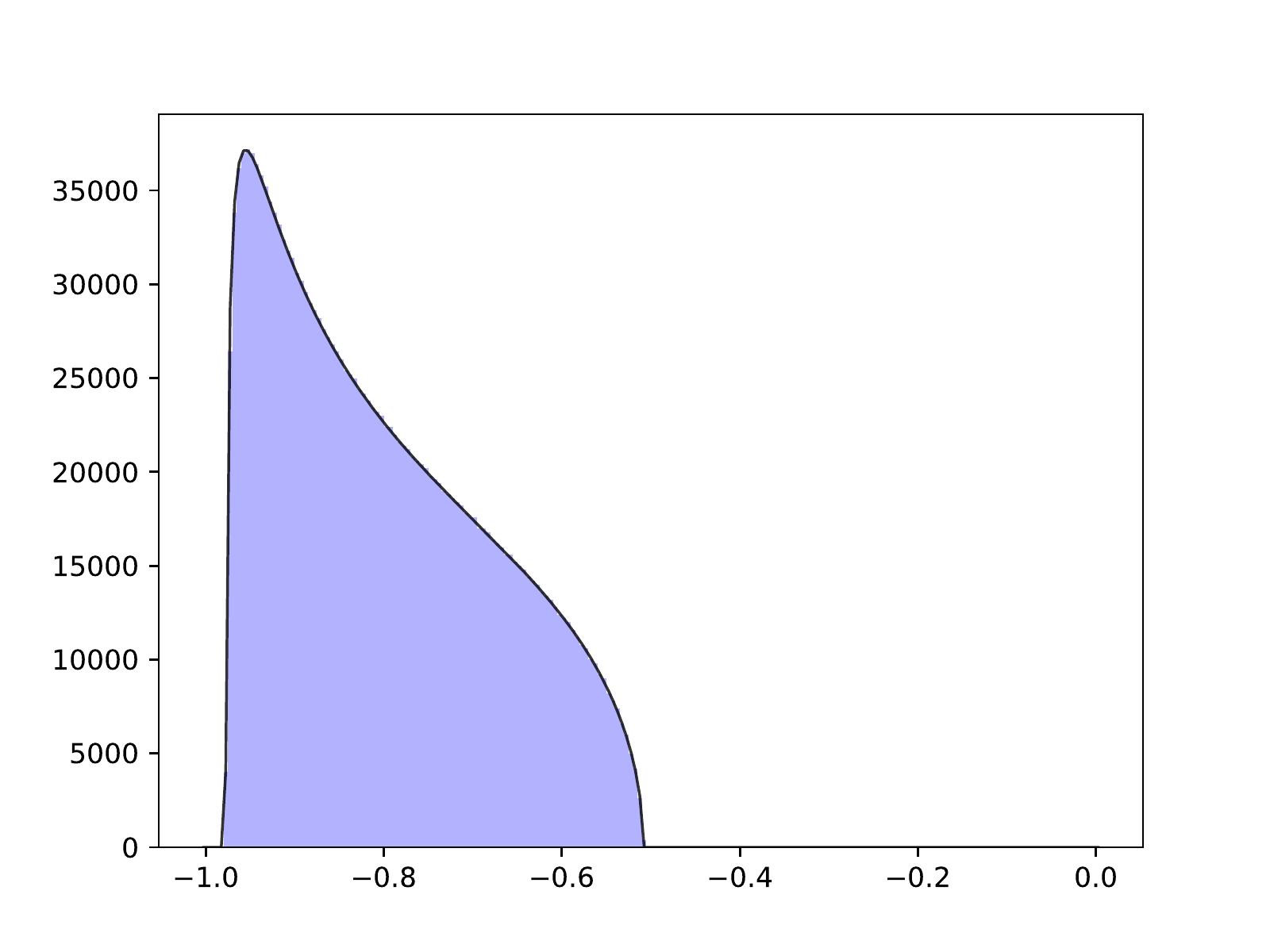}
\includegraphics[width=0.24\textwidth]{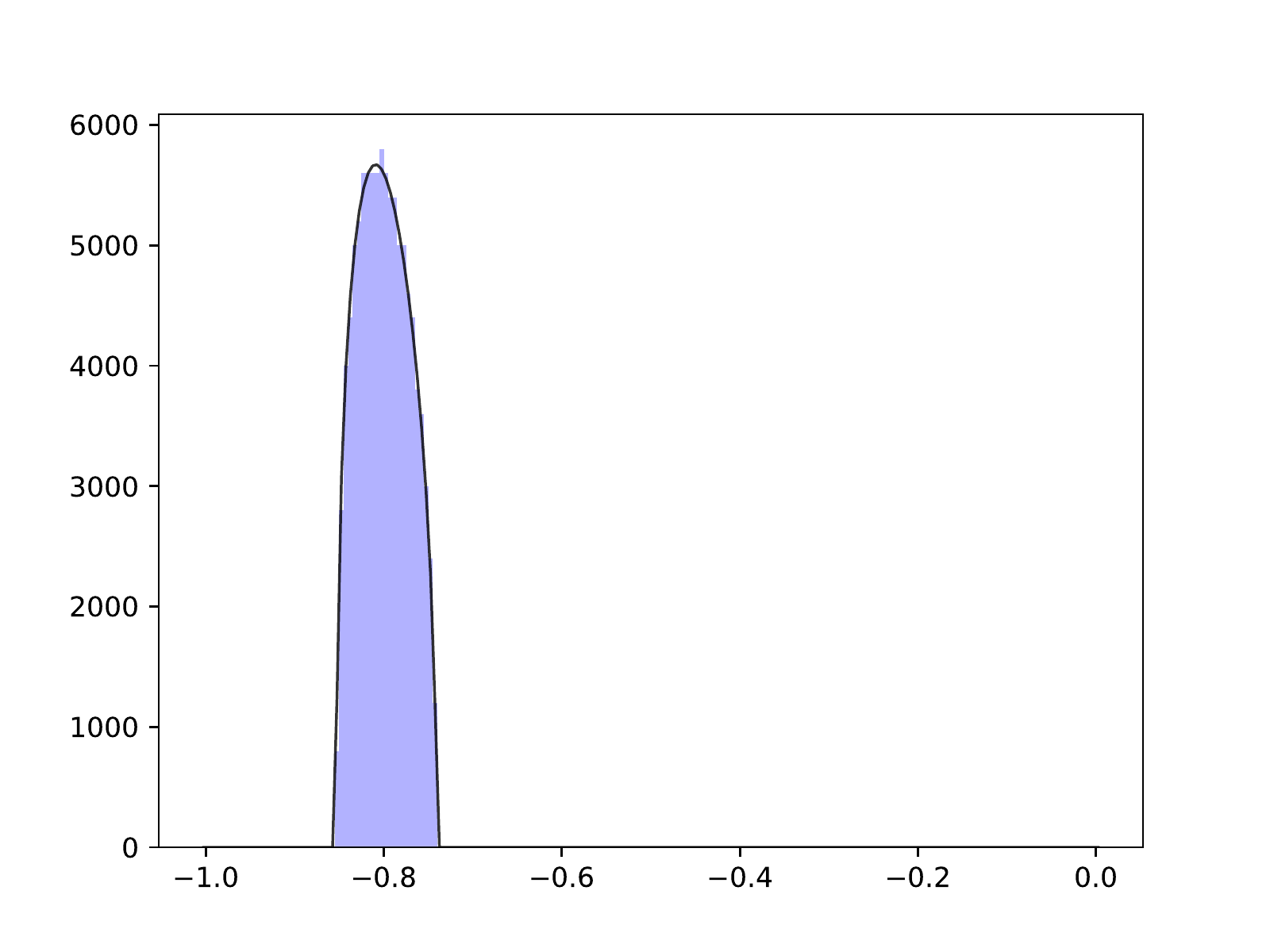}
\caption{
Histograms of zeroes of the repeated derivatives of the polynomial $x^n(1+x)^{4n}$ with $n=20000$ (blue) together with the theoretical densities (black). Multiple zeroes at $-1$ and $0$ are not shown. The orders of the derivatives are $100+1800k$ with  $k\in \{0,\ldots,11\}$.
}
\label{pic:two_zeroes_histogram}
\end{figure}

Let us now derive~\eqref{eq:u_x_t_real} by following the recipe described in  Section~\ref{subsec:recipe_real}. The initial condition is
$$
u(x,0) = m_1 \delta(x) + m_2\delta(x+1).
$$
The Cauchy-Stieltjes transform of $u(x,0)$ is
$$
G_0(x) = \frac {m_1}{x} + \frac{m_2}{x+1}.
$$
Solving~\eqref{eq:w_0_inverse_real} we obtain
$$
w_0(y) = \frac{y-m_1}{m_1 + m_2 - y}, \qquad  m_1 <  y < m_1+m_2.
$$
It follows from~\eqref{eq:w_t_w_0} that
$$
w_t(y) = \frac{(y + t - m_1)y}{(m_1 + m_2 - y - t)(y+t)}, \qquad m_1-t <  y  <  m_1+m_2 - t, \; y>0.
$$
Inserting this into~\eqref{eq:w_t_inverse_real} and solving quadratic equation, we arrive at
$$
G_t(x) = \frac{m_1 - t  + x (m_1+m_2-2t) + \sqrt{(t-m_1-x(m_1+m_2))^2 + 4 x t m_2}}{2x(1+x)}.
$$
The other solution of the quadratic equation can be ignored since the Cauchy-Stieltjes transform must have the property
$$
G_t (x) \sim  \frac{m_1 + m_2 - t}{x}, \quad \text{ as } x\to+\infty.
$$
Depending on the value of $t$, the function $G_t$ can have poles at $x=0$ and $x=-1$. Using the formula for $G_t(x)$, one easily checks a pole at $x=0$ (respectively, $x=-1$) exists provided $0<t<m_1$  (respectively, $0\leq t <m_2$), and the residues are given by
$$
\Res\limits_{x=0} G_t(x) = (m_1 - t)\ind_{\{0<t<m_1\}},
\qquad
\Res\limits_{x=0} G_t(x) = (m_2 - t)\ind_{\{0<t<m_2\}}.
$$
These two poles of $G_t(x)$, if they are present, correspond to the atoms of $u(x,t)$, the residues being their weights~\cite[p.~94]{hiai_petz_book}. To determine the absolute continuous part of $u(x,t)$ we use the Stieltjes inversion formula~\eqref{eq:stieltjes}. Considering the function
$$
D(x,t) = (t-m_1-x(m_1+m_2))^2 + 4 x t m_2
$$
as a quadratic function in $x$, we easily convince ourselves that $D(x) \leq 0$ if and only if $x_-(t)\leq x \leq x_+(t)$, where
$x_-(t)$ and $x_+(t)$, the zeroes of $D(x,t)$,  are given by~\eqref{eq:x_+-}. The function $G_t(x)$ is an analytic function on the complex plane with a cut at $[x_-(t), x_+(t)]$ and eventual poles at $x=0$ and $x=-1$.  The absolute continuous component of $u(x,t)$ vanishes on  $\R\backslash [x_-(t), x_+(t)]$ since $G_t(x)$ is real there.   On the interval $(x_-(t), x_+(t))$  we have $D(x,t) < 0$, and the Stieltjes inversion formula~\eqref{eq:stieltjes} yields
$$
u(x,t)
=
-\frac 1{\pi} \frac{\sqrt {-D(x,t)}}{2x(1+x)} 
=
\frac{(m_1+m_2) \sqrt{(x_+(t)-x)(x-x_-(t))}}{2\pi |x| (1+x)}, 
$$
which completes the derivation of~\eqref{eq:u_x_t_real}.

At least theoretically, the above  method could be generalized to more general initial conditions of the form $u(x,0) = \sum_{j=1}^k m_j \delta(x-x_j)$, but this would require a non-trivial analysis of Riemann surfaces of certain algebraic functions. Zero distribution of repeated derivatives of $(p(x))^n$ for a general polynomial $p$ has been studied by the steepest descent method in~\cite{shapiro}, which leads to topological difficulties when choosing the saddle point contour. It seems that no explicit formula for the simplest special case $p(x) = x^2 - 1$ (which shall be treated in the next Example~\ref{eq:x^2-1}) is stated in the preprint version of~\cite{shapiro} available to us.

\begin{example}\label{eq:x^2-1}
Let us compute the asymptotic zero distribution of the $[tn]$-th derivative of the polynomial $(x^2-1)^n$, where $0<t<2$.  This corresponds to the special  case $m_1=m_2=1$ of the above setting, after passing from the interval $[-1,0]$ to $[-1,1]$ by an affine transformation. The initial condition at time $t=0$ is
\begin{equation}\label{eq:delta_initial}
u(x,0) = \delta(x+1) + \delta(x-1).
\end{equation}
For general $t\in (0, 2)$, the affinely transformed solution~\eqref{eq:u_x_t_real} takes the form
\begin{equation}\label{eq:delta_sol}
u(x,t) =\frac{\sqrt{1 - (t-1)^2 - x^2}}{\pi\cdot(1-x^2)} \ind_{\{x^2< t(2-t)\}} + (1-t)\ind_{\{t\leq 1\}} \cdot (\delta(x-1) + \delta(x+1)).
\end{equation}
The case $t=1$ corresponds to the Legendre polynomials $P_n(x) = \frac {1}{2^n n!}\frac{\dd^n}{\dd x^n} (x^2-1)^n$  whose zeroes are distributed according to the arcsine density
$$
u(x,1) = \frac 1 {\pi\sqrt{1-x^2}} \ind_{\{|x|<1\}}
$$
by the estimates going back to Bruns, Markow and Stieltjes, see~\cite{szegoe}, or by a general theorem of Erd\H{os} and Turan on the distribution of zeroes of orthogonal polynomials; see~\cite{erdoes_turan}, \cite{ullman}, \cite[\S~1.2-1.3]{van_assche_book}.
A plot of the solution $u(x,t)$ is shown on Figure~\ref{pic:arcsine}. Modulo delta functions at $1$ and $-1$,  this solution has a time symmetry around the point $t=1$, namely we have
\begin{equation}\label{eq:time_symmetry_arcsine}
u(x,1+s) = u(x,1-s) - s \cdot (\delta(x-1) + \delta(x+1)),
\end{equation}
for $0<s<1$. At time $t\approx 0$, the evolution starts with an approximately Wigner distribution on a small interval around $1/2$ (together with atoms at $\pm 1$). At time $t=1$, the atoms disappear and the solution becomes the arcsine density. After that, it evolves back to an approximately Wigner distribution on a small interval around $1/2$, this time without atoms. At time $t=2$ the solution vanishes.
\end{example}

\begin{figure}[!tbp]
\includegraphics[width=0.49\textwidth]{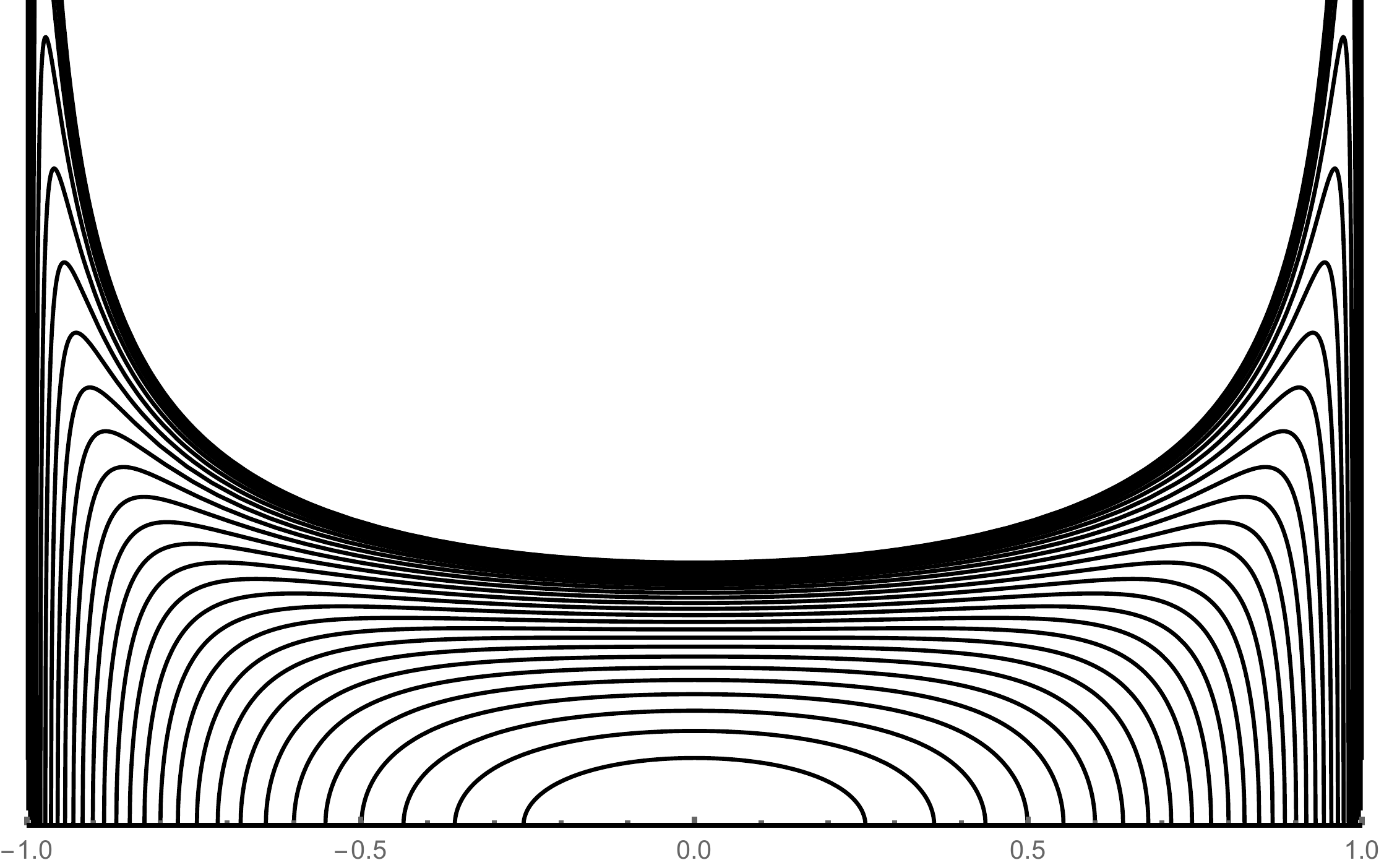}
\caption{
The evolution of the arcsine density~\eqref{eq:arcsine_initial}, \eqref{eq:arcsine_sol} on the interval $[-1,1]$. If we agree to ignore the atoms at $-1$ and $1$, the same figure can serve as an illustration of the evolution given by~\eqref{eq:delta_initial}, \eqref{eq:delta_sol} taking into account the time symmetry stated in~\eqref{eq:time_symmetry_arcsine}.
}
\label{pic:arcsine}
\end{figure}


\begin{figure}[!tbp]
\includegraphics[width=0.24\textwidth]{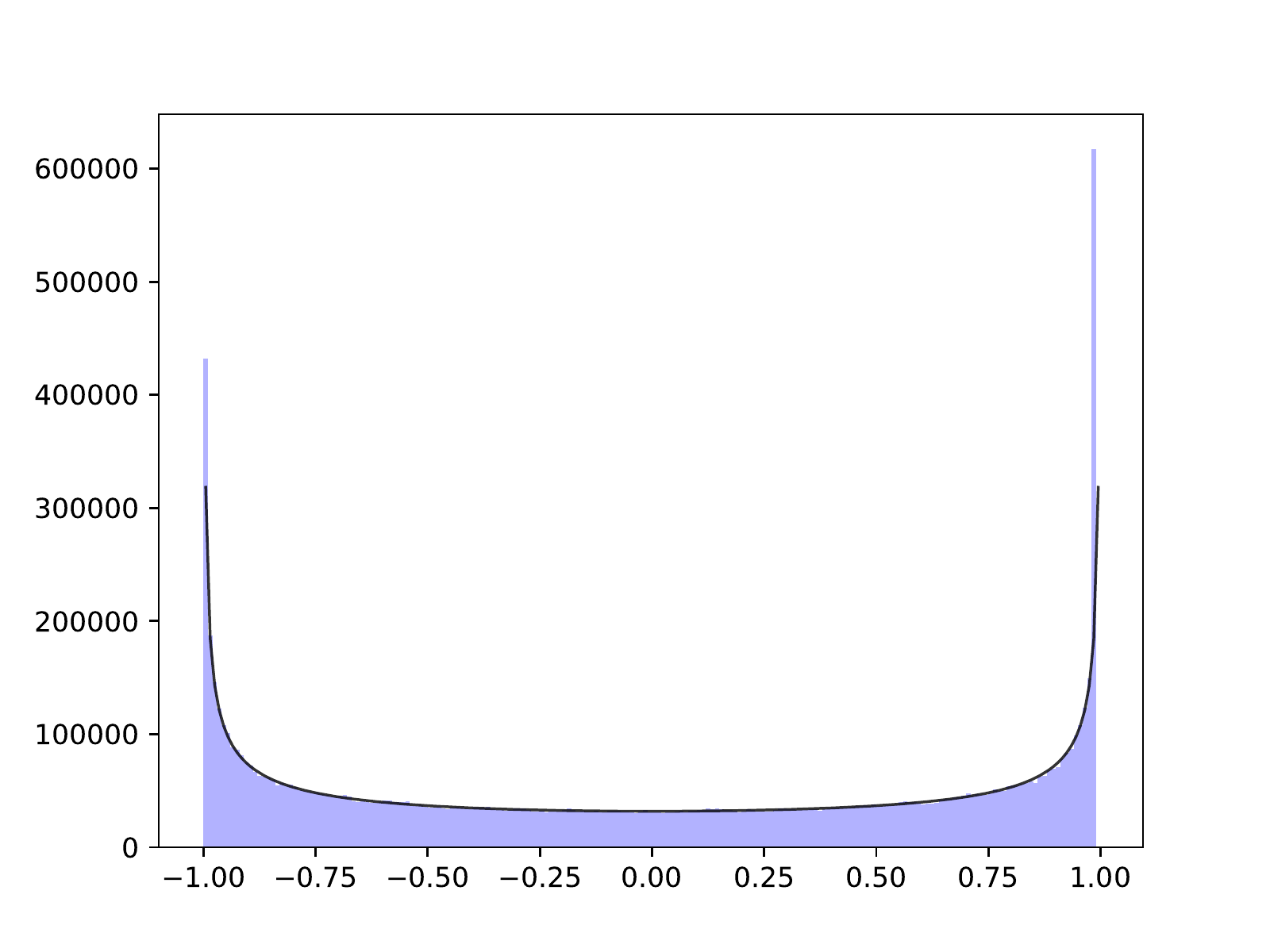}
\includegraphics[width=0.24\textwidth]{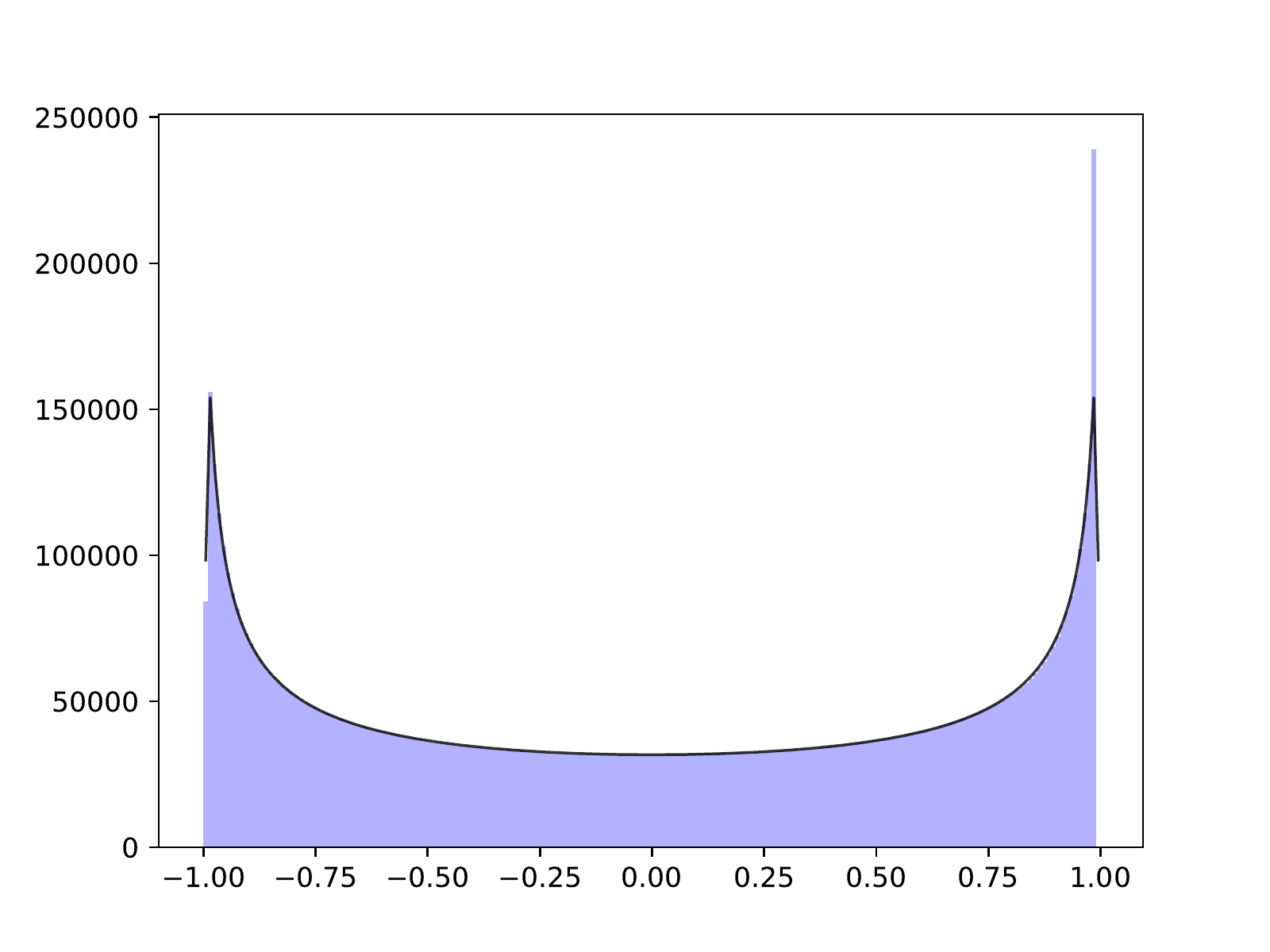}
\includegraphics[width=0.24\textwidth]{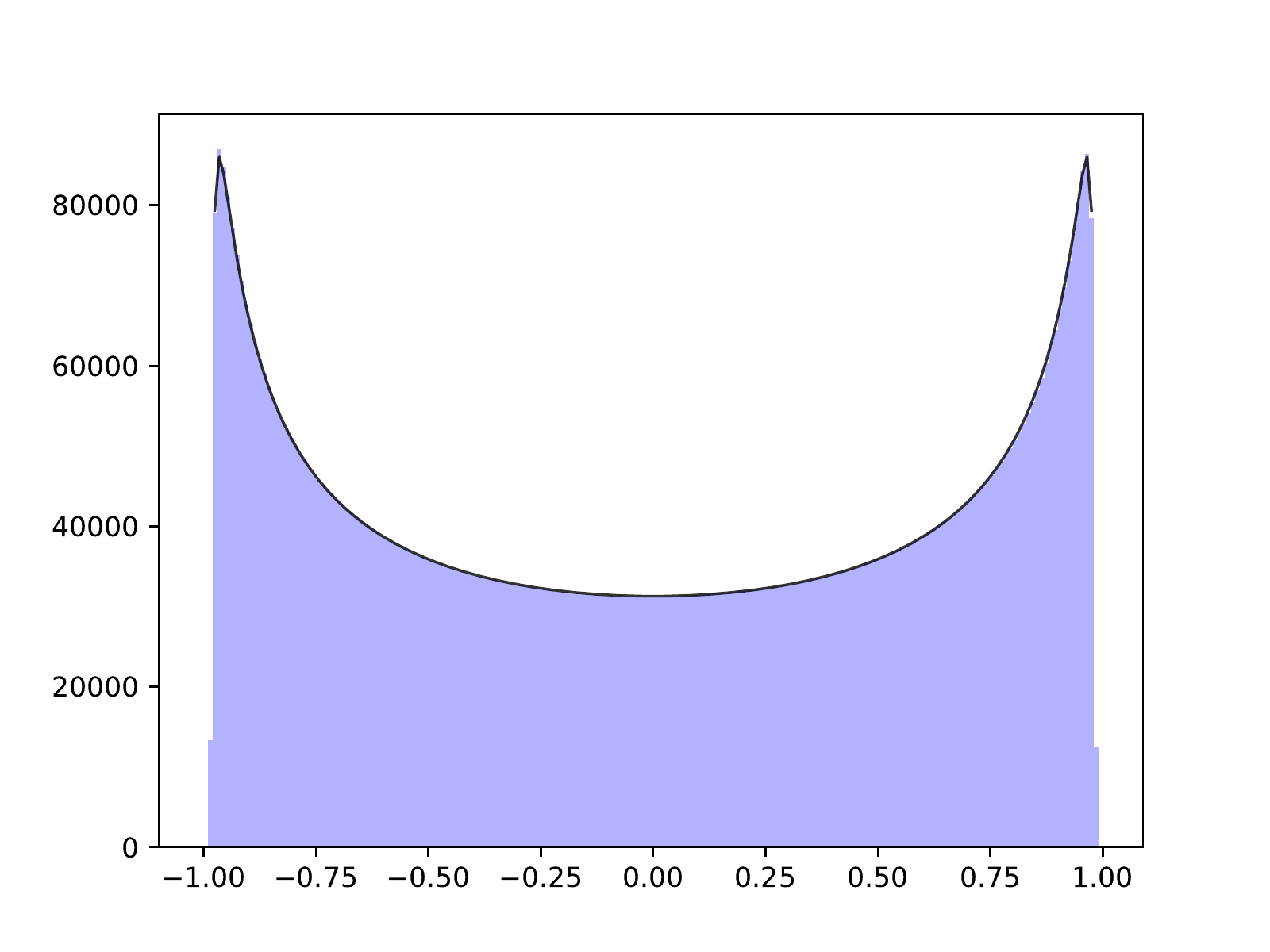}
\includegraphics[width=0.24\textwidth]{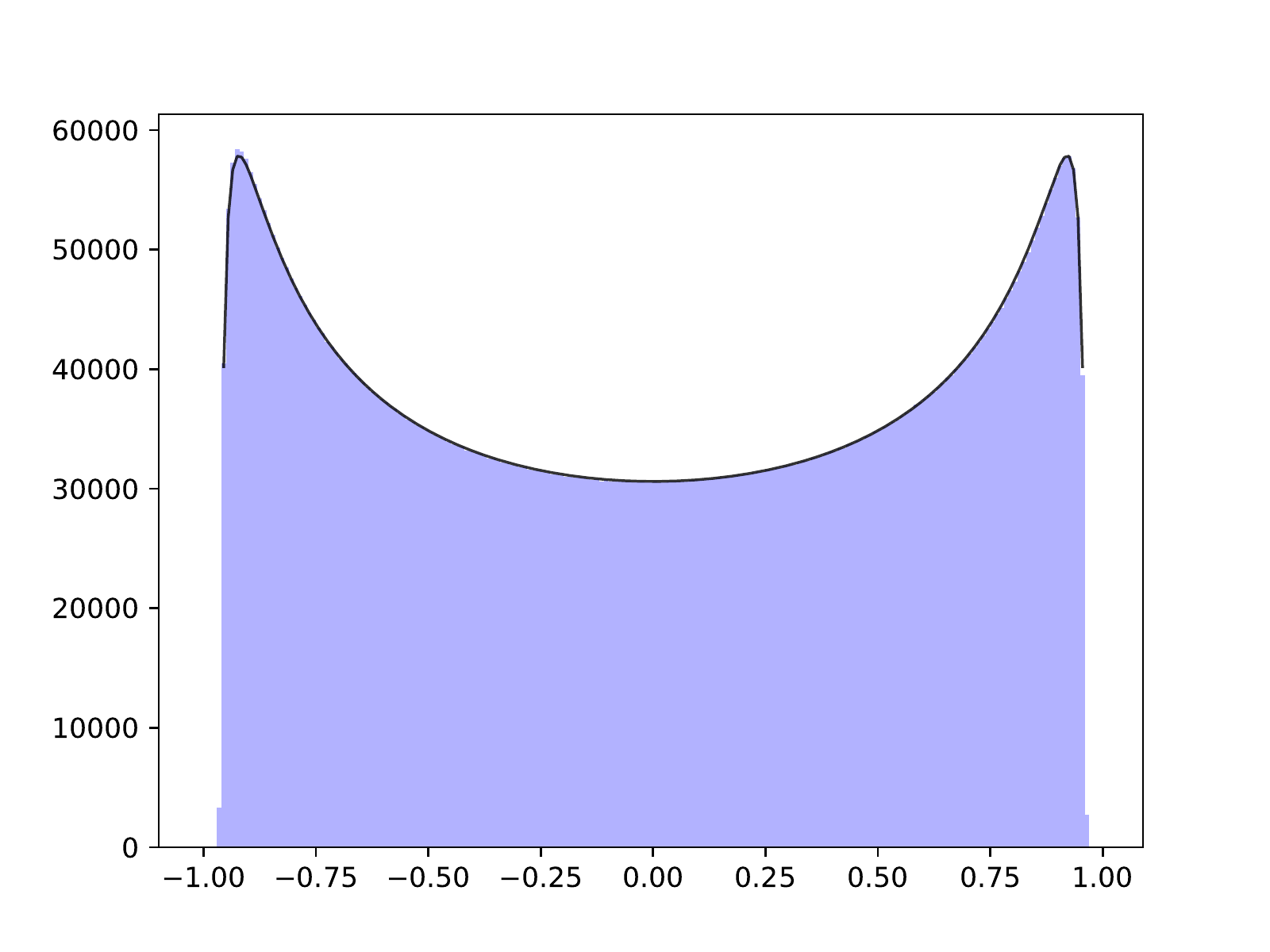}
\includegraphics[width=0.24\textwidth]{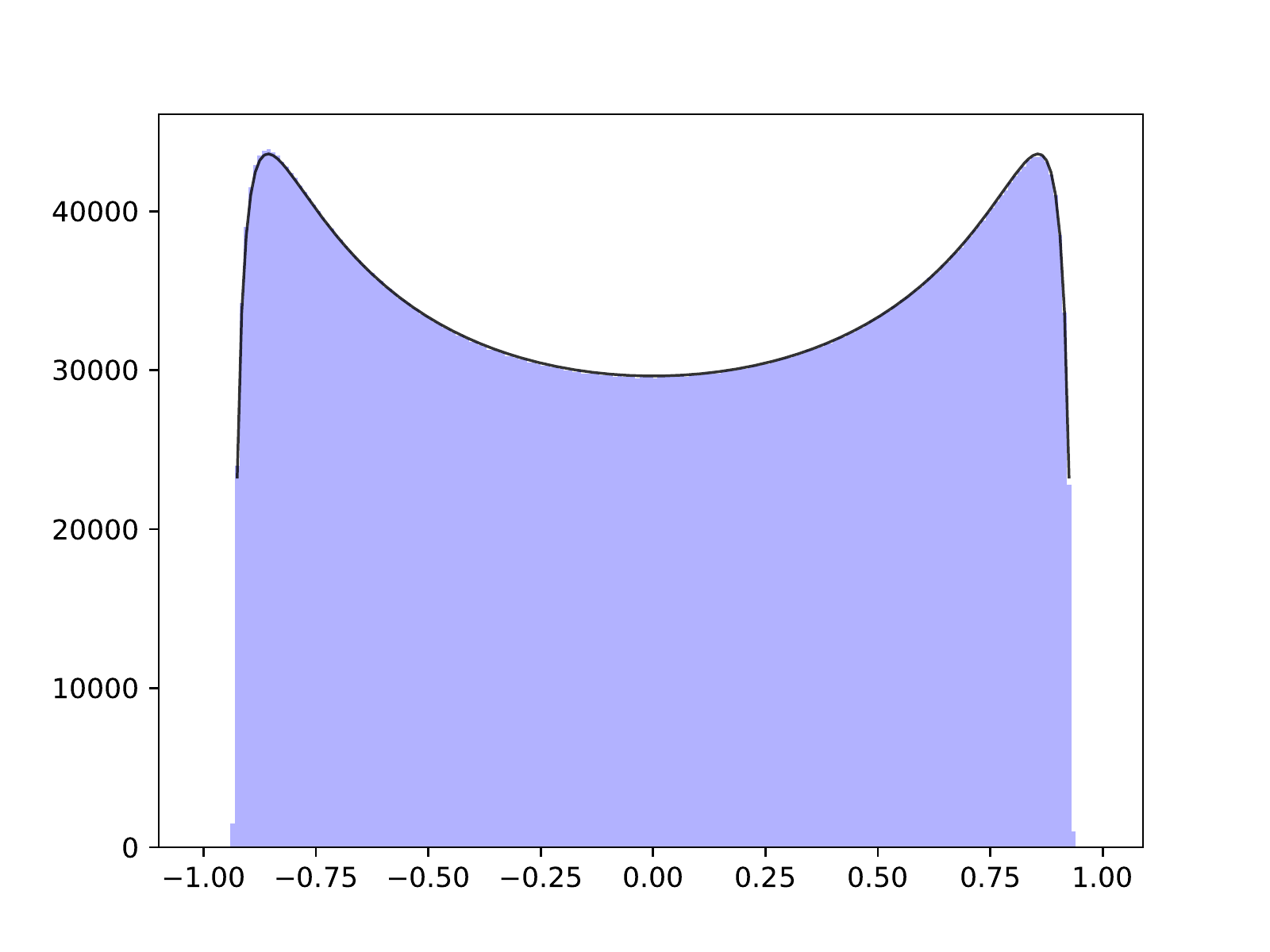}
\includegraphics[width=0.24\textwidth]{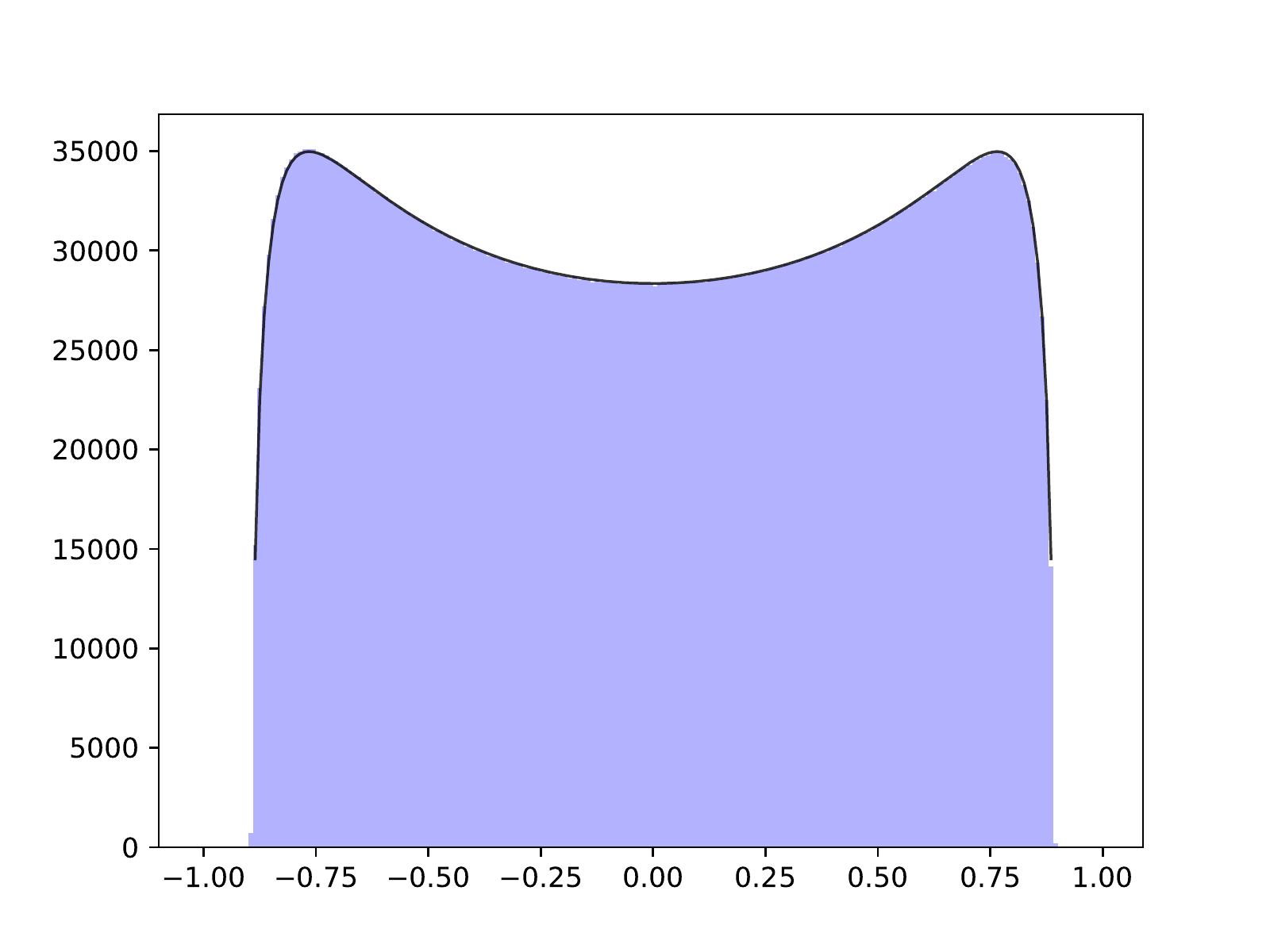}
\includegraphics[width=0.24\textwidth]{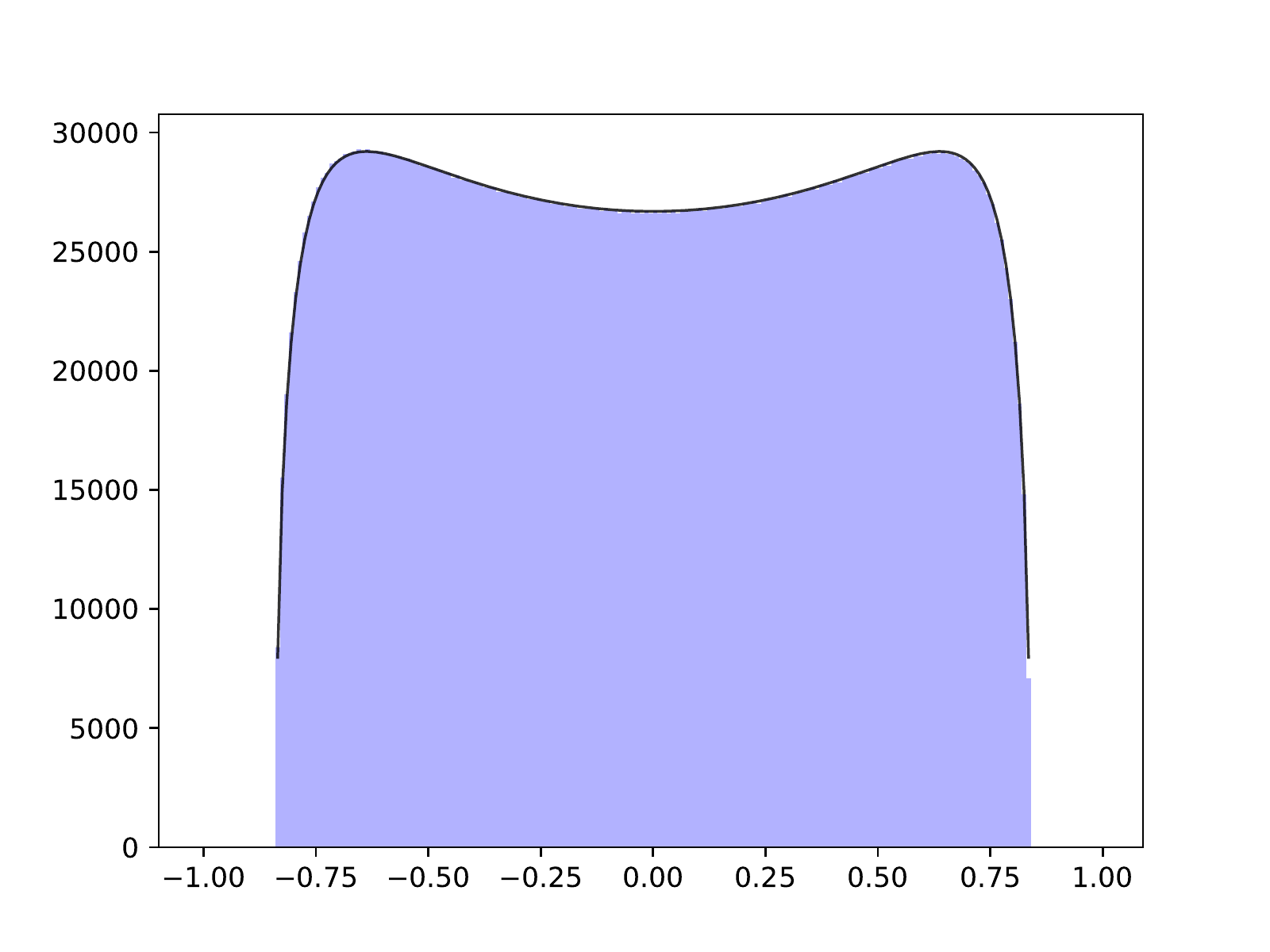}
\includegraphics[width=0.24\textwidth]{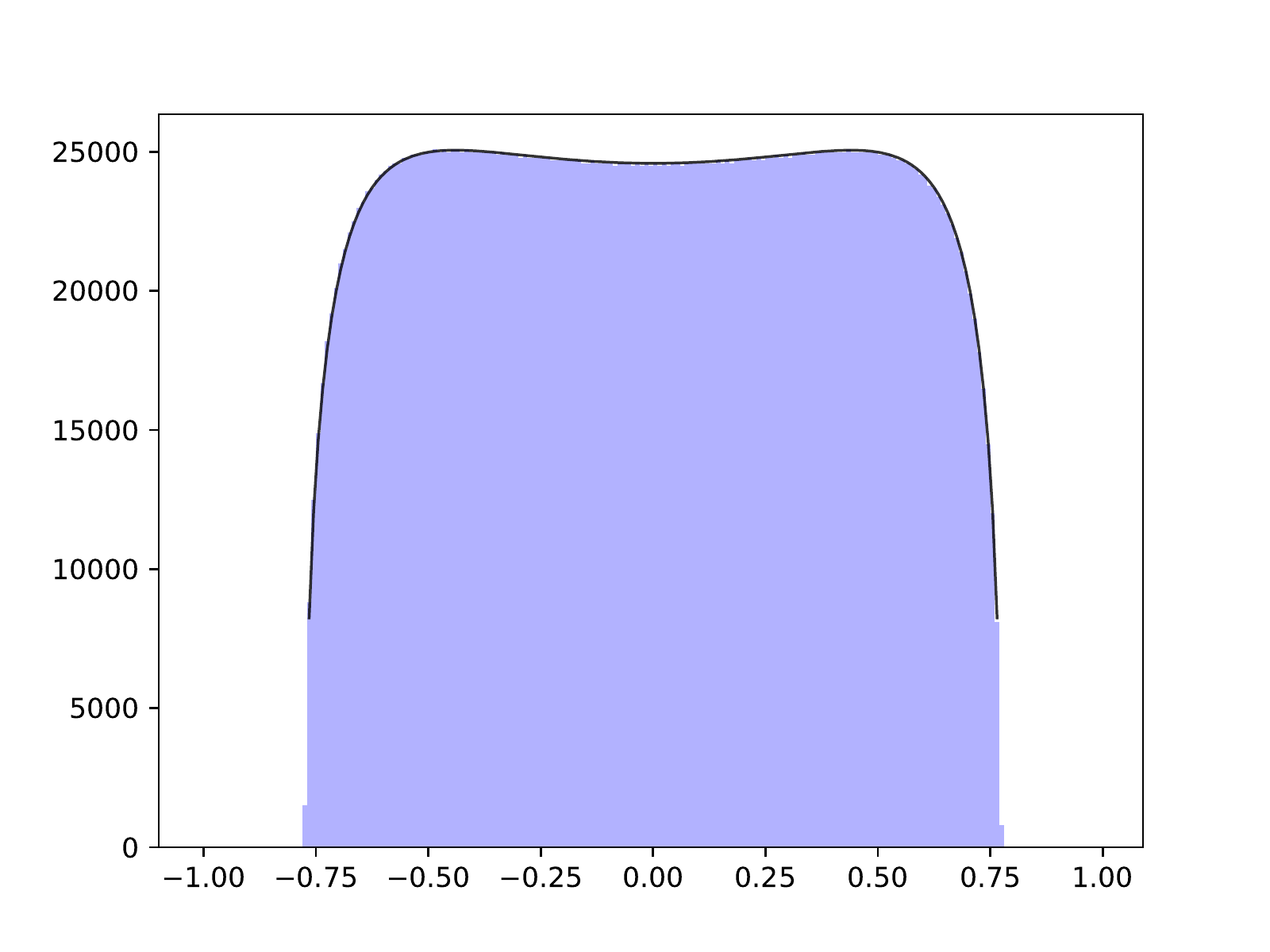}
\includegraphics[width=0.24\textwidth]{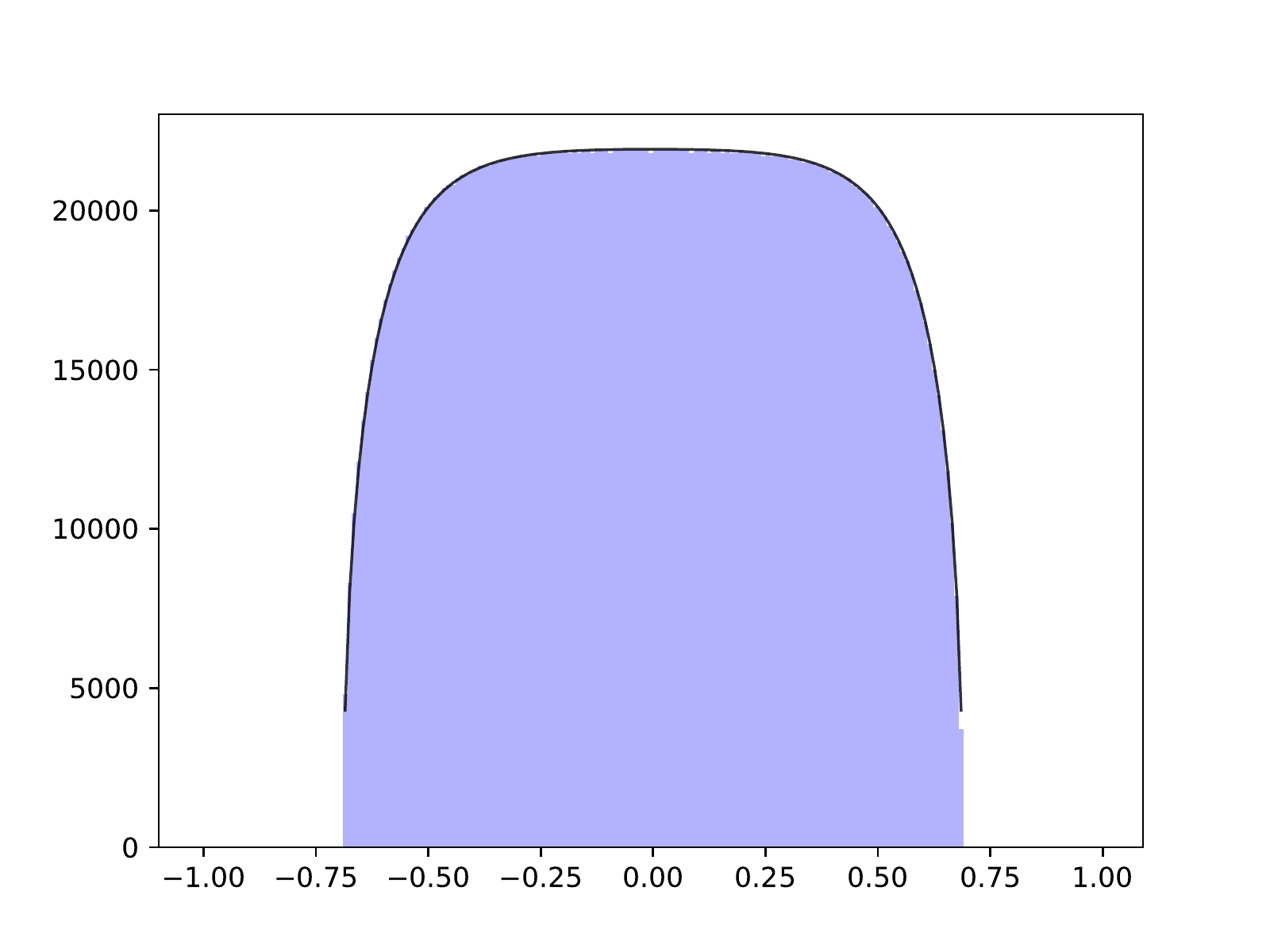}
\includegraphics[width=0.24\textwidth]{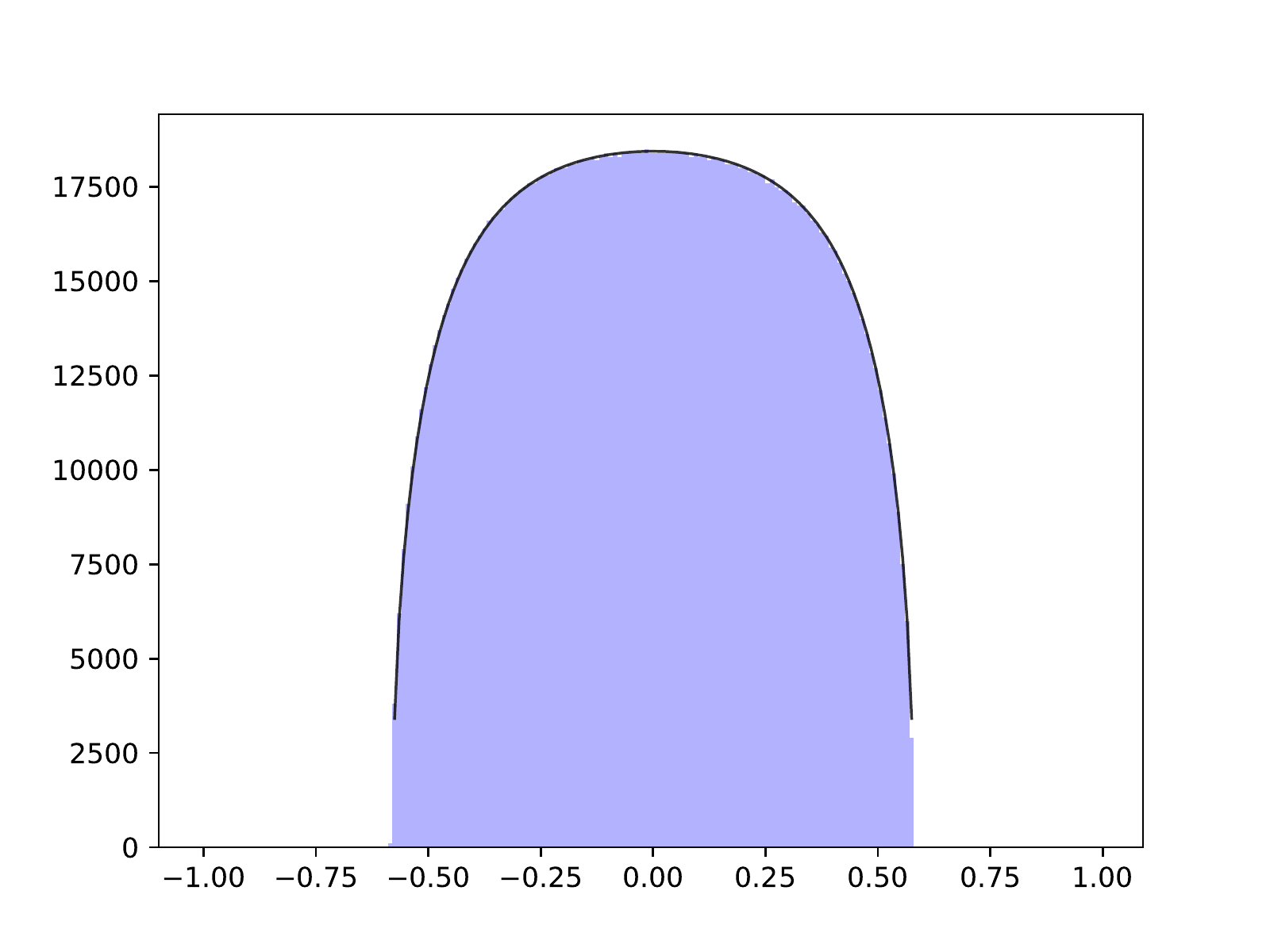}
\includegraphics[width=0.24\textwidth]{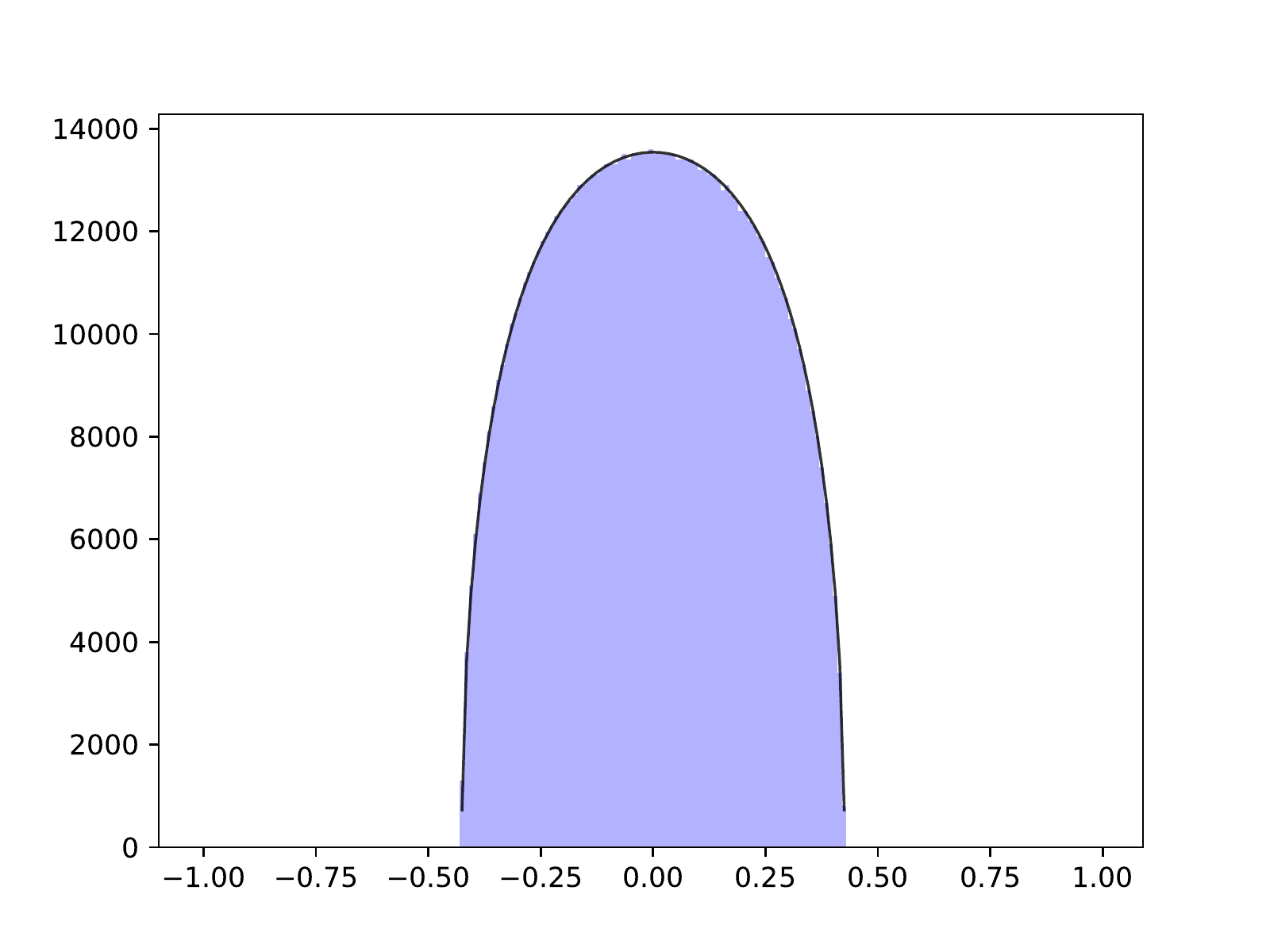}
\includegraphics[width=0.24\textwidth]{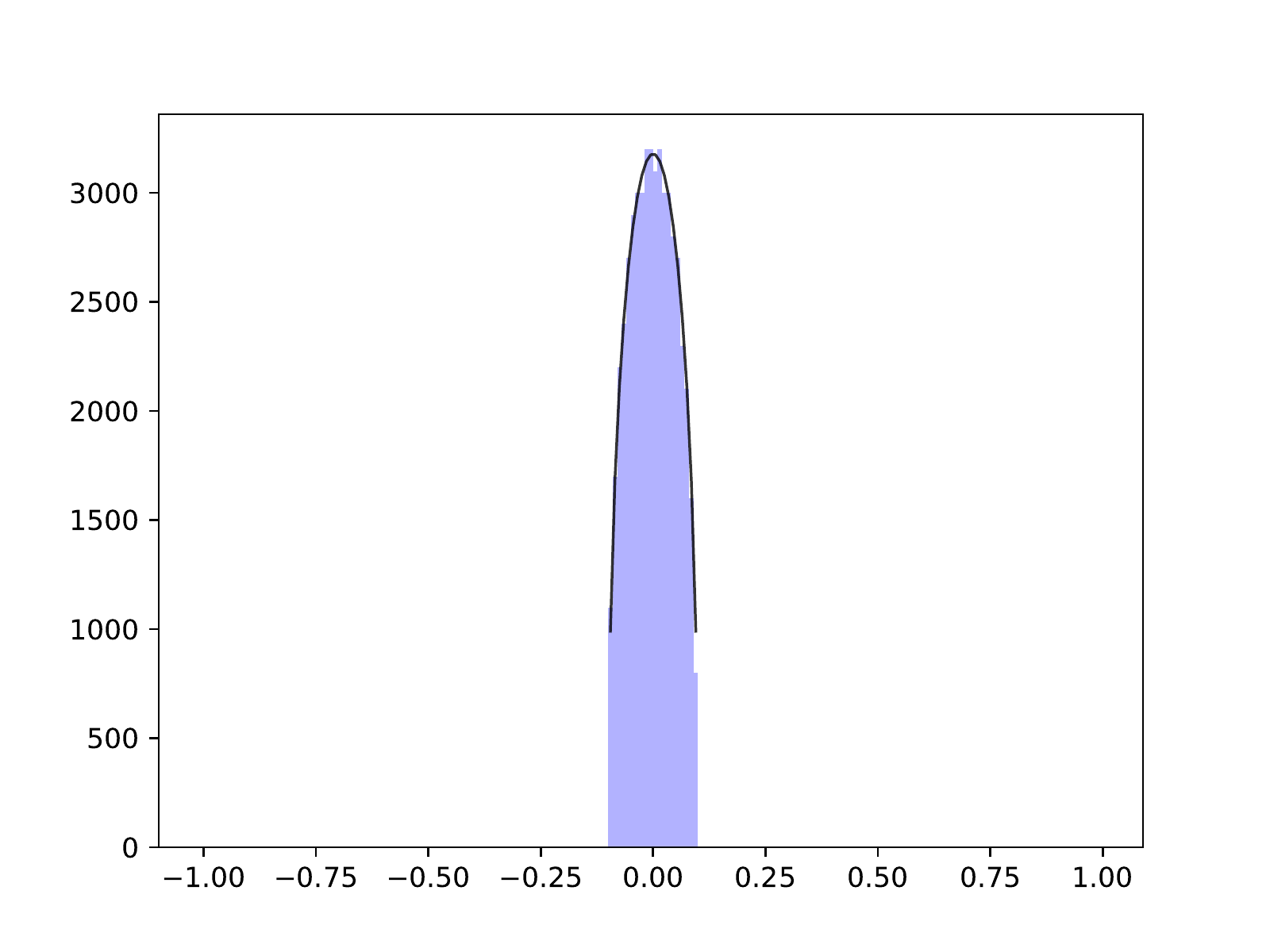}
\caption{Histograms for the evolution of $n=10^5$ i.i.d.\ roots having the arcsine density~\eqref{eq:arcsine_initial}.
The black curve is the solution~\eqref{eq:arcsine_sol}. The orders of the derivatives are $50+900k$  with  $k\in \{0,\ldots,11\}$.}
\label{pic:arcsine_histogram}
\end{figure}

\begin{example}
Let now the zeroes at time $s=0$ have the following arcsine density:
\begin{equation}\label{eq:arcsine_initial}
\rho(x,0) = \frac 1 {\pi\sqrt{1-x^2}} \ind_{\{|x|<1\}}.
\end{equation}
One example to keep in mind are the Legendre polynomials.
Substituting  $t = 1+s$ with $0<s<1$ in the previous example, we obtain the following asymptotic distribution of zeroes of the $[sn]$-th derivative:
\begin{equation}\label{eq:arcsine_sol}
\rho(x,s)
=
\frac {\sqrt{1-x^2-s^2}} {\pi\cdot (1-x^2)} \ind_{\{x^2<1-s^2\}}.
\end{equation}
It is interesting to compare this formula to the findings of Steinerberger~\cite{steinerberger_real}. Recall that he derived the PDE~\eqref{eq:PDE_real} describing the evolution of real roots under repeated differentiation and found three families of explicit solutions to this PDE. One of these solutions, called the stationary arcsine solution, is the arcsine density as in~\eqref{eq:arcsine_initial} without any dependence on $t$; see~\cite{steinerberger_real} and~\cite{coifman_steinerberger}.  Clearly, this solution is different from~\eqref{eq:arcsine_sol}. Numerical simulations confirm that the asymptotic distribution of roots of repeated derivatives is given by~\eqref{eq:arcsine_sol}; see Figure~\ref{pic:arcsine_histogram}.  As Steinerberger mentions, his stationary arcsine solution is a solution on $(-1,1)$, not on $\R$, which may be the reason why in this case the evolution of roots is described by a different formula. Let us finally mention that the other two special solutions of the PDE~\eqref{eq:PDE_real} discovered in~\cite{steinerberger_real}, namely the Wigner semicircle solution and the Marchenko-Pastur solution, can be recovered by the recipe of Section~\ref{subsec:recipe_real}. The Cauchy-Stieltjes transforms of these distributions are well known and given in Examples 3.1.1 and 3.3.5 of~\cite{hiai_petz_book}. We omit the straightforward but lengthy details.
\end{example}

\subsection{Connection to free probability}
Quite recently, Steinerberger~\cite{steinerberger_free} proposed a surprising interpretation of the density of zeroes of repeated derivatives in terms of free probability~\cite{voiculescu_nica_dykema_book,nica_speicher_book,hiai_petz_book}. Using his PDE~\eqref{eq:PDE_real} as a starting point, he has shown that the density of roots at time $t$ is up to a rescaling the $\frac {1}{1-t}$-th free convolution power of the initial distribution $\mu_0$, namely
\begin{equation}\label{eq:free_conv}
\mu_0^{\boxplus \frac{1}{1-t}} = u((1-t) x, t) \dd x, \qquad 0\leq t<1.
\end{equation}

For example, the densities~\eqref{eq:u_x_t_real} and~\eqref{eq:delta_sol} coincide (up to linear transformations) with free binomial distributions defined as  free convolution powers of the Bernoulli distribution; see~\cite[Example~3.6.7]{voiculescu_nica_dykema_book} for the general case,  \cite[Example~3.4.5]{voiculescu_nica_dykema_book}, \cite[Example~3.2.2]{hiai_petz_book} and~\cite[Examples~12.8,14.15,4.5]{nica_speicher_book} for some special cases,  and~\cite{szpojankowski_wesolowski}, \cite{saitoh_yoshida} for further references.

Let us re-derive~\eqref{eq:free_conv} using our approach (which is rigorous).  Referring  to~\cite{voiculescu_nica_dykema_book,nica_speicher_book,hiai_petz_book} for the necessary background on free probability theory, we only recall here the definition of free convolution powers.
If $\mu_0$ is a compactly supported probability measure on $\R$ and $G_0$ is its Cauchy-Stieltjes transform, then the $R$-transform of $\mu_0$ can be defined by the equation
$$
1+ R_0(G_0(z)) = zG_0(z);
$$
see~\cite[Theorem~3.2.1]{hiai_petz_book} or~\cite[Theorem~12.7]{nica_speicher_book} with $\mathcal R(z) = R(z)/z$. It is known that $R_0(z)$ is an analytic function in a sufficiently small complex neighborhood of $0$. For every $s\geq 1$, the $s$-th free convolution power of $\mu_0$ is a probability measure  $\mu_0^{\boxplus s}$ whose $R$-transform equals $sR_0(z)$; see~\cite[Corollary~14.13]{nica_speicher_book} for its existence and interpretation in terms of compressing by free projections.
\begin{theorem}\label{theo:real_free}
Consider a sequence of monic deterministic polynomials $(Q_n)_{n\in\N}$ whose zeroes belong to some bounded interval $J\subset \R$ and satisfy
\begin{equation*}
\frac 1n \sum_{z\in \R: Q_n(z) = 0} \delta_z \toweak \mu_0
\end{equation*}
for some probability measure $\mu_0$ on $J$. Then, for every $0\leq t <1$ we have
\begin{equation}\label{eq:real_free_conv}
\frac 1 {(1-t) n} \sum_{z\in \R: Q_n^{([tn])}(z) = 0} \delta_{\frac{z}{1-t}} \toweak \mu_0^{\boxplus \frac{1}{1-t}},
\end{equation}
where the right-hand side is a free convolution power of $\mu_0$. 
\end{theorem}
\begin{proof}
Fix some $0\leq t <1$. Applying a translation, if necessary, we may assume that $J\subset (-\infty,0)$. We know from Theorem~\ref{theo:real_recipe} that the left-hand side of~\eqref{eq:real_free_conv} converges weakly to the probability measure $\mu_t^*$ given by $\mu_t^*(A) = (1-t)^{-1} \mu_t((1-t)A)$, for all Borel sets $A\subset \R$.  The $R$-transforms $R_0$ and $R$ of the probability measures $\mu_0$ and $\mu_t^*$ satisfy the relations
\begin{equation}\label{eq:R_transforms}
1+ R_0(G_0(z)) = zG_0(z), \qquad 1 + R(G_t(z)) = \frac{z}{1-t} G_t(z);
\end{equation}
see~\cite[Theorem~3.2.1]{hiai_petz_book} or~\cite[Theorem~12.7]{nica_speicher_book} with $\mathcal R(z) = R(z)/z$.
For the second relation in~\eqref{eq:R_transforms}, we used that the Cauchy-Stieltjes transform of the probability measure $\mu_t^*$ is given by $G(y) = G_t((1-t) y)$, where $G_t$ is the Cauchy-Stieltjes transform of $\mu_t$.
Recall from~\eqref{eq:w_0_inverse_real}, \eqref{eq:w_t_inverse_real}, \eqref{eq:w_t_w_0} the identities
$$
y = w_0(y) G_0 (w_0(y)), \qquad y = w_t(y) G_t(w_t(y)), \qquad w_t(y) = w_0(y+t)\frac{y}{y+t}.
$$
Taking $z = w_0(y)$ and $z = w_t(y)$ in~\eqref{eq:R_transforms} we can write
\begin{equation}\label{eq:R_G}
R_0(G_0(w_0(y))) = y-1, \qquad    R(G_t(w_t(y))) = \frac{y}{1-t} - 1.
\end{equation}
It follows that for all $0\leq y < 1-t$,
$$
\frac{y}{1-t} - 1 = R(G_t(w_t(y))) = R \left(\frac{y}{w_t(y)}\right) = R \left(\frac{y+t}{w_0(y+t)}\right) = R\left(G_0(w_0(y+t))\right).
$$
On the other hand, it follows from the first equality in~\eqref{eq:R_G} that for all $0\leq y < 1-t$,
$$
R_0(G_0(w_0(y+t))) = y + t - 1 =  y - (1-t).
$$
By comparing these identities, it follows that $R(\lambda_j) = \frac 1{1-t} R_0(\lambda_j)$ for some sequence $\lambda_1>\lambda_2>\ldots >0$ converging to $0$.  Indeed, this follows from $\lim_{y\uparrow 1-t} G_0(w_0(y+t)) = 0$, which in turn follows from $\lim_{y\uparrow 1-t} w_0(y+t)= +\infty$, which has been proved in Section~\ref{subsec:recipe_real}. By the uniqueness principle for analytic functions, we have $R(\lambda) = \frac 1{1-t} R_0(\lambda)$ for all complex $\lambda$ with sufficiently small absolute value. By definition, this means that $\mu_t^*$ is the $\frac{1}{1-t}$-th free convolution of $\mu_0$.
\end{proof}

\section*{Acknowledgement}
We are grateful to the unknown referees for enlightening comments, in particular for suggesting an interpretation of the real case in terms of finite free probability.
ZK has been supported by the German Research Foundation under Germany's Excellence Strategy  EXC 2044 -- 390685587, Mathematics M\"unster: Dynamics - Geometry - Structure.

\bibliography{zeroes_rep_diff_bib}
\bibliographystyle{plainnat}

\end{document}